\documentclass[reqno, dvipsnames, table]{amsart}

\usepackage{amsfonts}
\usepackage{amsmath}
\usepackage{amssymb}
\usepackage{amsthm}
\usepackage{array}
\usepackage[backend=biber, maxbibnames=99, doi=false, isbn=false, url=false]{biblatex}
\usepackage{booktabs}
\usepackage{graphicx}
\usepackage{hyperref}
\usepackage{lscape}
\usepackage{longtable}
\usepackage{mathrsfs}
\usepackage{mathtools}
\usepackage{multirow}
\usepackage{rotating}
\usepackage{tikz-cd}
\usepackage{xcolor}
\usepackage{xparse}

\theoremstyle{plain}
\newtheorem{theorem}{Theorem}[section]
\newtheorem{corollary}[theorem]{Corollary}
\newtheorem{lemma}[theorem]{Lemma}

\newtheorem{conjecture}[theorem]{Conjecture}

\theoremstyle{definition}
\newtheorem{definition}[theorem]{Definition}

\theoremstyle{remark}

\newtheorem{remark}[theorem]{Remark}

\numberwithin{equation}{subsection}

\RenewDocumentCommand{\AA}{}{\mathbb{A}}
\NewDocumentCommand{\FF}{}{\mathbb{F}}
\NewDocumentCommand{\HH}{}{\mathbb{H}}
\NewDocumentCommand{\PP}{}{\mathbb{P}}
\NewDocumentCommand{\RR}{}{\mathbb{R}}
\NewDocumentCommand{\cB}{}{\mathcal{B}}
\NewDocumentCommand{\cG}{}{\mathcal{G}}
\NewDocumentCommand{\cH}{}{\mathcal{H}}

\DeclareMathOperator{\Aut}{Aut}
\DeclareMathOperator{\agr}{agr}

\DeclareMathOperator{\ddeg}{\mathbf{deg}}
\DeclareMathOperator{\Gal}{Gal}
\DeclareMathOperator{\Spec}{Spec}

\DeclareMathOperator{\Sym}{Sym}

\NewDocumentCommand{\clo}{m m}{{{}_{#1}{#2}}}
\NewDocumentCommand{\GL}{O{2}}{\operatorname{GL}_{#1}}

\NewDocumentCommand{\op}{}{\mathrm{op}}
\NewDocumentCommand{\Pone}{}{\ensuremath{\mathbb{P}^{1}}}
\NewDocumentCommand{\Q}{}{\mathbb{Q}}

\NewDocumentCommand{\Qbar}{}{\overline{\Q}}
\NewDocumentCommand{\Qcm}{}{\Q^{\mathrm{CM}}}
\NewDocumentCommand{\SL}{O{2}}{\operatorname{SL}_{#1}}

\NewDocumentCommand{\tors}{}{\mathrm{tors}}
\NewDocumentCommand{\Z}{s o}{\IfBooleanT{#1}{(}\mathbb{Z}\IfValueT{#2}{/ #2 \mathbb{Z}}\IfBooleanT{#1}{)^{\times}}}
\NewDocumentCommand{\Zhat}{s}{\widehat{\Z}\IfBooleanT{#1}{^{\times}}}

\NewDocumentCommand{\ldblbrace}{}{\{\mskip-5mu\{}
\NewDocumentCommand{\rdblbrace}{}{\}\mskip-5mu\}}

\DeclareFontFamily{U}{wncy}{}
\DeclareFontShape{U}{wncy}{m}{n}{<->wncyr10}{}
\DeclareSymbolFont{mcy}{U}{wncy}{m}{n}
\DeclareMathSymbol{\Sha}{\mathord}{mcy}{"58}
\DeclareMathSymbol{\shortminus}{\mathbin}{AMSa}{"39}

\makeatletter
\def\@@@cmidrule[#1-#2]#3#4{\global\@cmidla#1\relax
	\global\advance\@cmidla\m@ne
	\ifnum\@cmidla>0\global\let\@gtempa\@cmidrulea\else
	\global\let\@gtempa\@cmidruleb\fi
	\global\@cmidlb#2\relax
	\global\advance\@cmidlb-\@cmidla
	\global\@thisrulewidth=#3
	\@setrulekerning{#4}
	\ifnum\@lastruleclass=\z@\nobreak\vskip \aboverulesep\fi
	\ifnum0=`{\fi}\@gtempa
	\noalign{\ifnum0=`}\fi\futurenonspacelet\@tempa\@xcmidrule}
\def\@xcmidrule{%
	\ifx\@tempa\cmidrule
		\nobreak\vskip-\@thisrulewidth
		\global\@lastruleclass=\@ne
	\else \ifx\@tempa\morecmidrules
		\nobreak\vskip \cmidrulesep
		\global\@lastruleclass=\@ne\else
		\nobreak\vskip \belowrulesep
		\global\@lastruleclass=\z@
	\fi\fi
	\ifnum0=`{\fi}}
\makeatother

\hypersetup{
	colorlinks=true,
	urlcolor=NavyBlue,
	citecolor=Green,
	linkcolor=BrickRed
}

\title{Degrees of points with rational $j$-invariant on $X_{0}(n)$ and $X_{1}(n)$}
\date{}
\author{Kenji Terao}

\begin{document}
	\begin{abstract}
		We give a classification of the degrees of the points with rational $j$-invariant on the modular curves $X_{0}(n)$ and $X_{1}(n)$. The degrees which occur infinitely often are computed unconditionally, while those which occur finitely often are determined assuming a conjecture of Zywina. To achieve this, we define the notion of $\cH$-closures of subgroups of $\GL(\Zhat)$, and compute the $\cB_{0}(n)$- and $\cB_{1}(n)$-closures of images of Galois representations of elliptic curves defined over $\Q$. An application to computing the set of isolated $j$-invariants in $\Q$ is also given.
	\end{abstract}
	
	\maketitle
	
	\section{Introduction} \label{sec:introduction}
	
	Let $E/ K$ be an elliptic curve defined over a number field $K$. The celebrated Mordell-Weil theorem \cite{weil1929arithmetique} states that the set of $K$-rational points $E(K)$ on $E$ is a finitely generated abelian group. In particular, we have that
	\[
		E(K) \cong \Z^{r} \times E(K)_{\tors},
	\]
	for some $r \in \Z_{\geq 0}$ and some finite abelian group $E(K)_{\tors}$.
	
	The structure of the finite abelian group $E(K)_{\operatorname{tors}}$ has long been of interest, and has been deeply studied. For instance, a landmark result of Mazur \cite{mazur1977modular} shows that, for an elliptic curve $E / \Q$, we have
	\[
		E(\Q)_{\tors} \cong \begin{cases}
			\Z[n] & n \in \{1, \dots, 10, 12\} \\
			\Z[2] \times \Z[2n] & n \in \{1, \dots, 4\},
		\end{cases}
	\]
	and all of these possibilities occur for infinitely many $\Qbar$-isomorphism classes of elliptic curves over $\Q$.
	
	Merel's uniform boundedness theorem \cite{merel1996bornes} states that, for any number field $K$, there exists a constant $B(d)$, depending only on the degree $d$ of $K$, such that for any elliptic curve $E / K$,
	\[
		\left| E(K)_{\tors} \right| \leq B(d).
	\]
	In particular, there are only finitely many possible structures for the torsion subgroup $E(K)_{\tors}$ as $E$ varies through all elliptic curves defined over a number field $K$ of a fixed degree $d$.
	
	Classifying these structures then becomes a natural problem, which has been solved for small values of $d$. In addition to Mazur's result, which covers the case $d = 1$, the case of $d = 2$ was solved by Kamienny \cite{kamienny1992torsion}, Kenku and Momose \cite{kenku1988torsion}, while the case of $d = 3$ was completed by Derickx, Etropolski, van Hoeij, Morrow and Zureick-Brown in \cite{derickx2021sporadic}. Further partial results have also been obtained, such as the classification of torsion subgroups which occur infinitely often for elliptic curves over number fields of degrees $d = 4, 5$ and $6$, the former by Jeon, Kim and Schweizer in \cite{jeon2004torsion}, and the latter two by Derickx and Sutherland in \cite{derickx2017torsion}.
	
	One can also impose restrictions on the elliptic curve $E$, such as requiring the elliptic curve $E / K$ to be the base change of an elliptic curve $E / L$ defined over a smaller, usually fixed, number field $L$. This amounts to studying the growth of the torsion subgroups $E(K)_{\tors}$ as $K$ varies through all extensions of $L$. Much literature has been expounded on this subject, and we only briefly mention a few results. In particular, taking $L = \Q$, the torsion subgroups $E(K)_{\tors}$ have been classified for all number fields $K$ of degree $d = 2$ and 3 by Najman in \cite{najman2016torsion}, of degree $d = 4$ by Chou \cite{chou2016torsion}, Gonz\'alez-Jim\'enez and Najman \cite{gonzalez2020growth}, of degree $d = 5$ by Gonz\'alez-Jim\'enez in \cite{gonzalez2017complete}, of degree $d = 6$ by Gu\v{z}vi\'c in \cite{guzvic2021torsion}, of degree $d = p \geq 7$ by Gonz\'alez-Jim\'enez and Najman in \cite{gonzalez2020growth} and of degree $d = pq$ by Gu\v{z}vi\'c in \cite{guzvic2023torsion}, where $p$ and $q$ are prime. Gonz\'alez-Jim\'enez and Najman \cite{gonzalez2020growth} also determine the set of possible primes dividing the order $|E(K)_{\tors}|$ as $K$ varies through all number fields of degree $d$, for $d \leq 3342296$.
	
	The study of the torsion subgroup $E(K)_{\tors}$ is intimately linked to the study of the modular curves $X_{1}(n)$, which parametrize pairs $(E, P)$, where $E$ is an elliptic curve and $P \in E$ a point of exact order $n$. By \cite[Lemma 2.1]{bourdon2019level}, the degree of the corresponding point on $X_{1}(n)$ is closely related to the degree of the number field over which both $E$ and $P$ are defined. This allows one to reinterpret in large part the classifications of rational, quadratic and cubic torsion on elliptic curves as classifications of the rational, quadratic and cubic points on the modular curves $X_{1}(n)$ respectively. Restricting to elliptic curves $E$ defined over $\Q$ also has a natural interpretation in terms of the modular curves $X_{1}(n)$: namely, instead of classifying all degree $d$ points on the modular curves $X_{1}(n)$, one restricts to those degree $d$ points with rational $j$-invariant.
	
	\subsection{Main results} \label{sec:introduction:main_results}
	
	It is thus in this context that we undertake the main objective of this paper: a classification of the possible degrees of the points of $X_{1}(n)$ with rational $j$-invariant. In addition, we provide a similar classification for the modular curves $X_{0}(n)$, as many of the techniques developed for $X_{1}(n)$ apply naturally to this setting as well. As was the case for the classification of torsion subgroups of elliptic curves, it is natural to split this classification into two parts. The first is the classification of the degrees which occur for infinitely many points on the modular curve $X_{0}(n)$ or $X_{1}(n)$. This case is somewhat easier, and we obtain unconditional results. For the modular curves $X_{0}(n)$, we obtain the following.
	
	\begin{theorem} \label{thm:introduction:infinite_x0_degrees}
		Let $n \geq 1$. Define the sets
		\begin{align*}
			D_{0}^{\infty}(1) & = \{1\}, & D_{0}^{\infty}(6) & = \{1, 2\}, & D_{0}^{\infty}(12) & = \{1, 3\}, \\
			D_{0}^{\infty}(2) & = \{1, 2\}, & D_{0}^{\infty}(7) & = \{1, 2, 6, 7\}, & D_{0}^{\infty}(13) & = \{1, 13\}, \\
			D_{0}^{\infty}(3) & = \{1, 2, 3\}, & D_{0}^{\infty}(8) & = \{1\}, & D_{0}^{\infty}(16) & = \{1\}, \\
			D_{0}^{\infty}(4) & = \{1, 3\}, & D_{0}^{\infty}(9) & = \{1, 2\}, & D_{0}^{\infty}(18) & = \{1, 2, 4\}, \\
			D_{0}^{\infty}(5) & = \{1, 2, 4, 5\}, & D_{0}^{\infty}(10) & = \{1, 2, 5, 10\}, & D_{0}^{\infty}(25) & = \{1, 4\},
		\end{align*}
		and $D_{0}^{\infty}(m) = \emptyset$ otherwise. Then, we have that
		\begin{align*}
			& \{d : \exists^{\infty} \, x \in X_{0}(n) \text{ s.t. } \deg(x) = d, j(x) \in \Q\} \\
			& \quad = \bigcup_{m \mid n} \{d \cdot \deg(X_{0}(n) \to X_{0}(m)) : d \in D_{0}^{\infty}(m)\}.
		\end{align*}
	\end{theorem}
	
	The degree of the morphism $X_{0}(n) \to X_{0}(m)$ is easy to compute, and can be found for instance in \cite[Section 3.9]{diamond2005first}. The statement for the modular curves $X_{1}(n)$ is similar, and is as follows.
	
	\begin{theorem} \label{thm:introduction:infinite_x1_degrees}
		Let $n \geq 1$. Define the sets
		\begin{align*}
			D_{1}^{\infty}(1) & = \{1\}, & D_{1}^{\infty}(9) & = \{1, 3, 6\}, \\
			D_{1}^{\infty}(2) & = \{1, 2\}, & D_{1}^{\infty}(10) & = \{1, 2, 4, 5, 10, 20\}, \\
			D_{1}^{\infty}(3) & = \{1, 2, 3\}, & D_{1}^{\infty}(12) & = \{1, 2, 3, 6\}, \\
			D_{1}^{\infty}(4) & = \{1, 3\}, & D_{1}^{\infty}(13) & = \{2, 3, 6, 26, 39, 78\}, \\
			D_{1}^{\infty}(5) & = \{1, 2, 4, 5, 8, 10\}, & D_{1}^{\infty}(16) & = \{2\}, \\
			D_{1}^{\infty}(6) & = \{1, 2\}, & D_{1}^{\infty}(18) & = \{6, 12\}, \\
			D_{1}^{\infty}(7) & = \{1, 3, 6, 7, 18, 21\}, & D_{1}^{\infty}(25) & = \{5, 10, 20, 40\}, \\
			D_{1}^{\infty}(8) & = \{1, 2\},
		\end{align*}
		and $D_{1}^{\infty}(m) = \emptyset$ otherwise. Then, we have that
		\begin{align*}
			& \{d : \exists^{\infty} \, x \in X_{1}(n) \text{ s.t. } \deg(x) = d, j(x) \in \Q\} \\
			& \quad = \bigcup_{m \mid n} \{d \cdot \deg(X_{1}(n) \to X_{1}(m)) : d \in D_{1}^{\infty}(m)\}.
		\end{align*}
	\end{theorem}
	
	As for $X_{0}(n)$, the degree of the morphism $X_{1}(n) \to X_{1}(m)$ has a simple expression, given for instance in \cite[Section 3.9]{diamond2005first}.
	
	The second part is the classification of the degrees which occur for only finitely many points on the modular curve $X_{0}(n)$ or $X_{1}(n)$. This case is much harder, and we only obtain results conditionally on a wide-ranging conjecture of Zywina, which encompasses both Serre's uniformity conjecture as well as a conjecture on the non-existence of non-cuspidal, non-CM rational points on certain high-genus modular curves. We also restrict our attention to non-cuspidal, non-CM points, setting
	\[
		\Qcm = \left\{\begin{array}{c}0, 1728, -3375, 8000, -32768, 54000, 287496, -884736, -12288000, \\ 16581375, -884736000, -147197952000, -262537412640768000\end{array}\right\}
	\]
	to be the set of CM $j$-invariants in $\Q$. For the modular curves $X_{0}(n)$, this yields the following result.
	
	\begin{theorem} \label{thm:introduction:finite_x0_degrees}
		Suppose that Conjecture \ref{conj:finite_degrees:sl_intersections} holds. Let $n \geq 1$. Define the sets
		\begin{align*}
			D_{0}(7) & = D_{0}^{\infty}(7) \cup \{3\}, & D_{0}(17) & = \{1, 17\}, \\
			D_{0}(11) & = \{1, 11\}, & D_{0}(21) & = \{1, 3, 7, 21\}, \\
			D_{0}(12) & = D_{0}^{\infty}(12) \cup \{9\}, & D_{0}(28) & = \{3, 21\}, \\
			D_{0}(13) & = D_{0}^{\infty}(13) \cup \{6, 8\}, & D_{0}(37) & = \{1, 37\}, \\
			D_{0}(15) & = \{1, 2, 3, 5, 10, 15\},
		\end{align*}
		and $D_{0}(m) = D_{0}^{\infty}(m)$ otherwise. Then, we have that
		\begin{align*}
			& \{\deg(x) : x \in X_{0}(n), j(x) \in \Q \setminus \Qcm\} \\
			& \quad = \bigcup_{m \mid n} \{d \cdot \deg(X_{0}(n) \to X_{0}(m)) : d \in D_{0}(m)\}.
		\end{align*}
	\end{theorem}
	
	For the modular curves $X_{1}(n)$, assuming the same conjecture of Zywina, we obtain the following.
	
	\begin{theorem} \label{thm:introduction:finite_x1_degrees}
		Suppose that Conjecture \ref{conj:finite_degrees:sl_intersections} holds. Let $n \geq 1$. Define the sets
		\begin{align*}
			D_{1}(7) & = D_{1}^{\infty}(7) \cup \{9\}, & D_{1}(17) & = \{4, 8, 68, 136\}, \\
			D_{1}(11) & = \{5, 55\}, & D_{1}(21) & = \{3, 6, 9, 18, 21, 42, 63, 126\}, \\
			D_{1}(12) & = D_{1}^{\infty}(12) \cup \{9, 18\}, & D_{1}(28) & = \{9, 18, 63, 126\}, \\
			D_{1}(13) & = D_{1}^{\infty}(13) \cup \{36, 48\}, & D_{1}(37) & = \{6, 18, 222, 666\}, \\
			D_{1}(15) & = \{2, 4, 6, 10, 12, 20, 30, 60\},
		\end{align*}
		and $D_{1}(m) = D_{1}^{\infty}(m)$ otherwise. Then, we have that
		\begin{align*}
			& \{\deg(x) : x \in X_{1}(n), j(x) \in \Q \setminus \Qcm\} \\
			& \quad = \bigcup_{m \mid n} \{d \cdot \deg(X_{1}(n) \to X_{1}(m)) : d \in D_{1}(m)\}.
		\end{align*}
	\end{theorem}
	
	In addition to computing the possible degrees of the points with rational $j$-invariant on the modular curves $X_{0}(n)$ and $X_{1}(n)$, we go a step further by computing the degree of the fibers of these modular curves with rational $j$-invariant. That is to say, rather than determining the degree $\deg(x)$ as $x$ varies through all closed points $x$ on the desired modular curve with $j(x) \in \Q$, we determine the multiset
	\[
		\ldblbrace \deg(x) : x \in X_{i}(n), j(x) = j \rdblbrace,
	\]
	as $j$ varies through $\Q \setminus \Qcm$. The results we obtain are analogous, and in fact imply, the four theorems described above. In the interest of brevity, we omit their statements here, especially as the sets involved are much larger. For more details, see Theorems \ref{thm:infinite_degrees:infinite_fiber_degrees}, \ref{thm:finite_degrees:x0_fiber_degrees} and \ref{thm:finite_degrees:x1_fiber_degrees}.
	
	\subsection{\texorpdfstring{$\cH$}{H}-closures} \label{sec:introduction:closures}
	
	Let $E$ be an elliptic curve defined over $\Q$, and let $\alpha : \varprojlim_{n \geq 1} E[n] \to \Zhat^{2}$ be an isomorphism of $\Zhat$-modules, or equivalently, a compatible choice of bases of $E[n]$, for all $n \geq 1$. Recall that the adelic Galois representation associated to $E$ and $\alpha$ is the group homomorphism
	\begin{align*}
		\rho_{E, \alpha} : G_{\Q} & \to \GL(\Zhat) \\
		\sigma & \mapsto \alpha \circ \sigma^{-1} \circ \alpha^{-1},
	\end{align*}
	where $G_{\Q}$ is the absolute Galois group $\Gal(\Qbar / \Q)$.
	
	Given $n \geq 1$, the degrees of the closed points of the modular curve $X_{0}(n)$ can be determined from group-theoretic properties of the image of the adelic Galois representation $\rho_{E, \alpha}(G_{\Q})$. Let $B_{0}(n)$ be the subgroup of $\GL(\Zhat)$ defined by
	\[
		B_{0}(n) = \left\{\begin{bsmallmatrix} a & b \\ c & d \end{bsmallmatrix} \in \GL(\Zhat) : c \equiv 0 \pmod{n}\right\}.
	\]
	As shown in Section \ref{sec:preliminaries:modular_curve_point_degrees}, if $j(E) \notin \{0, 1728\}$, then the closed points of the modular curve $X_{0}(n)$ with $j$-invariant equal to $j(E)$ are in bijection with the orbits of the set of right cosets $B_{0}(n) \backslash {\GL(\Zhat)}$ under the action of $\rho_{E, \alpha}(G_{\Q})$ by right multiplication. The degree of each closed point is then equal to the size of the corresponding orbit. 
	
	Therefore, in order to determine the degrees of the points on the modular curve $X_{0}(n)$ with rational $j$-invariant, it is sufficient to classify the possible images $\rho_{E, \alpha}(G_{\Q})$ of the adelic Galois representation associated to elliptic curves $E/\Q$. This question, often referred to as Mazur's Program B in the literature, has been well-studied. However, at the present moment, there is no complete classification of such adelic images $\rho_{E, \alpha}(G_{\Q})$ for all elliptic curves $E/\Q$.
	
	To circumvent this issue, we note that instead of understanding the complete adelic image $\rho_{E, \alpha}(G_{\Q})$, it suffices to understand how the adelic image $\rho_{E, \alpha}(G_{\Q})$ acts on the set of right cosets $B_{0}(n) \backslash {\GL(\Zhat)}$. To this end, given a subgroup $H$ of $\GL(\Zhat)$, we say that two subgroups $G$ and $G'$ of $\GL(\Zhat)$ are $\cH$-equivalent if the sets of orbits of $H \backslash {\GL(\Zhat)}$ under $G$ and $G'$ are the same. Similarly, we say that $G$ and $G'$ are $H$-equivalent if the orbits of the identity coset of $H \backslash {\GL(\Zhat)}$ under $G$ and $G'$ are the same. The choice of notation is explained in Sections \ref{sec:preliminaries:gl2} and \ref{sec:closures}. One then finds that the degrees of the closed points on the modular curve $X_{0}(n)$ with $j$-invariant equal to $j(E)$ can be deduced from knowing subgroups $B_{0}(n)$-equivalent to each of the conjugates of $\rho_{E, \alpha}(G_{\Q})$. Moreover, the degree of the fiber of $X_{0}(n)$ with $j$-invariant $j(E)$ can be computed from any subgroup $\cB_{0}(n)$-equivalent to $\rho_{E, \alpha}(G_{\Q})$.
	
	Thus, in order to determine the degrees of the fibers of the modular curve $X_{0}(n)$ with rational $j$-invariant, it suffices to determine a subgroup $\cB_{0}(n)$-equivalent to the image of the adelic Galois representation $\rho_{E, \alpha}(G_{\Q})$ for all elliptic curves $E/\Q$. It so happens that there is a maximal such subgroup, which we call the $\cB_{0}(n)$-closure. The bulk of this paper is therefore concerned with classifying the $\cB_{0}(n)$- and $\cB_{1}(n)$-closures of the adelic images $\rho_{E, \alpha}(G_{\Q})$ of non-CM elliptic curves $E/\Q$.
	
	This classification is done in two parts. Firstly, we determine the $\cB_{0}(n)$- and $\cB_{1}(n)$-closures which occur infinitely often, as $E$ varies through all non-CM elliptic curves over $\Q$. This is done unconditionally, and serves as the basis of the proofs of Theorems \ref{thm:introduction:infinite_x0_degrees} and \ref{thm:introduction:infinite_x1_degrees}.
	
	Secondly, we determine the $\cB_{0}(n)$- and $\cB_{1}(n)$-closures which occur for only finitely many isomorphism classes of elliptic curves $E / \Q$. This is much harder, and we rely on the conjecture of Zywina mentioned above to achieve this. This computation also involves determining the set of rational points on 160 modular curves of positive genus. This second classification then yields the proofs of Theorems \ref{thm:introduction:finite_x0_degrees} and \ref{thm:introduction:finite_x1_degrees}.
	
	The distinction between $H$-equivalence and $\cH$-equivalence of subgroups of $\GL(\Zhat)$ is one of the primary reasons for considering the degrees of the fibers on the modular curves $X_{0}(n)$ and $X_{1}(n)$ rather than the degrees of the points. Indeed, as explained above, computing the degrees of the points naturally leads one to considering say $B_{0}(n)$-closures, while computing the degrees of the fibers leads to considering the $\cB_{0}(n)$-closure. While $\cB_{0}(n)$-equivalence is a stronger constraint than $B_{0}(n)$-equivalence, and hence is possibly harder to determine, computing $\cB_{0}(n)$-closures allows us to work with conjugacy classes of subgroups of $\GL(\Zhat)$, rather than individual subgroups. This greatly simplifies the calculations done in this paper. For more details, see Remark \ref{rmk:closures:fibers_vs_points}.
	
	\subsection{Connections to other work} \label{sec:introduction:connections}
	
	While defined in the context of the previous section, the notions of $H$-equivalence and $\cH$-equivalence have connections to related topics about closed points on modular curves. In particular, these notions are the correct group-theoretic conditions encoding many existing properties of points on modular curves.
	
	For instance, given two open subgroups $G$ and $H$ of $\GL(\Zhat)$, with $-I \in G \leq H$, there is an inclusion morphism of modular curves $f : X_{G} \to X_{H}$. Fixing some elliptic curve $E/\Q$ with $j(E) \notin \{0, 1728\}$, consider a closed point $x \in X_{G}$ with $j$-invariant equal to $j(E)$. It is a natural question to understand when $x$ is the pullback of its image $f(x)$, or in other words, when one has the equality
	\[
		\deg(x) = \deg(f) \cdot \deg(f(x)).
	\]
	Group-theoretically, this occurs precisely when the subgroups $G$ and $H$ are $\rho_{E, \alpha}(G_{\Q})$-equivalent, where $\alpha$ is closely related to the point $x$; see Corollary \ref{thm:closures:primitive_points_connection}.
	
	Going a step further, define a primitive point on a modular curve $X_{G}$ to be a closed point $x \in X_{G}$ such that there does not exist another modular curve $X_{H}$, with $G \leq H$, such that $\deg(x) = \deg(f) \cdot \deg(f(x))$, where $f$ is the inclusion morphism $f : X_{G} \to X_{H}$. In other words, the closed point $x$ is not the pullback of a closed point on some other modular curve. This mirrors the definition of primitive points given in \cite[Section 5]{bourdon2025towards} for the modular curves $X_{1}(n)$. By the above, the closed point $x \in X_{G}$ is primitive if and only if there are no overgroups of $G$ which are $\rho_{E, \alpha}(G_{\Q})$-equivalent to it; we say that such a group $G$ is $\rho_{E, \alpha}(G_{\Q})$-closed. More generally, for any closed point $x \in X_{G}$, there is a unique modular curve $X_{H}$ such that $x$ is the pullback of a primitive point on $X_{H}$. The subgroup $H$ is then equal to the $\rho_{E, \alpha}(G_{\Q})$-closure of $G$ defined in the previous section. To obtain the notion of primitive points defined in \cite[Section 5]{bourdon2025towards}, one restricts to solely considering the modular curves $X_{1}(n)$. These same definitions and results carry through to this setting, thus providing a group-theoretic approach to understanding and computing primitive points.
	
	The notions of $H$-equivalence and $H$-closure can also be applied to the study of isolated points on modular curves. Roughly speaking, a closed point $x$ of degree $d$ on a nice curve $C$ defined over a number field $k$ is said to be isolated if it does not belong to an infinite family of degree $d$ points of $C$ parametrized by some geometric object. For a more precise definition, we refer the reader to \cite{bourdon2019level, terao2025isolated, viray2025isolated}. The study of isolated points on modular curves has garnered much attention recently, and particularly the study of the isolated points on the modular curves $X_{1}(n)$.
	
	As for the case of primitive points, the notion of $H$-equivalence provides the correct group-theoretic formulation for understanding existing notions related to isolated points on modular curves. In particular, we show that an existing argument in the literature, used for instance in \cite[Theorem 38]{bourdon2025towards} and \cite[Theorem 5.3]{terao2025isolated}, can be generalized to a large class of subgroups characterized by $H$-equivalence. This allows us to recontextualize the results of the aforementioned citations as specific cases of this more general result.
	
	Moreover, this greater generality, along with the classification of $\cB_{1}(n)$-closures of adelic Galois images of elliptic curves over $\Q$ explained in the previous section, allows us to compute the set of $j$-invariants of isolated points with rational $j$-invariant on the modular curves $X_{1}(n)$. A conjectural classification of such $j$-invariants has already been put forward in \cite[Conjecture 4]{bourdon2025towards}, and we provide further evidence for this conjecture by showing that it is implied by the conjecture of Zywina mentioned in the previous sections. More precisely, we show the following.
	
	\begin{theorem} \label{thm:introduction:x1_isolated_j_invariants}
		Suppose that Conjecture \ref{conj:finite_degrees:sl_intersections} holds. Fix $n \geq 1$, and let $x \in X_{1}(n)$ be a non-cuspidal, non-CM isolated closed point with $j(x) \in \Q$. Then
		\[
			j(x) \in \left\{-\frac{140625}{8}, \frac{351}{4}, -9317, -162677523113838677\right\}.
		\]
		In particular, \cite[Conjecture 4]{bourdon2025towards} holds.
	\end{theorem}
	
	\subsection{Structure of the paper}
	
	The structure of the paper is as follows. In Section \ref{sec:preliminaries}, we recall the definition and properties of modular curves and their associated objects. Of particular importance, in Corollary \ref{thm:preliminaries:modular_curve_point_degrees}, we give an expression for the degrees of points on modular curves in terms of Galois representations which will be used throughout. In Section \ref{sec:closures}, we define the notions of $H$-equivalence and $\cH$-equivalence briefly explained in Section \ref{sec:introduction:closures}, and undertake a detailed study of their properties. In Section \ref{sec:infinite_degrees}, we classify the degrees of the fibers with rational $j$-invariant on the modular curves $X_{0}(n)$ and $X_{1}(n)$ which occur infinitely often, while Section \ref{sec:finite_degrees} is dedicated to the classification of the degrees of the fibers with rational $j$-invariant on the modular curves $X_{0}(n)$ and $X_{1}(n)$ which occur finitely often. In Section \ref{sec:rational_points}, we compute the rational points on the 160 modular curves mentioned above, which forms a key part of the argument of Section \ref{sec:finite_degrees}. Finally, in Section \ref{sec:isolated_points}, we explore the link between $H$-equivalence and isolated points on modular curves, which culminates in the proof of Theorem \ref{thm:introduction:x1_isolated_j_invariants}.
	
	\subsection{Code}
	
	The computations throughout this paper were performed using the computer algebra system \texttt{Magma} \cite{magma}. The code can be found in the following GitHub repository:
	\begin{center}
		\url{https://github.com/kenjiterao/degrees-x0-x1}
	\end{center}
	
	\subsection{Acknowledgments}
	
	The author thanks Abbey Bourdon, Travis Morrison and James Rawson for helpful and insightful conversations. The author also thanks Samir Siksek for their guidance and comments on this manuscript. This work was supported by the Additional Funding Programme for Mathematical Sciences, delivered by EPSRC (EP/V521917/1) and the Heilbronn Institute for Mathematical Research.
	
	\section{Preliminaries} \label{sec:preliminaries}
	
	In this section, we recall some basic definitions about modular curves and subgroups of $\GL(\Zhat)$ which will be used throughout the remainder of this paper. Most of the discussion of modular curves will follow the exposition and definitions used in \cite{terao2025isolated}, to which we refer the reader to for more details and additional references.
	
	\subsection{Points on curves} \label{sec:preliminaries:points_on_curves}
	
	Let $C$ be a smooth projective curve defined over a number field $k$; that is to say, let $C$ be a smooth, projective, integral, one-dimensional scheme over $k$. Note that while we require the curve $C$ to be integral, we do not require it to be geometrically integral. Given a closed point $x \in C$, we define the degree of $x$, denoted $\deg(x)$, to be the degree of the field extension
	\[
		\deg(x) = [k(x) : k].
	\]
	Recalling that a closed point $x \in C$ corresponds to the Galois orbit of a geometric point $y \in C(\overline{k})$, the degree of $x$ is equivalently the size of the Galois orbit of $y$. Given such a geometric point $y \in C(\overline{k})$, we will denote by $\overline{y} \in C$ the closed point corresponding to the Galois orbit of $y$.
	
	Let $C \to \Spec K \to \Spec k$ be the Stein factorization of $C$, where the morphism $C \to \Spec K$ makes $C$ into a geometrically integral curve over the number field $K$. In particular, if $C/k$ is geometrically integral, we have that $K = k$. As $C$ is a smooth curve, for any closed point $x \in C$, we have that
	\[
		\deg(x) = [K(x) : K] [K : k].
	\]
	In particular, the degree of the field extension $K/k$ divides the degree of the closed point $\deg(x)$. Hence, if $C$ is geometrically disconnected, then $C(k) = \emptyset$.
	
	Let $f : C \to D$ be a finite locally free morphism of degree $d$ between smooth projective curves defined over $k$. Let $y \in D$ be a closed point, and recall that the fiber of $y$ under $f$ is the preimage $f^{-1}(y) \subset C$. We define the degree of the fiber of $y$, denoted $\ddeg_{C}(y)$, to be the multiset
	\[
		\ddeg_{C}(y) = \ldblbrace \deg(x) : x \in C, f(x) = y \rdblbrace.
	\]
	We note that, as $f$ is a finite, locally free morphism of degree $d$, the fiber $f^{-1}(y)$ is a finite union of closed points of $C$, and
	\[
		\sum_{k \in \ddeg_{C}(y)} k = d.
	\]
	
	\subsection{Subgroups of \texorpdfstring{$\GL(\Zhat)$}{GL\_2(Zhat)}} \label{sec:preliminaries:gl2}
	
	Throughout the remainder of this paper, we will heavily rely on subgroups of $\GL(\Zhat)$. As a result, we recall some basic definitions and properties relating to this group.
	
	We denote by $\Zhat$ the ring of profinite integers $\Zhat = \varprojlim_{n \geq 1} \Z[n]$. The group $\GL(\Zhat)$ is a profinite group, with $\GL(\Zhat) = \varprojlim_{n \geq 1} \GL(\Z[n])$. As a result, we obtain surjective projection maps $\pi_{n} : \GL(\Zhat) \to \GL(\Z[n])$ for every $n \geq 1$. By the Chinese reminder theorem, we have that
	\[
		\GL(\Zhat) \cong \prod_{p} \GL(\Z_{p}),
	\]
	where the product ranges over all primes $p$.
	
	We equip $\GL(\Zhat)$ with the profinite topology, and say that subgroups of $\GL(\Zhat)$ are open or closed if they are open or closed in the profinite topology respectively. As is the case in all profinite groups, a subgroup of $\GL(\Zhat)$ is open if and only if it is both closed and has finite index. Given an open subgroup $G$ of $\GL(\Zhat)$, there exists an integer $n \geq 1$ such that $\ker(\pi_{n}) \leq G$. The least such $n$ is called the $\GL$-level, or simply level, of $G$.
	
	For any integer $n \geq 1$, there is a natural bijection between the set of open subgroups of $\GL(\Zhat)$ with level dividing $n$ and the set of subgroups of $\GL(\Z[n])$, given by $G \mapsto \pi_{n}(G)$. As such, throughout the remainder of the paper, we will freely switch between these two equivalent notions.
	
	Given a subgroup $G$ of $\GL(\Zhat)$ and an integer $n \geq 1$, we denote by $G(n)$ the product $G \ker \pi_{n}$, or equivalently, the preimage $\pi_{n}^{-1}(\pi_{n}(G))$. This is an open subgroup of $\GL(\Zhat)$ whose level divides $n$.
	
	The subgroup $\SL(\Zhat)$ is a closed subgroup of $\GL(\Zhat)$. We equip $\SL(\Zhat)$ with the subspace topology, which coincides with the profinite topology arising from the identification $\SL(\Zhat) = \varprojlim_{n \geq 1} \SL(\Z[n])$. We define open and closed subgroups of $\SL(\Zhat)$ as was done for $\GL(\Zhat)$. In the same way, given an open subgroup $G$ of $\SL(\Zhat)$, we define the $\SL$-level of $G$ to be the least $n \geq 1$ such that $G$ contains the kernel of the projection map $\SL(\Zhat) \to \SL(\Z[n])$. By extension, we define the $\SL$-level of an open subgroup $G$ of $\GL(\Zhat)$ to be the $\SL$-level of the open subgroup $G \cap \SL(\Zhat) \leq \SL(\Zhat)$.
	
	Given a subgroup $G$ of $\GL(\Zhat)$, we will denote by $\cG$ the conjugacy class of $G$ in $\GL(\Zhat)$; that is to say, the set $\{g G g^{-1} : g \in \GL(\Zhat)\}$. Similarly, we will use $\cH$ and $\cB$ to denote the conjugacy classes of subgroups $H$ and $B$ respectively. If $G$ is additionally an open subgroup of $\GL(\Zhat)$ of level $n$, the conjugates $g G g^{-1}$ are also open subgroups of $\GL(\Zhat)$ of level $n$, for all $g \in \GL(\Zhat)$. Moreover, conjugation commutes with the identification between open subgroups of $\GL(\Zhat)$ of level dividing $n$ and subgroups of $\GL(\Z[n])$, in the sense that the conjugacy class $\cG$ is precisely the preimage of the conjugacy class of $\pi_{n}(G)$ in $\GL(\Z[n])$. Given two conjugacy classes $\cG$ and $\cG'$ of subgroups of $\GL(\Zhat)$, we write $\cG \leq \cG'$ if there exist two representatives $G \in \cG$ and $G' \in \cG'$ such that $G \subseteq G'$.
	
	There are a number of properties of $G$ which are invariant under conjugation by elements of $\GL(\Zhat)$: for instance, the value of the index $[\GL(\Zhat) : G]$ and the value of the level of $G$, if $G$ is open. By abuse of notation, we will designate the corresponding property of the conjugacy class $\cG$ to be that of $G$, and employ the same notation as for $G$. For instance, we will write the index $[\GL(\Zhat) : \cG]$ of a conjugacy class $\cG$ of open subgroups of $\GL(\Zhat)$ to mean the index $[\GL(\Zhat) : G]$ of any representative $G \in \cG$.
	
	Throughout the remainder of the paper, we will restrict ourselves to using the left action of $\GL(\Zhat)$ on $\Zhat^{2}$ defined by left multiplication on column vectors. In particular, the product $g h$ of two elements $g, h \in \GL(\Zhat)$ corresponds to the composition $g \circ h$ of linear maps, rather than the composition $h \circ g$. For consistency, we therefore define the product $\sigma \eta$ of elements of a symmetric group $\Sym(\Omega)$ to be the composition $\sigma \circ \eta$. A group homomorphism $G \to \Sym(\Omega)$ thus corresponds to a left action of the group $G$ on the set $\Omega$, while a right action gives rise to a homomorphism
	\[
		G^{\op} \to \Sym(\Omega),
	\]
	where $G^{\op}$ denotes the opposite group of $G$. This convention mirrors that used in the majority of the literature on modular curves, and importantly, in \cite{rouse2022ell-adic} and the LMFDB \cite{lmfdb}. However, it is opposite to the convention used in \cite{rouse2015elliptic} and in group theory more broadly.
	
	Let $G$ and $H$ be subgroups of $\GL(\Zhat)$, and denote by $H \backslash {\GL(\Zhat)}$ the set of right cosets of $H$ in $\GL(\Zhat)$. There is a right action of $\GL(\Zhat)$ on the set $H \backslash {\GL(\Zhat)}$ given by right multiplication, which restricts to give a right action of $G$ on $H \backslash {\GL(\Zhat)}$. Given a coset $Hg \in H \backslash {\GL(\Zhat)}$, we will denote by $Hg \cdot G$ the corresponding orbit under the action of $G$. The set of orbits $(H \backslash {\GL(\Zhat)}) / G$ is in bijection with the set of double cosets $H \backslash {\GL(\Zhat)} / G$. For readability, we will therefore use the latter notation to denote the former.
	
	Given an element $k \in \GL(\Zhat)$, we have an isomorphism of right actions given by the commutative diagram
	\[\begin{tikzcd}
		G^{\op} \ar[r] \ar[d, "\cong"] & \Sym(H \backslash {\GL(\Zhat)}) \ar[d, "\cong"] \\
		(k^{-1} G k)^{\op} \ar[r] & \Sym(H \backslash {\GL(\Zhat)}),
	\end{tikzcd}\]
	where the top and bottom arrows are given by the right actions of $G$ and $k^{-1} G k$ respectively on the set $H \backslash {\GL(\Zhat)}$ as defined above, the left arrow is induced by the conjugation map $G \to k^{-1} G k$, and the right arrow is induced by the bijection of sets $Hg \mapsto Hgk$. By abuse of notation, we will therefore define the action of the conjugacy class $\cG$ on the set $H \backslash {\GL(\Zhat)}$ to be the action of $G$ on the set $H \backslash {\GL(\Zhat)}$ for some fixed representative $G \in \cG$. As before, we will then write $H \backslash {\GL(\Zhat)} / \cG$ to denote the set of orbits $H \backslash {\GL(\Zhat)} / G$. We will not however employ the notation $Hg \cdot \cG$, as the orbit of the specific coset $Hg$ may depend on the representative $G \in \cG$ chosen.
	
	Suppose moreover that $H$ is an open subgroup of $\GL(\Zhat)$ of level $n$. The set of cosets $H \backslash {\GL(\Zhat)}$ is then in bijection with the set of cosets $\pi_{n}(H) \backslash {\GL(\Z[n])}$, and there is a morphism of group actions given by the commutative diagram
	\begin{equation}
		\begin{tikzcd}
			G^{\op} \ar[r] \ar[d, twoheadrightarrow] & \Sym(H \backslash {\GL(\Zhat)}) \ar[d, "\cong"] \\
		\pi_{n}(G)^{\op} \ar[r] & \Sym(\pi_{n}(H) \backslash {\GL(\Z[n])}),
		\end{tikzcd}
		\label{eq:preliminaries:finite_group_action_correspondence}
	\end{equation}
	where the bottom arrow is induced by the right action of $\pi_{n}(G)$ on $\pi_{n}(H) \backslash {\GL(\Z[n])}$ given by right multiplication.
	
	Finally, we define two families of open subgroups of $\GL(\Zhat)$, as follows. For any integer $n \geq 1$, we let
	\begin{align*}
		B_{0}(n) & = \pi_{n}^{-1} (\{\begin{bsmallmatrix}a & b \\ 0 & c\end{bsmallmatrix} \in \GL(\Z[n]) : a, c \in \Z*[n], b \in \Z[n]\}), \\
		B_{1}(n) & = \pi_{n}^{-1} (\{\begin{bsmallmatrix}a & b \\ 0 & c\end{bsmallmatrix} \in \GL(\Z[n]) : a \in \{\pm 1\}, b \in \Z[n], c \in \Z*[n]\}).
	\end{align*}
	The subgroups $B_{0}(n)$ and $B_{1}(n)$ are both open subgroups of $\GL(\Zhat)$ of level exactly $n$. In addition, we define the closed subgroups
	\begin{align*}
		B_{0} = \bigcap_{n \geq 1} B_{0}(n), && B_{1} = \bigcap_{n \geq 1} B_{1}(n).
	\end{align*}
	We note that the subgroups $B_{0}(n)$ and $B_{1}(n)$ coincide precisely with the notation $G(n)$, defined above for subgroups $G$ of $\GL(\Zhat)$, with $G$ respectively $B_{0}$ and $B_{1}$.
	
	\subsection{Galois representations of elliptic curves} \label{sec:preliminaries:galois_representations}
	
	Let $E / k$ be an elliptic curve defined over a number field $k$. Following \cite[Section 4.5]{terao2025isolated}, we define a profinite level structure on $E$ to be an isomorphism of $\Zhat$-modules
	\[
		\alpha : \varprojlim_{n \geq 1} E[n] \to \Zhat^{2},
	\]
	or equivalently, a compatible choice of bases of $E[n]$ for all $n \geq 1$.
	
	Given a profinite level structure $\alpha$ on $E$, we define the adelic Galois representation associated to $E$ and $\alpha$ to be the group homomorphism
	\begin{align*}
		\rho_{E, \alpha} : G_{k} & \to \GL(\Zhat) \\
		\sigma & \mapsto \alpha \circ \sigma \circ \alpha^{-1},
	\end{align*}
	where $G_{k}$ denotes the absolute Galois group $\Gal(\overline{k} / k)$. The image $\rho_{E, \alpha}(G_{k})$ is called the adelic Galois image associated to $E$ and $\alpha$. We note that if $E$ is a non-CM elliptic curve, then the adelic Galois image $\rho_{E, \alpha}(G_{k})$ is an open subgroup of $\GL(\Zhat)$.
	
	We let $A_{E, \alpha}$ be the subgroup
	\[
		A_{E, \alpha} = \{\alpha \circ \varphi \circ \alpha^{-1}: \varphi \in \Aut(E_{\Qbar})\} \leq \GL(\Zhat),
	\]
	and call it the automorphism image associated to $E$ and $\alpha$. We note that, if $j(E) \notin \{0, 1728\}$, then $A_{E, \alpha} = \{\pm I\}$. The automorphism image $A_{E, \alpha}$ is normalized by the adelic Galois image $\rho_{E, \alpha}(G_{\Q(j)})$, and so their product $\rho_{E, \alpha}(G_{\Q(j)}) A_{E, \alpha}$ is a subgroup of $\GL(\Zhat)$. We write
	\[
		G_{E, \alpha} = \rho_{E, \alpha}(G_{\Q(j)}) A_{E, \alpha},
	\]
	and call this the extended adelic Galois image associated to $E$ and $\alpha$.
	
	Given a second profinite level structure $\beta$ on $\alpha$, we naturally have that
	\[
		\rho_{E, \beta}(G_{k}) = (\beta \circ \alpha^{-1}) \rho_{E, \alpha}(G_{k}) (\beta \circ \alpha^{-1})^{-1}.
	\]
	Viewing $\beta \circ \alpha^{-1}$ as an element of $\GL(\Zhat)$, it follows that the conjugacy class of the adelic Galois image $\rho_{E, \alpha}(G_{k})$ is independent of the profinite level structure $\alpha$. We denote this conjugacy class $\cG_{E}$, and call it the adelic Galois image of $E$.
	
	Let $j \in \Qbar$ be an algebraic number, and let $E$ be an elliptic curve defined over $\Q(j)$ such that $j(E) = j$. As before, let $\alpha$ be a profinite level structure on $E$. By a similar argument, the conjugacy class of the extended adelic Galois image $G_{E, \alpha}$ does not depend on the choice of profinite level structure $\alpha$, nor on the twist of $E$, that is to say, the choice of $E / \Q(j)$ with $j(E) = j$. We denote this conjugacy class $\cG_{j}$, and call it the extended adelic Galois image of $j$.
	
	\subsection{Modular curves} \label{sec:preliminaries:modular_curves}
	
	Let $H$ be an open subgroup of $\GL(\Zhat)$ of level $n$. We define the modular curves $X_{H}$ and $Y_{H}$ to be the generic fibers of the coarse moduli spaces of the algebraic stacks $\mathcal{M}_{\pi_{n}(H)}$ and $\mathcal{M}_{\pi_{n}(H)}^{0}$ respectively, which parametrize generalized elliptic curves, or respectively, elliptic curves, with $\pi_{n}(H)$-level structure. For more information on the definition of the stacks $\mathcal{M}_{\pi_{n}(H)}$ and $\mathcal{M}_{\pi_{n}(H)}^{0}$, we refer the reader to \cite{deligne1973schemas}.
	
	The modular curve $X_{H}$ is a smooth projective curve defined over $\Q$, and is geometrically integral if and only if $\det(H) = \Zhat*$. The modular curve $Y_{H}$ is an affine subscheme of $X_{H}$. We call the closed points of $X_{H} \setminus Y_{H}$ the cusps of $X_{H}$, while the closed points of $Y_{H}$ are said to be non-cuspidal.
	
	By virtue of the definition of $Y_{H}$ as a coarse moduli space, one can give an explicit description of the set of geometric points of $Y_{H}$. Given an elliptic curve $E$ defined over a number field $k$ and two profinite level structures $\alpha, \beta$ on $E$, we say that $\alpha$ and $\beta$ are $H$-equivalent if $\alpha \circ \beta^{-1} \in H$. Denoting by $[\alpha]_{H}$ the $H$-equivalence class of the profinite level structure $\alpha$, we say that two pairs $(E, [\alpha]_{H})$ and $(E', [\alpha']_{H})$ are equivalent if there exists an isomorphism $\varphi : E \to E'$ of elliptic curves such that
	\[
		[\alpha]_{H} = [\alpha' \circ \varphi]_{H}.
	\]
	Then, by \cite[Section 4]{terao2025isolated}, the set of geometric points $Y_{H}(\Qbar)$ is in bijection with the set of equivalence classes $[(E, [\alpha]_{H})]$, as $E$ varies through all elliptic curves defined over a number field, and $\alpha$ varies through all profinite level structures on $E$.
	
	The Galois action of the absolute Galois group $G_{\Q}$ on the geometric points can also be described explicitly. Namely, by \cite[Section 4]{terao2025isolated}, the left action of $G_{\Q}$ on the geometric points $Y_{H}(\Qbar)$ is given by
	\[
		\sigma \cdot [(E, [\alpha]_{H})] = [(\sigma E, [\alpha \circ \sigma^{-1}]_{H})],
	\]
	where $\sigma E$ is the elliptic curve obtained by acting by $\sigma$ on the coefficients of the defining equation of $E$.
	
	We denote by $X(1)$ the modular curve $X_{\GL(\Zhat)}$, and note that there is a canonical isomorphism $X(1) \cong \Pone_{\Q}$ defined over $\Q$, given on $Y_{H}(\Qbar)$ by
	\[
		[(E, [\alpha]_{\GL(\Zhat)})] \mapsto j(E).
	\]
	Throughout, we make implicit use of this isomorphism, and write $j \in X(1)(\Qbar)$ to denote the geometric point of $X(1)$ corresponding to $j \in \Qbar = \AA^{1}_{\Q}(\Qbar) \subset \Pone_{\Q}(\Qbar)$.
	
	Given an element $k \in \GL(\Zhat)$, there is a $\Q$-isomorphism of modular curves $X_{H} \cong X_{k H k^{-1}}$. Thus, given a conjugacy class $\cH$ of open subgroups of $\GL(\Zhat)$, we define the modular curve $X_{\cH}$ to be the modular curve $X_{H}$, for any representative $H \in \cH$.
	
	Let $H \leq H'$ be two open subgroups of $\GL(\Zhat)$. The inclusion $H \leq H'$ induces a finite locally free morphism of modular curves $f : X_{H} \to X_{H'}$, of degree $[\pm H' : \pm H]$, given on non-cuspidal geometric points by
	\begin{alignat*}{2}
		f : && Y_{H}(\Qbar) & \to Y_{H'}(\Qbar) \\
		&& [(E, [\alpha]_{H})] & \mapsto [(E, [\alpha]_{H'})].
	\end{alignat*}
	In particular, for any open subgroup $H$ of $\GL(\Zhat)$, we obtain a morphism $j : X_{H} \to X(1)$ called the $j$-map. Let $x \in X_{H}$ be a closed point. Then $x$ is cuspidal if $j(x) = \infty \in X(1)$, and we say that $x$ is CM if $j(x) \in X(1)$ is the Galois orbit of a CM $j$-invariant $j \in \Qbar$.
	
	The $j$-map $j : X_{H} \to X(1)$ commutes with the isomorphism $X_{H} \cong X_{H'}$, for any conjugate $H' \in \cH$ of $H$. Therefore, we may also speak of the $j$-map $j : X_{\cH} \to X(1)$, and define cuspidal and CM-points on $X_{\cH}$ in the same way. In general however, given two conjugacy classes $\cH \leq \cH'$ of open subgroups of $\GL(\Zhat)$, there may be multiple inclusion morphisms $X_{\cH} \to X_{\cH'}$, since there may be multiple conjugates $H \in \cH$ contained in the same conjugate $H' \in \cH'$.
	
	As is customary, we denote by $X_{0}(n)$ the modular curve $X_{B_{0}(n)}$ and by $X_{1}(n)$ the modular curve $X_{B_{1}(n)}$, for all $n \geq 1$. We note that the modular curve $X_{1}(n)$ is usually defined to be the modular curve $X_{B_{1}(n)'}$, where
	\[
		B_{1}(n)' = \pi_{n}^{-1} (\{\begin{bsmallmatrix}1 & a \\ 0 & b\end{bsmallmatrix} \in \GL(\Z[n]) : a \in \Z[n], b \in \Z*[n]\}).
	\]
	However, as $B_{1}(n) = \pm B_{1}(n)'$, the two resulting curves are isomorphic over $\Q$. We prefer our definition of $B_{1}(n)$ as having a subgroup containing $-I$ will simplify the notation throughout the remainder of the paper.
	
	\subsection{Degrees of points on modular curves} \label{sec:preliminaries:modular_curve_point_degrees}
	
	The moduli descriptions of $Y_{H}(\Qbar)$ and its Galois action, presented in the previous section, provide a practical way to understand the degrees of non-cuspidal points on the modular curve $X_{H}$. In this section, we provide a reformulation of this in terms of the adelic Galois representations defined in Section \ref{sec:preliminaries:galois_representations}. More precisely, we prove the following result.
	
	\begin{theorem} \label{thm:preliminaries:galois_action_bijection}
		Let $H$ be an open subgroup of $\GL(\Zhat)$. Let $j \in \Qbar$, let $E / \Q(j)$ be an elliptic curve such that $j(E) = j$, and let $\alpha$ be a profinite level structure on $E$. Then, there is a morphism of group actions given by the commutative diagram
		\[\begin{tikzcd}
			G_{\Q(j)} \arrow[d] \arrow[r] & \Sym\!\left(\{x \in X_{H}(\Qbar) : j(x) = j\}\right) \ar[d, "\cong"]\\
			G_{E, \alpha} {}^{\operatorname{op}} \arrow[r] & \Sym(H \backslash {\GL(\Zhat)} / A_{E, \alpha}),
		\end{tikzcd}\]
		where the top arrow is given by the left action of $G_{\Q}$ on $X_{H}(\Qbar)$ and the bottom arrow is defined by the right action of $G_{E, \alpha}$ on $H \backslash {\GL(\Zhat)} / A_{E, \alpha}$ given by
		\[
			(H k A_{E, \alpha})^{g} = H k g A_{E, \alpha},
		\]
		for all $g \in G_{E, \alpha}$ and $k \in \GL(\Zhat)$.
		
		Moreover, the image of $G_{\Q(j)}$ in $\Sym(H \backslash {\GL(\Zhat)} / A_{E, \alpha})$ is equal to the image of $G_{E, \alpha} {}^{\operatorname{op}}$.
	\end{theorem}
	
	\begin{proof}
		By the moduli description of $Y_{H}(\Qbar)$, we know that the set $\{x \in X_{H}(\Qbar) : j(x) = j\}$ is in bijection with the set of equivalence classes $[(E', [\beta]_{H})]$, where $E'$ is an elliptic curve with $j(E') = j$, and $\beta$ is a profinite level structure on $E'$. By \cite[Lemma 4.9]{terao2025isolated}, we may take $E' = E$. Moreover, for two profinite level structures $\beta$ and $\beta'$ on $E$, we have that
		\begin{align*}
			& [(E, [\beta]_{H})] = [(E, [\beta']_{H})] \\
			& \quad \iff \exists \, \varphi \in \Aut(E_{\overline{\Q}}) \text{ s.t. } [\beta]_{H} = [\beta' \circ \varphi]_{H} \\
			& \quad \iff \exists \, \varphi \in \Aut(E_{\overline{\Q}}) \text{ s.t. } \beta' \circ \varphi \circ \beta^{-1} \in H \\
			& \quad \iff \exists \, \varphi \in \Aut(E_{\overline{\Q}}) \text{ s.t. } (\beta' \circ \alpha^{-1}) (\alpha \circ \varphi \circ \alpha^{-1}) (\alpha \circ \beta^{-1}) \in H \\
			& \quad \iff \exists \, \varphi \in \Aut(E_{\overline{\Q}}) \text{ s.t. } H (\beta' \circ \alpha^{-1}) (\alpha \circ \varphi \circ \alpha^{-1}) = H (\beta \circ \alpha^{-1}) \\
			& \quad \iff H (\beta' \circ \alpha^{-1}) A_{E, \alpha} = H (\beta \circ \alpha^{-1}) A_{E, \alpha}.
		\end{align*}
		Precomposition by $\alpha^{-1}$ gives a bijection between the set of profinite level structures on $E$ and the group $\GL(\Zhat)$. Therefore, the above shows that the set $\{x \in X_{H}(\Qbar) : j(x) = j\}$ is in bijection with the set of double cosets $H \backslash {\GL(\Zhat)} / A_{E, \alpha}$.
		
		Let $\sigma$ be an element of the absolute Galois group $G_{\Q(j)}$. Given a point $[(E, [\beta]_{H})] \in X_{H}(\Qbar)$, we know that
		\[
			\sigma \cdot [(E, [\beta]_{H})] = [(\sigma E, [\beta \circ \sigma^{-1}]_{H})].
		\]
		As $E$ is defined over $\Q(j)$, we have that $\sigma E = E$, and so
		\[
			\sigma \cdot [(E, [\beta]_{H})] = [(E, [\beta \circ \sigma^{-1}]_{H})].
		\]
		Combining this with the above bijection, we obtain a commutative diagram
		\[\begin{tikzcd}
			G_{\Q(j)} \ar[r] \ar[rd] & \Sym(\{x \in X_{H}(\Qbar) : j(x) = j\}) \ar[d, "\cong"] \\
			& \Sym(H \backslash {\GL(\Zhat)} / A_{E, \alpha}),
		\end{tikzcd}\]
		where the bottom-left homomorphism maps an element $\sigma \in G_{\Q(j)}$ to the bijection defined by
		\[
			H k A_{E, \alpha} \mapsto H k \rho_{E, \alpha} (\sigma^{-1}) A_{E, \alpha},
		\]
		for all $k \in \GL(\Zhat)$. It follows that we may write
		\[\begin{tikzcd}
			G_{\Q(j)} \ar[r] \ar[d] & \Sym(\{x \in X_{H}(\Qbar) : j(x) = j\}) \ar[d, "\cong"] \\
			G_{E, \alpha} \ar[r] & \Sym(H \backslash {\GL(\Zhat)} / A_{E, \alpha}),
		\end{tikzcd}\]
		where the left map is the adelic Galois representation $\rho_{E, \alpha}$, and the bottom homomorphism maps an element $g \in G_{E, \alpha}$ to the bijection defined by
		\[
			H k A_{E, \alpha} \mapsto H k g^{-1} A_{E, \alpha},
		\]
		for all $k \in \GL(\Zhat)$. Note that the latter map is well-defined as the automorphism image $A_{E, \alpha}$ is a normal subgroup of $G_{E, \alpha}$.
		
		To obtain the first statement of the theorem, it now suffices to note that there exists an isomorphism $G_{E, \alpha} \to G_{E, \alpha} {}^{\op}$ given by $g \mapsto g^{-1}$.
		
		For the second statement, recall that we have $G_{E, \alpha} = \rho_{E, \alpha}(G_{\Q(j)}) A_{E, \alpha}$. Moreover, by construction, the automorphism image $A_{E, \alpha}$ is contained in the kernel of the map $G_{E, \alpha} \to \Sym(H \backslash {\GL(\Zhat)} / A_{E, \alpha})$. Therefore, the image of $G_{\Q(j)}$ in $\Sym(H \backslash {\GL(\Zhat)} / A_{E, \alpha})$ is equal to the image of $G_{E, \alpha}$. The statement now follows from the fact that the map $G_{E, \alpha} \to G_{E, \alpha} {}^{\op}$ is an isomorphism.
	\end{proof}
	
	Taking $j \notin \{0, 1728\}$, we obtain a very simple expression for the degree of the non-cuspidal points and fibers of the modular curve $X_{H}$, which will be central to the remainder of the paper.
	
	\begin{corollary} \label{thm:preliminaries:modular_curve_point_degrees}
		Let $H$ be an open subgroup of $\GL(\Zhat)$, and let $j \in \Qbar \setminus \{0, 1728\}$. Then, we have
		\[
			\{\deg(x) : x \in X_{H}, j(x) = \overline{j} \in X(1)\} = \{[\Q(j) : \Q] \left|(\pm H) \cdot G\right| : G \in \cG_{j}\},
		\]
		and
		\begin{align*}
			\ddeg_{X_{H}}(\overline{j}) = \ldblbrace [\Q(j) : \Q] \left| \Omega \right| : \Omega \in \pm H \backslash {\GL(\Zhat)} / \cG_{j} \rdblbrace,
		\end{align*}
		where we view $\pm H \backslash {\GL(\Zhat)} / \cG_{j}$ as the set of orbits of $\pm H \backslash {\GL(\Zhat)}$ under the right action of some representative $G \in \cG_{j}$.
	\end{corollary}
	
	\begin{proof}
		Let $x \in X_{H}$ be a closed point with $j(x) = \overline{j} \in X(1)$, and let $y \in X_{H}(\Qbar)$ be a geometric point such that $j(y) = j \in X(1)(\Qbar)$ and $\overline{y} = x$. By definition, the absolute Galois group $G_{\Q}$ acts transitively on the conjugates of $j \in X(1)(\Qbar)$, with stabilizer equal to the absolute Galois group $G_{\Q(j)}$. Therefore, we have that
		\[
			\deg(x) = |G_{\Q} \cdot y| = [\Q(j) : \Q] |G_{\Q(j)} \cdot y|.
		\]
		It follows that
		\[
			\ddeg_{X_{H}}(\overline{j}) = \ldblbrace [\Q(j) : \Q] \left| \Omega \right| : \Omega \in G_{\Q(j)} \backslash \{y \in X_{H}(\Qbar) : j(y) = j\} \rdblbrace,
		\]
		where $G_{\Q(j)} \backslash \{y \in X_{H}(\Qbar) : j(y) = j\}$ denotes the set of orbits of the fiber $\{y \in X_{H}(\Qbar) : j(y) = j\}$ under the left action of $G_{\Q(j)}$. By Theorem \ref{thm:preliminaries:galois_action_bijection}, for $E$ and $\alpha$ defined as in the statement of Theorem \ref{thm:preliminaries:galois_action_bijection}, we obtain that
		\[
			\ddeg_{X_{H}}(\overline{j}) = \ldblbrace [\Q(j) : \Q] \left| \Omega \right| : \Omega \in (H \backslash {\GL(\Zhat)} / A_{E, \alpha}) / G_{E, \alpha} \rdblbrace,
		\]
		where the right action of $G_{E, \alpha}$ on the set of double cosets $H \backslash {\GL(\Zhat)} / A_{E, \alpha}$ is given by
		\[
			(H k A_{E, \alpha})^{g} = H k g A_{E, \alpha},
		\]
		for all $g \in G_{E, \alpha}$ and $k \in \GL(\Zhat)$.
		
		Since $j \notin \{0, 1728\}$, it follows that $A_{E, \alpha} = \{\pm I\}$. In particular, the set of double cosets $H \backslash {\GL(\Zhat)} / A_{E, \alpha}$ is equal to the set of single cosets $\pm H \backslash {\GL(\Zhat)}$. Under this identification, the action of $G_{E, \alpha}$ is simply given by right multiplication on the cosets of $\pm H$, and we obtain
		\[
			\ddeg_{X_{H}}(\overline{j}) = \ldblbrace [\Q(j) : \Q] \left| \Omega \right| : \Omega \in \pm H \backslash {\GL(\Zhat)} / G_{E, \alpha} \rdblbrace.
		\]
		Noting that $\cG_{j}$ is the conjugacy class of $G_{E, \alpha}$, we obtain the second statement.
		
		Let $k \in \GL(\Zhat)$. We have that
		\begin{align*}
			|(\pm H k) \cdot G_{E, \alpha}| & = |\{\pm H k g : g \in G_{E, \alpha}\}| \\
			& = |\{\pm H k g k^{-1} : g \in G_{E, \alpha}\}| \\
			& = |(\pm H) \cdot (k G_{E, \alpha} k^{-1})|,
		\end{align*}
		where the second equality stems from the fact that multiplication by $k^{-1}$ is a bijection on $\GL(\Zhat)$. Therefore, we have
		\begin{align*}
			& \{\deg(x) : x \in X_{H}, j(x) = \overline{j} \in X(1)\} \\
			& \quad = \{[\Q(j) : \Q] \left| \Omega \right| : \Omega \in \pm H \backslash {\GL(\Zhat)} / G_{E, \alpha}\} \\
			& \quad = \{[\Q(j) : \Q] \left| (\pm H k) \cdot G_{E, \alpha} \right| : k \in \GL(\Zhat)\} \\
			& \quad = \{[\Q(j) : \Q] \left| (\pm H) \cdot k G_{E, \alpha} k^{-1} \right| : k \in \GL(\Zhat)\} \\
			& \quad = \{[\Q(j) : \Q] \left| (\pm H) \cdot G \right| : G \in \cG_{j}\},
		\end{align*}
		as required.
	\end{proof}
	
	We conclude this section with a well-known corollary for the existence of rational points on modular curves.
	
	\begin{corollary} \label{thm:preliminaries:modular_curve_rational_point}
		Let $H$ be an open subgroup of $\GL(\Zhat)$, and denote by $\cH$ the conjugacy class of $H$. Let $j \in \Qbar \setminus \{0, 1728\}$. Then there exists a point $x \in X_{H}(\Q(j))$ such that $j(x) = j$ if and only if $\cG_{j} \leq \pm \cH$.
	\end{corollary}
	
	\begin{proof}
		By Corollary \ref{thm:preliminaries:modular_curve_point_degrees}, there exists a point $x \in X_{H}(\Q(j))$ such that $j(x) = j$ if and only if $|(\pm H) \cdot G| = 1$ for some representative $G \in \cG_{j}$. The latter condition is equivalent to the condition that $G \leq \pm H$ for some representative $G \in \cG_{j}$, or in other words, $\cG_{j} \leq \pm \cH$.
	\end{proof}
	
	\section{\texorpdfstring{$\cH$}{H}-closures} \label{sec:closures}
	
	Let $H$ be an open subgroup of $\GL(\Zhat)$. By Corollary \ref{thm:preliminaries:modular_curve_point_degrees}, the degrees of the points of the modular curve $X_{H}$ with $j$-invariant $j \in \Qbar$ can be determined from knowledge of the extended adelic Galois image $\cG_{j}$ of $j$. This latter question is precisely the statement of Mazur's Program B, and has been extensively studied in existing work. However, despite the wealth of knowledge on this question, presently there does not exist a complete classification of the possible images $\cG_{j}$, even when restricting to rational $j$-invariants.
	
	However, knowing the full extended adelic Galois image $\cG_{j}$ of $j$ is not necessary in order to compute the degrees of the points of $X_{H}$ with $j$-invariant $j$. Instead, it suffices to understand the orbits of $\pm H \backslash {\GL(\Zhat)}$ under the right action of $\cG_{j}$. In particular, rather than determine the image $\cG_{j}$, it is sufficient to determine a larger conjugacy class $\cG_{j} \leq \cG$ such that the orbits of $\pm H \backslash {\GL(\Zhat)}$ under the right actions of $\cG_{j}$ and $\cG$ are the same. As it happens, there is a unique maximal such conjugacy class $\cG$, which we call the $\pm \cH$-closure of $\cG_{j}$. More precisely, we set the following definitions.
	
	\begin{definition}
		Let $G, G'$ and $H$ be three subgroups of $\GL(\Zhat)$, and let $\cH$ denote the conjugacy class of $H$. Recall that the group $\GL(\Zhat)$ acts on the coset space $H \backslash {\GL(\Zhat)}$ by right multiplication.
		\begin{enumerate}
			\item We say that $G$ and $G'$ are \textbf{$H$-equivalent} if the orbits $H \cdot G$ and $H \cdot G'$ are equal, or equivalently, if the products $H G$ and $H G'$ are equal.
			\item We say that $G$ is \textbf{$H$-closed} if there does not exist a proper overgroup $G'$ of $G$ which is $H$-equivalent to $G$.
			\item We define the \textbf{$H$-closure of $G$}, denoted $\clo{H}{G}$, to be the unique $H$-closed subgroup of $\GL(\Zhat)$ which is $H$-equivalent to $G$. Equivalently, it is the largest overgroup of $G$ contained in the product $H G$.
			\item We say that $G$ and $G'$ are \textbf{$\cH$-equivalent} if the sets of orbits $H \backslash {\GL(\Zhat)} / G$ and $H \backslash {\GL(\Zhat)} / G'$ are equal, or equivalently, if the double cosets $H g G$ and $H g G'$ are equal for all $g \in \GL(\Zhat)$. This is also equivalent to $G$ and $G'$ being $H'$-equivalent for all $H' \in \cH$, whence follows the terminology.
			\item We say that $G$ is \textbf{$\cH$-closed} if there does not exist a proper overgroup $G'$ of $G$ which is $\cH$-equivalent to $G$.
			\item We define the \textbf{$\cH$-closure of $G$}, denoted $\clo{\cH}{G}$, to be the unique $\cH$-closed subgroup of $\GL(\Zhat)$ which is $\cH$-equivalent to $G$. Equivalently, it is the largest overgroup of $G$ contained in the intersection $\bigcap_{H' \in \cH} H' G$.
			\item We note that, for any $g \in \GL(\Zhat)$, we have $\clo{\cH}{(G^{g})} = (\clo{\cH}{G})^{g}$. Therefore, if $\cG$ denotes the conjugacy class of $G$, we define the \textbf{$\cH$-closure of $\cG$}, denoted $\clo{\cH}{\cG}$, to be the conjugacy class of $\clo{\cH}{G}$. Note that the same does not work for $\clo{H}{G}$, as $\clo{H}{(G^{g})} \neq (\clo{H}{G})^{g}$.
		\end{enumerate}
	\end{definition}
	
	\begin{remark} \label{rmk:closures:explicit_description}
		The statements in the above definition, such as the existence and uniqueness of $H$- and $\cH$-closures, are straightforward to deduce from the definitions. However, we can also give an explicit construction for the $H$- and $\cH$-closures of $G$, which provides an alternative method for resolving these questions. This explicit description will also prove to be useful later. The construction is as follows.
		
		The right action of $\GL(\Zhat)$ on the coset space $H \backslash {\GL(\Zhat)}$ yields a group homomorphism
		\[
			\varphi : \GL(\Zhat)^{\op} \to \Sym(H \backslash {\GL(\Zhat)}).
		\]
		Let $\{\Omega_{i}\}_{i \in I}$ be the set of orbits of $H \backslash {\GL(\Zhat)}$ under the action of $G$, where we let $\Omega_{1}$ be the orbit of the identity coset $H$. Then, the $H$-closure of $G$ is equal to
		\[
			\clo{H}{G} = \varphi^{-1}\!\left(\Sym(\Omega_{1}) \times \Sym\left(\bigcup_{i \in I \setminus \{1\}} \Omega_{i}\right)\right) \leq \GL(\Zhat)^{\op} \cong \GL(\Zhat).
		\]
		Similarly, the $\cH$-closure of $G$ is given by
		\[
			\clo{\cH}{G} = \varphi^{-1}\!\left(\prod_{i \in I} \Sym(\Omega_{i})\right).
		\]
	\end{remark}
	
	Leveraging these definitions, we can now restate Corollary \ref{thm:preliminaries:modular_curve_point_degrees} using the notions of $H$- and $\cH$-closures.
	
	\begin{corollary} \label{thm:closures:modular_curve_point_degrees}
		Let $H$ be an open subgroup of $\GL(\Zhat)$, and let $j \in \Qbar$ be such that $j \notin \{0, 1728\}$. Denote by $\pm \cH$ the conjugacy class of $\pm H$ in $\GL(\Zhat)$. Then, we have
		\[
			\{\deg(x) : x \in X_{H}, j(x) = \overline{j} \in X(1)\} = \{[\Q(j) : \Q] \left|(\pm H) \cdot \clo{\pm H}{G}\right| : G \in \cG_{j}\},
		\]
		and
		\begin{align*}
			\ddeg_{X_{H}}(\overline{j}) = \ldblbrace [\Q(j) : \Q] \left| \Omega \right| : \Omega \in \pm H \backslash {\GL(\Zhat)} / \clo{\pm \cH}{\cG_{j}} \rdblbrace.
		\end{align*}
	\end{corollary}
	
	This reformulation demonstrates that, in order to compute the degrees of the points on the modular curve $X_{H}$ with $j$-invariant $j$, it suffices to know the $\pm H$-closures $\clo{\pm H}{G}$, for all conjugates $G \in \cG_{j}$. Similarly, in order to determine the degree of the fiber of $X_{H}$ above $j$, it suffices to know the $\pm \cH$-closure $\clo{\pm \cH}{\cG_{j}}$. Thus, throughout the remainder of the paper, our principal aim will be to classify the $\cB_{0}(n)$- and $\cB_{1}(n)$-closures of the conjugacy classes $\cG_{j}$, as $j$ varies through all non-CM rational $j$-invariants $j \in \Q$.
	
	\begin{remark} \label{rmk:closures:fibers_vs_points}
		We take a moment to explain the rationale behind computing the degrees of the fibers of the modular curves $X_{0}(n)$ and $X_{1}(n)$, rather than simply the degrees of the points on said curves. Firstly, this provides a stronger result, as the degrees of the points on the curves can be recovered from the degrees of the fibers.
		
		However, another consideration was the difference between $H$-closures and $\cH$-closures. As explained above, computing the degrees of the points naturally leads to considering $H$-closures, while computing the degrees of the fibers requires computing $\cH$-closures. Working with $\cH$-closures allows one to work at the level of conjugacy classes of subgroups of $\GL(\Zhat)$, while working with $H$-closures requires one to work with subgroups of $\GL(\Zhat)$, as $H$-closed subgroups are not invariant under conjugation. The latter introduces additional computational complexity, and we found it simpler to work with the former.
		
		This choice does bring some downsides however, notably the fact that the $\cH$-closure is often a subgroup of higher index than the $H$-closure. As a result, it may be possible to obtain stronger, perhaps unconditional, results by working directly with $H$-closures instead. We leave this for future work.
	\end{remark}
	
	\subsection{Properties of \texorpdfstring{$\cH$}{H}-closures} \label{sec:closures:properties}
	
	Having now established the notion of $\cH$-closures as the principal object of study throughout the remainder of the paper, we dedicate this section to understanding some of the fundamental properties of $\cH$-closures. Unless specified, all of the properties in this section hold for both $H$- and $\cH$-closures. However, as we will exclusively work with the latter, we omit the proofs for the former. As before, throughout we use the calligraphic $\cG$ and $\cH$ to denote the conjugacy classes of subgroups $G$ and $H$ in $\GL(\Zhat)$.
	
	Firstly, we show that the $\cH$-closure $\clo{\cH}{G}$ is monotone in both $G$ and $\cH$.
	
	\begin{lemma} \label{thm:closures:inclusion_preserving_in_G}
		Let $G$, $G'$ and $H$ be three subgroups of $\GL(\Zhat)$ such that $G \leq G'$. Then, we have
		\[
			\clo{\cH}{G} \leq \clo{\cH}{G'}.
		\]
	\end{lemma}
	
	\begin{proof}
		We know that
		\[
			\clo{\cH}{G} \subseteq \bigcap_{H' \in \cH} H' G.
		\]
		Fix some conjugate $H' \in \cH$ of $H$. For any $g \in G'$, we have that
		\[
			\clo{\cH}{G} \subseteq \bigcap_{H' \in \cH} H' G \subseteq g^{-1} H' g G,
		\]
		and so
		\[
			H' g (\clo{\cH}{G}) \subseteq H' g (g^{-1} H' g G) = H' g G \subseteq H' G'.
		\]
		As this holds for any $g \in G'$, it follows that $H' G' (\clo{\cH}{G}) \subseteq H' G'$. Since $H' G' G' = H' G'$, we therefore obtain that $H' \langle G', \clo{\cH}{G} \rangle = H' G'$, and so $\langle G', \clo{\cH}{G} \rangle \subseteq H' G'$. This holds for all $H' \in \cH$, and thus $\langle \clo{\cH}{G}, G' \rangle \subseteq \bigcap_{H' \in \cH} H' G'$. Since the group $\langle \clo{\cH}{G}, G' \rangle$ is an overgroup of $G'$, it follows that
		\[
			\langle \clo{\cH}{G}, G' \rangle \leq \clo{\cH}{G'},
		\]
		and so
		\[
			\clo{\cH}{G} \leq \langle \clo{\cH}{G}, G' \rangle \leq \clo{\cH}{G'},
		\]
		as desired.
	\end{proof}
	
	\begin{remark}
		The above result does not hold for $H$-closures; that is to say, with $G, G'$ and $H$ as above, we may have that $\clo{H}{G} \nleq \clo{H}{G'}$. This provides an additional justification for considering $\cH$-closures rather than $H$-closures throughout the remainder of the paper.
	\end{remark}
	
	The complementary statement for $\cH$ is as follows.
	
	\begin{lemma} \label{thm:closures:inclusion_preserving_in_H}
		Let $G$, $H$ and $H'$ be three subgroups of $\GL(\Zhat)$ such that $H \leq H'$. Then
		\[
			\clo{\cH}{G} \leq \clo{\cH'}{G}.
		\]
		In particular, if $G$ is $\cH'$-closed, it is also $\cH$-closed.
	\end{lemma}
	
	\begin{proof}
		Since $H \leq H'$, we have
		\[
			G \leq \clo{\cH}{G} \subseteq \bigcap_{g \in \GL(\Zhat)} H^{g} G \subseteq \bigcap_{g \in \GL(\Zhat)} H'^{g} G. 
		\]
		By the maximality of $\clo{\cH'}{G}$, it follows that $\clo{\cH}{G} \leq \clo{\cH'}{G}$.
		
		For the second part, note that if $G$ is $\cH'$-closed, then $\clo{\cH'}{G} = G$. Therefore, by the above, we obtain
		\[
			G \leq \clo{\cH}{G} \leq \clo{\cH'}{G} = G,
		\]
		and so $\clo{\cH}{G} = G$. Therefore, $G$ is also $\cH$-closed.
	\end{proof}
	
	Given knowledge of the subgroup $H$, we can construct explicit overgroups of $G$ which are $\cH$-equivalent to $G$, as the following result shows.
	
	\begin{lemma} \label{thm:closures:normal_product_equivalent}
		Let $G$ and $H$ be two subgroups of $\GL(\Zhat)$, and let $N$ be a normal subgroup of $\GL(\Zhat)$ such that $N \leq H$. Then $N G$ is a subgroup of $\GL(\Zhat)$ which is $\cH$-equivalent to $G$.
	\end{lemma}
	
	\begin{proof}
		As $N$ is a normal subgroup of $\GL(\Zhat)$, it follows that $N G$ is a subgroup of $\GL(\Zhat)$. Fix an element $g \in \GL(\Zhat)$. Then, we have that
		\[
			H g N G = H g N g^{-1} g G = H N g G = H g G,
		\]
		where the last equality stems from the fact that $N \leq H$. As the above holds for all $g \in \GL(\Zhat)$, it follows that $G$ and $N G$ are $\cH$-equivalent.
	\end{proof}
	
	While seemingly innocuous, this result has very strong implications for the structure of $\cH$-closed subgroups. For instance, if $H$ is an open subgroup of $\GL(\Zhat)$, then the $\cH$-closure of any subgroup is also open.
	
	\begin{corollary} \label{thm:closures:open_closures_preserve_level}
		Let $H$ be an open subgroup of $\GL(\Zhat)$ of level $n$, and let $G$ be a subgroup of $\GL(\Zhat)$. Then, the $\cH$-closure of $G$ is an open subgroup of level dividing $n$. In particular, there are finitely many $\cH$-closed subgroups of $\GL(\Zhat)$.
	\end{corollary}
	
	\begin{proof}
		Since the subgroup $H$ has level $n$, it contains the kernel $K = \ker(\pi_{n})$ of the reduction mod-$n$ map $\pi_{n} : \GL(\Zhat) \to \GL(\Z[n])$. Since $K$ is a normal subgroup, by Lemma \ref{thm:closures:normal_product_equivalent}, we have that $G$ and $K G$ are $\cH$-equivalent subgroups of $\GL(\Zhat)$. In particular, the $\cH$-closure of $G$ contains the product $KG$, and in particular, contains the kernel $K$ of $\pi_{n}$. Therefore, the $\cH$-closure of $G$ has level dividing $n$.
		
		The second statement follows from the fact that the subgroups of $\GL(\Zhat)$ of level dividing $n$ are in bijection with the subgroups of $\GL(\Z[n])$, of which there are finitely many.
	\end{proof}
	
	Similarly, given knowledge of the subgroup $G$, we can place restrictions on its $\cH$-closure $\clo{\cH}{G}$. For instance, by definition, the $\cH$-closure of an open subgroup $G$ of $\GL(\Zhat)$ contains $G$, and therefore is itself an open subgroup. Assuming some conditions on the subgroup $H$, we can obtain much stronger results, such as the following.
	
	\begin{lemma} \label{thm:closures:determinant_gives_sl2_level}
		Let $H$ be a subgroup of $\GL(\Zhat)$ such that $\det(H \cap \ker(\pi_{m})) = 1 + m\Zhat$ for all $m \geq 1$. Let $G$ be a subgroup of $\GL(\Zhat)$ with $\SL$-level $n$. Then the $\cH$-closure of $G$ is an open subgroup with $\GL$-level dividing $n$.
	\end{lemma}
	
	\begin{proof}
		Since $G$ has $\SL$-level $n$, $G$ contains the subgroup $K = \ker(\pi_{n}) \cap \SL(\Zhat) \leq \GL(\Zhat)$. Note that $K$ is a normal subgroup of $\GL(\Zhat)$, and so
		\[
		\bigcap_{g \in \GL(\Zhat)} H^{g} G = \bigcap_{g \in \GL(\Zhat)} H^{g} K G = \bigcap_{g \in \GL(\Zhat)} (H K)^{g} G.
		\]
		By assumption, we have $\det(H \cap \ker(\pi_{n})) = 1 + n \Zhat = \det(\ker(\pi_{n}))$. Therefore, we have
		\begin{align*}
			H K \cap \ker(\pi_{n}) &= (H \cap \ker(\pi_{n})) K = \ker(\pi_{n}).
		\end{align*}
		Therefore, we have
		\[
			\bigcap_{g \in \GL(\Zhat)} H^{g} G = \bigcap_{g \in \GL(\Zhat)} (H K)^{g} G \supseteq \bigcap_{g \in \GL(\Zhat)} (\ker(\pi_{n}))^{g} G = \ker(\pi_{n}) G.
		\]
		The product $\ker(\pi_{n}) G$ is a subgroup of $\GL(\Zhat)$, and so the $\cH$-closure of $G$ contains the product $\ker(\pi_{n}) G$. In particular, we have $\ker(\pi_{n}) \leq \clo{\cH}{G}$, and so the $\GL$-level of $\clo{\cH}{G}$ divides $n$.
	\end{proof}
	
	While all of the results above consider a single subgroup $H$, our main objective will be to determine the possible $\cB_{0}(n)$- or $\cB_{1}(n)$ closures of $\cG_{j}$, for all $n \geq 1$. The subgroups $B_{0}(n)$ or $B_{1}(n)$ are closely related as $n$ varies, and this relationship translates to the corresponding closures. More precisely, we have the following result. Recall that the group $H(m)$ is defined to be the product $H \ker \pi_{m}$, for any subgroup $H$ of $\GL(\Zhat)$.
	
	\begin{lemma} \label{thm:closures:family_closures}
		Fix $m \geq 1$. Let $H$ be a subgroup of $\GL(\Zhat)$, and let $k$ be the level of $H(m)$. Let $G$ be an open $\cH$-closed subgroup of $\GL(\Zhat)$ of level $n$. Then $G$ is $\cH(m)$-closed if and only if $n$ divides $k$.
	\end{lemma}
	
	\begin{proof}
		Suppose first that $G$ is $\cH(m)$-closed. Therefore, by Lemma \ref{thm:closures:open_closures_preserve_level}, the level of $\clo{\cH(m)}{G}$ divides $k$. Since $G$ is $\cH(m)$-closed, it follows that $G = \clo{\cH(m)}{G}$, and so $n$ divides $k$.
		
		Suppose now that $n$ divides $k$. Let $G'$ be a subgroup of $\GL(\Zhat)$ such that $G \leq G'$ and $G'$ is $\cH(m)$-equivalent to $G$. As $G$ and $G'$ are $\cH(m)$-equivalent, we have that
		\[
			H(m) g G = H(m) g G',
		\]
		for all $g \in \GL(\Zhat)$. Note that since $H(m)$ has level $k$, we have
		\[
			H(m) = H(m) \ker(\pi_{k}) = H \ker(\pi_{m}) \ker(\pi_{k}) = H \ker(\pi_{(m, k)}),
		\]
		where the last equality follows from \cite[Lemma 3.5]{terao2025isolated}. As $H(m)$ has level exactly $k$, we must have that $(m, k)$ is a multiple of $k$, and so $(m, k) = k$. Thus,
		\[
			H(m) = H \ker(\pi_{k}),
		\]
		and so
		\[
			H \ker(\pi_{k}) g G = H \ker(\pi_{k}) g G',
		\]
		for all $g \in \GL(\Zhat)$. As the kernel $\ker \pi_{k}$ is a normal subgroup of $\GL(\Zhat)$, it follows that
		\[
			H g \ker(\pi_{k}) G = H g \ker(\pi_{k}) G',
		\]
		for all $g \in \GL(\Zhat)$. Since $G$ has level $n$ dividing $k$, we have that $\ker(\pi_{k}) \leq G$ and $\ker(\pi_{k}) G = G$. Since $G \leq G'$, it follows that $\ker(\pi_{k}) \leq G'$, and similarly, $\ker(\pi_{k}) G' = G'$. Thus, we have
		\[
			H g G = H g G',
		\]
		for all $g \in \GL(\Zhat)$, and so $G$ and $G'$ are $\cH$-equivalent. Since $G$ is $\cH$-closed, it follows that $G = G'$. Therefore, $G$ is $\cH(m)$-closed, as desired.
	\end{proof}
	
	\pagebreak
	
	When defining the notions of $H$-equivalence and $H$-closure, we have so far considered the set of double cosets $H \backslash {\GL(\Zhat)} / G$ as the set of orbits of $H \backslash {\GL(\Zhat)}$ under right multiplication by $G$. However, we can equally consider the set of double cosets $H \backslash {\GL(\Zhat)} / G$ as the set of orbits of ${\GL(\Zhat)} / G$ under left multiplication by $H$. Under this duality, the notion of $H$-equivalence can still be interpreted as giving strong properties on the set of orbits, such as the following equality concerning the sizes of the respective orbits.
	
	\begin{lemma} \label{thm:closures:equivalent_degree_formula}
		Let $G$ be a subgroup of $\GL(\Zhat)$, and let $H$ and $H'$ be two open subgroups of $\GL(\Zhat)$ such that $H \leq H'$. Fix $g \in \GL(\Zhat)$. Then, $H$ and $H'$ are $g G g^{-1}$-equivalent if and only if
		\[
			|(H g) \cdot G| = [H' : H] |(H' g) \cdot G|.
		\]
	\end{lemma}
	
	\begin{proof}
		As $H$ and $H'$ are $g G g^{-1}$-equivalent, we have that the double cosets $(g G g^{-1}) H$ and $(g G g^{-1}) H'$ are equal. As left multiplication by $g \in \GL(\Zhat)$ is a bijection on $\GL(\Zhat)$, it follows that the double cosets $G g^{-1} H$ and $G g^{-1} H'$ are equal. Moreover, taking inverses, we obtain that the double cosets $H g G$ and $H' g G$ are also equal. Since $H \leq H'$, we have that each right coset of $H'$ is a union of $[H' : H]$ right cosets of $H$. Thus, we have
		\begin{align*}
			|(H g) \cdot G| &= |\{ H g' \in H \backslash {\GL(\Zhat)} : H g' \subset H g G \}| \\
			&= |\{ H g' \in H \backslash {\GL(\Zhat)} : H g' \subset H' g G \}| \\
			&= \sum_{H' g' \subset H' g G} |\{ H g'' \in H \backslash {\GL(\Zhat)} : H g'' \subset H' g' \}| \\
			&= \sum_{H' g' \subset H' g G} [H' : H] \\
			&= [H' : H] |(H' g) \cdot G|.
		\end{align*}
		The other direction is analogous.
	\end{proof}
	
	This result gives the connection, detailed in Section \ref{sec:introduction:connections}, between the notions of $H$-equivalence and that of primitive points on modular curves.
	
	\begin{corollary} \label{thm:closures:primitive_points_connection}
		Let $H$ and $H'$ be open subgroups of $\GL(\Zhat)$ such that $H \leq H'$, and let $f$ denote the inclusion morphism $f : X_{H} \to X_{H'}$. Let $j \in \Qbar$ be such that $j \notin \{0, 1728\}$, let $E / \Q(j)$ be an elliptic curve such that $j(E) = j$, and let $\alpha$ be a profinite level structure on $E$. Let $x \in X_{H}$ be the closed point corresponding to the $G_{\Q}$-orbit of the geometric point $[(E, [\alpha]_{H})] \in X_{H}(\Qbar)$. Then, we have that
		\[
			\deg(x) = \deg(f) \cdot \deg(f(x))
		\]
		if and only if $\pm H$ and $\pm H'$ are $G_{E, \alpha}$-equivalent. In particular, we have
		\[
			\ddeg_{X_{H}}(\overline{j}) = \ldblbrace d \cdot \deg(f) : d \in \ddeg_{X_{H'}}(\overline{j}) \rdblbrace
		\]
		if and only if $\pm H$ and $\pm H'$ are $\cG_{j}$-equivalent.
	\end{corollary}
	
	\begin{proof}
		By the proofs of Theorem \ref{thm:preliminaries:galois_action_bijection} and Corollary \ref{thm:preliminaries:modular_curve_point_degrees}, we have that
		\[
			\deg(x) = [\Q(j) : \Q] |(\pm H) \cdot G_{E, \alpha}|,
		\]
		and
		\[
			\deg(f(x)) = [\Q(j) : \Q] |(\pm H') \cdot G_{E, \alpha}|.
		\]
		By Lemma \ref{thm:closures:equivalent_degree_formula}, we have that $\pm H$ and $\pm H'$ are $G_{E, \alpha}$-equivalent if and only if
		\[
			|(\pm H) \cdot G_{E, \alpha}| = [\pm H' : \pm H] |(\pm H') \cdot G_{E, \alpha}|.
		\]
		Therefore, we obtain that $\pm H$ and $\pm H'$ are $G_{E, \alpha}$-equivalent if and only if
		\begin{align*}
			\deg(x) &= [\Q(j) : \Q] [\pm H' : \pm H] |(\pm H') \cdot G_{E, \alpha}| \\
			&= [\pm H' : \pm H] \deg(f(x)) \\
			&= \deg(f) \cdot \deg(f(x)).
		\end{align*}
		The second statement of the theorem follows from the fact that $\pm H$ and $\pm H'$ are $\cG_{j}$-equivalent if and only if they are $G$-equivalent for all $G \in \cG_{j}$; that is to say, they are $G_{E, \beta}$-equivalent for all profinite level structures $\beta$ on $E$.
	\end{proof}
	
	\section{Degrees of rational fibers on \texorpdfstring{$X_{0}(n)$}{X\_0(n)} and \texorpdfstring{$X_{1}(n)$}{X\_1(n)} which occur infinitely often} \label{sec:infinite_degrees}
	
	Now that the notion of $\cH$-closed subgroups has been established and thoroughly studied, we turn our attention to the main subject of this paper. In this section, we determine the degrees of the fibers with rational $j$-invariant on the modular curves $X_{0}(n)$ and $X_{1}(n)$ which occur infinitely often. As a consequence, we then obtain Theorems \ref{thm:introduction:infinite_x0_degrees} and \ref{thm:introduction:infinite_x1_degrees}.
	
	We divide the problem into three main steps. Firstly, building on Corollary \ref{thm:closures:modular_curve_point_degrees}, we show that it suffices to determine the conjugacy classes of subgroups which occur as the $\cB_{0}(n)$- or $\cB_{1}(n)$-closure of $\cG_{j}$ for infinitely many rational $j$-invariants $j \in \Q$. Secondly, we determine the set of conjugacy classes $\cG$ of $B_{0}(n)$- and $B_{1}(n)$-closed open subgroups of $\GL(\Zhat)$ such that the modular curve $X_{\cG}$ has infinitely many rational points. Finally, we determine which such conjugacy classes occur as the $\cB_{0}(n)$- or $\cB_{1}(n)$-closure of $\cG_{j}$ for infinitely many rational $j$-invariants $j \in \Q$, and compute the degrees of the corresponding fibers of $X_{0}(n)$ and $X_{1}(n)$.
	
	\subsection{Reducing to \texorpdfstring{$\cB_{0}(n)$}{B\_0(n)}- and \texorpdfstring{$\cB_{1}(n)$}{B\_1(n)}-closures which occur infinitely often} \label{sec:infinite_degrees:reduction_to_infinite_closures}
	
	As explained above, the first step is to show that it suffices to determine the conjugacy classes of subgroups of $\GL(\Zhat)$ which occur as the $\cB_{0}(n)$- or $\cB_{1}(n)$-closure of $\cG_{j}$ for infinitely many rational $j$-invariants $j \in \Q$. This is achieved by the following result.
	
	\begin{theorem} \label{thm:infinite_degrees:reduction_to_infinite_closures}
		Let $n \geq 1$. Then, we have
		\begin{align*}
			&\{M : \exists^{\infty} j \in \Q \text{ with } M = \ddeg_{X_{0}(n)}(j)\} \\
			& \quad = \{\ldblbrace \left| \Omega \right| : \Omega \in B_{0}(n) \backslash {\GL(\Zhat)} / \cG \rdblbrace : \exists^{\infty} j \in \Q \text{ with } \clo{\cB_{0}(n)}{\cG_{j}} = \cG\},
		\end{align*}
		and
		\begin{align*}
			&\{M : \exists^{\infty} j \in \Q \text{ with } M = \ddeg_{X_{1}(n)}(j)\} \\
			& \quad = \{\ldblbrace \left| \Omega \right| : \Omega \in B_{1}(n) \backslash {\GL(\Zhat)} / \cG \rdblbrace : \exists^{\infty} j \in \Q \text{ with } \clo{\cB_{1}(n)}{\cG_{j}} = \cG\}.
		\end{align*}
	\end{theorem}
	
	\begin{proof}
		We solely consider the case of $X_{0}(n)$, as the argument for $X_{1}(n)$ is entirely analogous.
		
		As the set $\{0, 1728\}$ is finite, it suffices to prove that
		\begin{align*}
			&\{M : \exists^{\infty} j \in \Q \setminus \{0, 1728\} \text{ with } M = \ddeg_{X_{0}(n)}(j)\} \\
			& \quad = \{\ldblbrace \left| \Omega \right| : \Omega \in B_{0}(n) \backslash {\GL(\Zhat)} / \cG \rdblbrace :  \exists^{\infty} j \in \Q \setminus \{0, 1728\} \text{ with } \clo{\cB_{0}(n)}{\cG_{j}} = \cG\}.
		\end{align*}
		Let $M$ be a multiset such that $M = \ddeg_{X_{0}(n)}(j)$ for infinitely many rational $j$-invariants $j \in \Q \setminus \{0, 1728\}$. By Corollary \ref{thm:closures:modular_curve_point_degrees}, it follows that
		\[
			M = \ldblbrace \left| \Omega \right| : \Omega \in B_{0}(n) \backslash {\GL(\Zhat)} / \clo{\cB_{0}(n)}{\cG_{j}} \rdblbrace,
		\]
		for infinitely many rational $j$-invariants $j \in \Q \setminus \{0, 1728\}$. For each such $j$-invariant $j$, the conjugacy class $\clo{\cB_{0}(n)}{\cG_{j}}$ is a conjugacy class of $\cB_{0}(n)$-closed subgroups of $\GL(\Zhat)$. By Lemma \ref{thm:closures:open_closures_preserve_level}, there are finitely many such subgroups. In particular, there is a conjugacy class $\cG$ of $\cB_{0}(n)$-closed subgroups of $\GL(\Zhat)$ such that there exist infinitely many rational $j$-invariants $j \in \Q \setminus \{0, 1728\}$ with $\cG = \clo{\cB_{0}(n)}{\cG_{j}}$, and
		\[
			M = \ldblbrace \left| \Omega \right| : \Omega \in B_{0}(n) \backslash {\GL(\Zhat)} / \cG \rdblbrace.
		\]
		On the other hand, if $M$ is a multiset such that a conjugacy class $\cG$ with the above properties exists, it is clear that $M = \ddeg_{X_{0}(n)}(j)$ for infinitely many rational $j$-invariants $j \in \Q \setminus \{0, 1728\}$. Therefore, we have
		\begin{align*}
			&\{M : \exists^{\infty} j \in \Q \setminus \{0, 1728\} \text{ with } M = \ddeg_{X_{0}(n)}(j)\} \\
			& \quad = \{\ldblbrace \left| \Omega \right| : \Omega \in B_{0}(n) \backslash {\GL(\Zhat)} / \cG \rdblbrace :  \exists^{\infty} j \in \Q \setminus \{0, 1728\} \text{ with } \clo{\cB_{0}(n)}{\cG_{j}} = \cG\},
		\end{align*}
		as desired.
	\end{proof}
	
	\subsection{Finding \texorpdfstring{$\cB_{0}(n)$}{B\_0(n)}- and \texorpdfstring{$\cB_{1}(n)$}{B\_1(n)}-closed subgroups with infinitely many rational points} \label{sec:infinite_degrees:infinite_closed_subgroups}
	
	In light of Theorem \ref{thm:infinite_degrees:reduction_to_infinite_closures}, we now strive to determine the conjugacy classes $\cG$ of subgroups of $\GL(\Zhat)$ for which there exists $\cH = \cB_{0}(n)$ or $\cB_{1}(n)$ for some $n \geq 1$ such that $\cG = \clo{\cH}{\cG_{j}}$ for infinitely many $j$-invariants $j \in \Q$.
	
	Let $\cG$ be such a conjugacy class. We note that, by definition, $\cG$ is a conjugacy class of $\cH$-closed subgroups of $\GL(\Zhat)$. Moreover, as there are finitely many rational CM $j$-invariants $j \in \Qcm$, there exists a non-CM $j$-invariant $j \in \Q$ such that $\cG = \clo{\cH}{\cG_{j}}$. Since $\cG_{j} \leq \clo{\cH}{\cG_{j}}$ and $\cG_{j}$ is a conjugacy class of open subgroups of $\GL(\Zhat)$, it follows that $\cG$ must also be a conjugacy class of open subgroups of $\GL(\Zhat)$. Finally, for any $j$-invariant $j \in \Q \setminus \{0, 1728\}$ such that $\cG = \clo{\cH}{\cG_{j}}$, we have
	\[
		\cG_{j} \leq \clo{\cH}{\cG_{j}} = \cG.
	\]
	Thus, by Corollary \ref{thm:preliminaries:modular_curve_rational_point}, there exists a rational point $x \in X_{\cG}(\Q)$ such that $j(x) = j$. In particular, the modular curve $X_{\cG}$ has infinitely many rational points.
	
	Therefore, as a first step, we aim to compute the set of conjugacy classes $\cG$ of open subgroups of $\GL(\Zhat)$ which are $\cB_{0}(n)$- or $\cB_{1}(n)$-closed for some $n \geq 1$, and such that the modular curve $X_{\cG}$ has infinitely many rational points.
	
	We begin with a few observations which greatly simplify this problem. Firstly, note that, by Lemmas \ref{thm:closures:inclusion_preserving_in_H} and \ref{thm:closures:family_closures}, the conjugacy class $\cG$ is $\cB_{0}(n)$-closed for some $n \geq 1$ if and only if $\cG$ is $\cB_{0}$-closed, and similarly for $\cB_{1}$. Therefore, the above reduces to finding the conjugacy classes $\cG$ of open subgroups of $\GL(\Zhat)$ which are $\cB_{0}$- or $\cB_{1}$-closed, and such that the modular curve $X_{\cG}$ has infinitely many rational points.
	
	Moreover, by Lemma \ref{thm:closures:determinant_gives_sl2_level}, if $\cG$ is $\cB_{0}$- or $\cB_{1}$-closed, then $\cG$ has $\GL$-level equal to its $\SL$-level. Since the modular curve $X_{\cG}$ is smooth, and geometrically disconnected if $\det(\cG) \neq \Zhat*$, it follows that if the modular curve $X_{\cG}$ has a rational point, then $\det(\cG) = \Zhat*$.
	
	By Faltings's theorem, we know that if the modular curve $X_{\cG}$ has infinitely many rational points, then the genus of $X_{\cG}$ is at most 1. Moreover, if $X_{\cG}$ has genus 0, then $X_{\cG}$ has infinitely many rational points if and only if $X_{\cG}(\Q) \neq \emptyset$. If $X_{\cG}$ has genus 1, then $X_{\cG}$ has infinitely many rational points if and only if $X_{\cG}$ has algebraic rank at least 1.
	
	Therefore, the problem reduces to finding the conjugacy classes $\cG$ of open subgroups of $\GL(\Zhat)$ such that:
	\begin{enumerate}
		\item $\det(\cG) = \Zhat*$,
		\item $X_{\cG}$ has genus $\leq 1$,
		\item $\cG$ has $\GL$-level equal to its $\SL$-level,
		\item $\cG$ is $\cB_{0}$- or $\cB_{1}$-closed,
		\item If $X_{\cG}$ has genus 0, then $X_{\cG}(\Q) \neq \emptyset$,
		\item If $X_{\cG}$ has genus 1, then $X_{\cG}$ has non-zero algebraic rank.
	\end{enumerate}
	
	The genus of the modular curve $X_{\cG}$ is determined solely by the intersection $\cG \cap \SL(\Zhat)$. Moreover, for each integer $g \geq 0$, there is a finite set $\Sigma_{g}$ of open subgroups of $\SL(\Zhat)$ such that the modular curve $X_{\cG}$ has genus $g$ if and only if the intersection $\cG \cap \SL(\Zhat)$ belongs to $\Sigma_{g}$. The sets $\Sigma_{g}$ have been computed explicitly by Cummins and Pauli \cite{cummins2003congruence}, for $g \leq 24$. In particular, we can compute the $\SL$-level of all subgroups in $\Sigma_{0}$ and $\Sigma_{1}$, and deduce that, if the modular curve $X_{\cG}$ has genus at most 1, then the $\SL$-level of $\cG$ is in the set
	\[
		S = \left\{\begin{array}{c}1, 2, 3, 4, 5, 6, 7, 8, 9, 10, 11, 12, 13, 14, 15, 16, 17, 18, 19, 20, \\
			21, 22, 24, 25, 26, 27, 28, 30, 32, 33, 36, 39, 40, 42, 48, 49, 52\end{array}\right\}.
	\]
	By the third condition above, it therefore suffices to restrict ourselves to conjugacy classes $\cG$ of open subgroups of $\GL(\Zhat)$ whose $\GL$-level is in $S$. Therefore, the problem can be reduced to finding the conjugacy classes $\cG$ of open subgroups of $\GL(\Zhat)$ satisfying the following five conditions:
	\begin{enumerate}
		\item $\det(\cG) = \Zhat*$, $\cG$ has $\GL$-level in $S$ and $X_{\cG}$ has genus $\leq 1$,
		\item $\cG$ has $\GL$-level equal to its $\SL$-level,
		\item $\cG$ is $\cB_{0}$- or $\cB_{1}$-closed,
		\item If $X_{\cG}$ has genus 0, then $X_{\cG}(\Q) \neq \emptyset$,
		\item If $X_{\cG}$ has genus 1, then $X_{\cG}$ has non-zero algebraic rank.
	\end{enumerate}
	This problem can now be easily solved using the data stored in the LMFDB \cite{lmfdb}. Indeed, at the time of writing, the LMFDB contains a list of all conjugacy classes $\cG$ of open subgroups of $\GL(\Zhat)$ of level at most 70 such that $\det(\cG) = \Zhat*$, as well as many of the properties of the associated modular curve $X_{\cG}$, such as its genus and level. Therefore, we can enumerate the finite list of conjugacy classes $\cG$ satisfying the first condition. Moreover, for each such conjugacy class $\cG$, it is straightforward to check whether it satisfies the other four conditions:
	\begin{description}
		\item[$\cG$ has $\GL$-level equal to its $\SL$-level] \hfill \\ The LMFDB stores the Cummins-Pauli label of the intersection $\cG \cap \SL(\Zhat)$, from which it is possible to extract the $\SL$-level of $\cG$.
		\item[$\cG$ is $\cB_{0}$- or $\cB_{1}$-closed] \hfill \\ By Theorem \ref{thm:closures:family_closures}, $\cG$ is $\cB_{0}$-closed if and only if it is $\cB_{0}(n)$-closed, where $n$ is the level of $\cG$. Using Remark \ref{rmk:closures:explicit_description} and the correspondence between group actions given in (\ref{eq:preliminaries:finite_group_action_correspondence}), it is straightforward to compute the $\cB_{0}(n)$-closure of $\cG$ from the action of $\GL(\Z[n])$ on the cosets of $B_{0}(n)$, which can be determined using \texttt{Magma}. One can then check whether $\cG$ is equal to its $\cB_{0}(n)$-closure, to determine if $\cG$ is $\cB_{0}$-closed. The same holds for $\cB_{1}$.
		\item[If $X_{\cG}$ has genus 0, then $X_{\cG}(\Q) \neq \emptyset$] \hfill \\ The LMFDB stores whether the modular curve $X_{\cG}$ has a rational point. Overall, this data is incomplete, however, it is complete for all subgroups $\cG$ which are required in our application.
		\item[If $X_{\cG}$ has genus 1, then $X_{\cG}$ has non-zero algebraic rank] \hfill \\ The LMFDB stores an upper bound for the analytic rank of the modular curve $X_{\cG}$. In particular, if this upper bound is zero, the analytic rank of $X_{\cG}$ is also zero, and so, by Kolyvagin \cite{kolyvagin1988finiteness}, the modular curve $X_{\cG}$ has (algebraic) rank zero. On the other hand, if the upper bound is nonzero, it becomes much harder to determine the rank of $X_{\cG}$. However, surprisingly, all subgroups $\cG$ in our application fall into the first case.
	\end{description}
	
	This procedure is implemented in the pair of files \texttt{infinite\_b0\_closed.m} and \texttt{infinite\_b1\_closed.m} found in the GitHub repository. Their output is summarized by the following result.
	
	\begin{theorem} \label{thm:infinite_degrees:infinite_closed_subgroups}
		Let $\cG$ be a conjugacy class of open $\cB_{0}(n)$-closed subgroups of $\GL(\Zhat)$, for some $n \geq 1$, such that the modular curve $X_{\cG}$ has infinitely many rational points. Then the conjugacy class $\cG$ is listed in Table \ref{tbl:infinite_b_closures:infinite_b0_closures}.
		
		Similarly, if $\cG$ is a conjugacy class of $\cB_{1}(n)$-closed subgroups satisfying the same conditions, then $\cG$ is listed in Table \ref{tbl:infinite_b_closures:infinite_b1_closures}. Moreover, both Tables \ref{tbl:infinite_b_closures:infinite_b0_closures} and \ref{tbl:infinite_b_closures:infinite_b1_closures} are minimal.
	\end{theorem}
	
	\begin{remark} \label{rmk:infinite_degrees:all_have_genus_0}
		As alluded to above, one surprising observation is that, for all $\cB_{0}$- and $\cB_{1}$-closed conjugacy classes $\cG$ such that the modular curve $X_{\cG}$ has genus 1, the analytic, and hence algebraic, rank of $X_{\cG}$ is zero. Therefore, the modular curve $X_{\cG}$ has genus 0, for all $\cG$ appearing in Tables \ref{tbl:infinite_b_closures:infinite_b0_closures} and \ref{tbl:infinite_b_closures:infinite_b1_closures}. This is quite unexpected, and we have no explanation for this phenomenon. Nevertheless, this coincidence will prove to be very useful later.
	\end{remark}
	
	\subsection{Computing degrees of rational fibers on \texorpdfstring{$X_{0}(n)$}{X\_0(n)} and \texorpdfstring{$X_{1}(n)$}{X\_1(n)} which occur infinitely often} \label{sec:infinite_degrees:computing_degrees}
	
	In Section \ref{sec:infinite_degrees:reduction_to_infinite_closures}, we showed that it the problem of determining the degrees of the rational fibers on the modular curves $X_{0}(n)$ and $X_{1}(n)$ which occur infinitely often can be distilled to understanding the conjugacy classes of subgroups which occur as $\cB_{0}(n)$- and $\cB_{1}(n)$-closures of $\cG_{j}$ for infinitely many $j \in \Q$. Moreover, in Section \ref{sec:infinite_degrees:infinite_closed_subgroups}, we computed a finite set of conjugacy classes of subgroups of $\GL(\Zhat)$, given in Tables \ref{tbl:infinite_b_closures:infinite_b0_closures} and \ref{tbl:infinite_b_closures:infinite_b1_closures}, containing these latter conjugacy classes. Therefore, it remains to determine which of the conjugacy classes in Tables \ref{tbl:infinite_b_closures:infinite_b0_closures} and \ref{tbl:infinite_b_closures:infinite_b1_closures} do occur as the $\cB_{0}(n)$- or $\cB_{1}(n)$-closures of $\cG_{j}$ for infinitely many $j \in \Q$.
	
	To do so, we rely heavily on the following consequence of Hilbert's irreducibility theorem.
	
	\begin{theorem} \label{thm:infinite_degrees:genus_0_occurs_infinitely}
		Let $\cG$ be a conjugacy class of open subgroups of $\GL(\Zhat)$ of level $n$, containing $-I$, such that the modular curve $X_{\cG}$ has genus 0 and infinitely many rational points. Let $m$ be a multiple of $n$. Then, there are infinitely many rational $j$-invariants $j \in \Q$ such that $\cG_{j}(m) = \cG$.
	\end{theorem}
	
	\begin{proof}
		Consider a $j$-invariant $j \in j(X_{\cG}(\Q))$. By Corollary \ref{thm:preliminaries:modular_curve_rational_point}, we know that $\cG_{j} \leq \cG$, and so
		\[
			\cG_{j}(m) \leq \cG(m) = \cG,
		\]
		where the last equality uses the fact that $\cG$ is a conjugacy class of open subgroups of $\GL(\Zhat)$ of level dividing $m$.
		
		Suppose now that $\cG_{j}(m) \neq \cG$. As $\cG_{j} \leq \cG_{j}(m)$, it follows by Corollary \ref{thm:preliminaries:modular_curve_rational_point} that there exists a rational point $x \in X_{\cG_{j}(m)}(\Q)$ such that $j(x) = j$. In particular, since $\cG_{j}(m) \lneq \cG$, it follows that
		\[
			j \in \bigcup_{\cH \in \Sigma} j(X_{\cH}(\Q)),
		\]
		where $\Sigma$ is the set of conjugacy classes $\cH$ of open subgroups of $\GL(\Zhat)$ of level dividing $m$ and containing $-I$ such that $\cH \lneq \cG$.
		
		Fix a conjugacy class $\cH \in \Sigma$. As $\cH \leq \cG$, we obtain an inclusion morphism $f_{\cH} : X_{\cH} \to X_{\cG}$. Note that there may be multiple such morphisms $f_{\cH}$; the precise choice does not matter. The $j$-map $j : X_{\cH} \to X(1)$ factors as $X_{\cH} \xrightarrow{f_{\cH}} X_{\cG} \xrightarrow{j} X(1)$, and so we obtain that
		\[
			j \in \bigcup_{\cH \in \Sigma} (j \circ f_{\cH})(X_{\cH}(\Q)) = j \! \left( \bigcup_{\cH \in \Sigma} f_{\cH}(X_{\cH}(\Q)) \right).
		\]
		Since $\cH \neq \cG$, we have that
		\[
			\deg(f_{\cH}) = [\pm \cG : \pm \cH] = [\cG : \cH] > 1.
		\]
		In particular, the image $f_{\cH}(X_{\cH}(\Q))$ is a thin subset of $X_{\cG}(\Q)$, as defined in \cite[Definition 3.1.1]{serre1992topics}.
		
		The set $\Sigma$ is a subset of the set of open subgroups of $\GL(\Zhat)$ of level dividing $m$. The latter is in bijection with the set of subgroups of $\GL(\Z[m])$; in particular, this set is finite. Therefore, the set $\Sigma$ is finite, and the union
		\[
			A = \bigcup_{\cH \in \Sigma} f_{\cH}(X_{\cH}(\Q))
		\]
		is a thin subset of $X_{\cG}(\Q)$.
		
		As the modular curve $X_{\cG}$ has genus 0 and a rational point, we have that $X_{\cG} \cong \PP^{1}_{\Q}$. In particular, ordering by naive height, it is a classical result of Mertens and Ces\`aro that the number of points of $X_{\cG}(\Q)$ of height at most $N$ is asymptotically $\frac{12}{\pi^{2}} N^{2}$ as $N \to \infty$. By \cite[Proposition 3.4.2]{serre1992topics}, since $A$ is a thin subset of $X_{\cG}(\Q)$, the number of points of $A$ of height at most $N$ is $O(N)$ as $N \to \infty$.
		
		Denote by $H_{\cG}$ and $H$ the heights on $X_{\cG}$ and $X(1)$ respectively. The map $j : X_{\cG} \to X(1)$ is finite, and so, by \cite[Section VIII, Theorem 5.6]{silverman2009arithmetic}, we know that
		\[
			H \circ j \asymp (H_{\cG})^{n},
		\]
		where $n$ is the degree of $j$. Therefore, the number of points of $j(X_{\cG}(\Q))$ of height (on $X(1)$) at most $N$ is $\Theta(N^{\frac{2}{n}})$ as $N \to \infty$. Similarly, the number of points of $j(A)$ of height at most $N$ is $O(N^{\frac{1}{n}})$ as $N \to \infty$. It follows that there are infinitely many $j \in j(X_{\cG}(\Q))$ such that $j \notin j(A)$. By the argument above, for each such $j$, we have that $\cG_{j}(m) = \cG$, and so the result follows.
	\end{proof}
	
	\begin{remark}
		While the above theorem holds for any multiple $m$ of the level $n$ of $X_{\cH}$, the conclusion of the theorem cannot be replaced by the statement $\cG_{j} = \cH$. In fact, it is possible that the modular curve $X_{\cH}$ has genus 0 and infinitely many rational points, but there are no $j$-invariants $j \in \Q$ such that $\cG_{j} = \cH$. For instance, the modular curve $X_{\GL(\Zhat)} = X(1)$ has genus 0 and infinitely many rational points, however, there are no $j$-invariants $j \in \Q$ such that $\cG_{j} = \GL(\Zhat)$, by \cite[Proposition 22]{serre1972proprietes}.
		
		The assumption that the modular curve has genus 0 is also necessary, see \cite[Example 6.1]{rouse2015elliptic}.
	\end{remark}
	
	As a consequence of the above result, we can settle the question posed in at the start of this section for genus 0 curves. We note that, by the above remark, the case for genus 1 modular curves is much more delicate. This illustrates the usefulness of knowing that the modular curves given in Table \ref{tbl:infinite_b_closures:infinite_b0_closures} and \ref{tbl:infinite_b_closures:infinite_b1_closures} are of genus 0.
	
	\begin{corollary} \label{thm:infinite_degrees:genus_0_closed_occurs_infinitely}
		Let $H$ be a subgroup of $\GL(\Zhat)$, and let $\cG$ be an $\cH$-closed conjugacy class of open subgroups of $\GL(\Zhat)$ of level $n$, such that the modular curve $X_{\cG}$ has genus 0 and infinitely many rational points. Let $m \geq 1$ be such that the level of $H(m)$ is a multiple of $n$. Then, there are infinitely many $j$-invariants $j \in \Q$ such that the $\cH(m)$-closure $\clo{\cH(m)}{\cG_{j}}$ is equal to $\cG$.
	\end{corollary}
	
	\begin{proof}
		By Theorem \ref{thm:infinite_degrees:genus_0_occurs_infinitely}, there exist infinitely many $j$-invariants $j \in \Q$ such that $\cG_{j}(m) = \cG$. By definition, we know that the kernel $\ker(\pi_{m})$ of the reduction mod-$m$ map is a subgroup of $H(m)$. The former is a normal subgroup of $\GL(\Zhat)$, and so, by Lemma \ref{thm:closures:normal_product_equivalent}, the subgroups $\cG_{j}$ and $\cG_{j}(m)$ are $\cH(m)$-equivalent. In particular, we have
		\[
			\clo{\cH(m)}{\cG_{j}} = \clo{\cH(m)}{\cG_{j}(m)}.
		\]
		The level $n$ of $\cG$ divides the level of $H(m)$ by assumption, and so, by Lemma \ref{thm:closures:family_closures}, the conjugacy class $\cG$ is $\cH(m)$-closed. In particular, we have
		\[
			\clo{\cH(m)}{\cG_{j}} = \clo{\cH(m)}{\cG_{j}(m)} = \clo{\cH(m)}{\cG} = \cG,
		\]
		as required.
	\end{proof}
	
	Equipped with these results, and leveraging Lemma \ref{thm:closures:equivalent_degree_formula}, we can now give a complete classification of the degrees of the fibers with rational $j$-invariant on the modular curves $X_{0}(n)$ and $X_{1}(n)$ which occur infinitely often.
	
	\begin{theorem} \label{thm:infinite_degrees:infinite_fiber_degrees}
		Fix $n \geq 1$. Then, for any conjugacy class $\cG$ of open subgroups of $\GL(\Zhat)$, there exist infinitely many $j$-invariants $j \in \Q$ such that $\clo{\cB_{0}(n)}{\cG_{j}} = \cG$ if and only if $\cG$ is listed in Table \ref{tbl:infinite_b_closures:infinite_b0_closures}, and the level of $\cG$ divides $n$. Similarly, there exist infinitely many $j$-invariants $j \in \Q$ such that $\clo{\cB_{1}(n)}{\cG_{j}} = \cG$ if and only if $\cG$ is listed in Table \ref{tbl:infinite_b_closures:infinite_b1_closures}, and the level of $\cG$ divides $n$.
		
		In particular, we have
		\begin{align*}
			&\{M : \exists^{\infty} j \in \Q \setminus \Qcm \text{ with } M = \ddeg_{X_{0}(n)}(j)\} \\
			& = \{\ldblbrace [B_{0}(m) : B_{0}(n)] |\Omega| : \Omega \in B_{0}(m) \backslash {\GL(\Zhat)} / \cG \rdblbrace : \cG \text{ in Table \ref{tbl:infinite_b_closures:infinite_b0_closures} of level } m \mid n\},
		\end{align*}
		and
		\begin{align*}
			&\{M : \exists^{\infty} j \in \Q \setminus \Qcm \text{ with } M = \ddeg_{X_{1}(n)}(j)\} \\
			& = \{\ldblbrace [B_{1}(m) : B_{1}(n)] |\Omega| : \Omega \in B_{1}(m) \backslash {\GL(\Zhat)} / \cG \rdblbrace : \cG \text{ in Table \ref{tbl:infinite_b_closures:infinite_b1_closures} of level } m \mid n\}.
		\end{align*}
	\end{theorem}
	
	\begin{proof}
		Let $\cG$ be a conjugacy class of $\cB_{0}(n)$-closed subgroups of $\GL(\Zhat)$ such that there exists infinitely many $j$-invariants $j \in \Q$ such that $\clo{\cB_{0}(n)}{\cG_{j}} = \cG$. By the argument at the start of Section \ref{sec:infinite_degrees:infinite_closed_subgroups}, $\cG$ is a conjugacy class of open $\cB_{0}(n)$-closed subgroups of $\GL(\Zhat)$ such that the modular curve $X_{\cG}$ has infinitely many rational points. In particular, by Theorem \ref{thm:infinite_degrees:infinite_closed_subgroups}, the conjugacy class $\cG$ is listed in Table \ref{tbl:infinite_b_closures:infinite_b0_closures}. Note that, as $\cG$ is $\cB_{0}(n)$-closed, by Lemma \ref{thm:closures:open_closures_preserve_level}, the level of $\cG$ must divide $n$.
		
		On the other hand, let $\cG$ be a conjugacy class of subgroups of $\GL(\Zhat)$ listed in Table \ref{tbl:infinite_b_closures:infinite_b0_closures} such that the level of $\cG$ divides $n$. In particular, $\cG$ is a conjugacy class of $\cB_{0}$-closed subgroups of $\GL(\Zhat)$ such that the modular curve $X_{\cG}$ has infinitely many rational points. By Remark \ref{rmk:infinite_degrees:all_have_genus_0}, the modular curve $X_{\cG}$ has genus 0. Therefore, by Corollary \ref{thm:infinite_degrees:genus_0_closed_occurs_infinitely}, there are infinitely many $j$-invariants $j \in \Q$ such that $\clo{\cB(n)}{\cG_{j}} = \cG$. This completes the proof of the first assertion.
		
		By Theorem \ref{thm:infinite_degrees:reduction_to_infinite_closures}, we have that
		\begin{align*}
			&\{M : \exists^{\infty} j \in \Q \setminus \Qcm \text{ with } M = \ddeg_{X_{0}(n)}(j)\} \\
			& \quad = \{\ldblbrace \left| \Omega \right| : \Omega \in B_{0}(n) \backslash {\GL(\Zhat)} / \cG \rdblbrace : \exists^{\infty} j \in \Q \text{ with } \clo{\cB_{0}(n)}{\cG_{j}} = \cG\}.
		\end{align*}
		By the above, we thus obtain that
		\begin{align*}
			&\{M : \exists^{\infty} j \in \Q \setminus \Qcm \text{ with } M = \ddeg_{X_{0}(n)}(j)\} \\
			& \quad = \{\ldblbrace \left| \Omega \right| : \Omega \in B_{0}(n) \backslash {\GL(\Zhat)} / \cG \rdblbrace : \cG \text{ in Table \ref{tbl:infinite_b_closures:infinite_b0_closures} of level } m \mid n\}.
		\end{align*}
		Let $\cG$ be a conjugacy class of open subgroups of $\GL(\Zhat)$ listed in Table \ref{tbl:infinite_b_closures:infinite_b0_closures} of level $m$ dividing $n$, and let $G \in \cG$ be a representative. By definition, we have that the kernel $\ker (\pi_{m})$ of the reduction mod-$m$ map is contained in $G$. By Lemma \ref{thm:closures:normal_product_equivalent}, it follows that the subgroups $B_{0}(n)$ and $(\ker \pi_{m})B_{0}(n) = B_{0}(m)$ are $\cG$-equivalent. Therefore, by Lemma \ref{thm:closures:equivalent_degree_formula}, we have that
		\begin{align*}
			& \ldblbrace \left| \Omega \right| : \Omega \in B_{0}(n) \backslash {\GL(\Zhat)} / \cG \rdblbrace \\
			& = \ldblbrace \left| (B_{0}(n) \, g) \cdot G \right| : (B_{0}(n) \, g) \cdot G \in B_{0}(n) \backslash {\GL(\Zhat)} / G \rdblbrace \\
			& = \ldblbrace [B_{0}(m) : B_{0}(n)] \left| (B_{0}(m) \, g) \cdot G \right| : (B_{0}(m) \, g) \cdot G \in B_{0}(m) \backslash {\GL(\Zhat)} / G \rdblbrace \\
			& = \ldblbrace [B_{0}(m) : B_{0}(n)] \left| \Omega \right| : \Omega \in B_{0}(m) \backslash {\GL(\Zhat)} / \cG \rdblbrace,
		\end{align*}
		and the second result follows. The argument for $X_{1}(n)$ is analogous.
	\end{proof}
	
	From the data in Tables \ref{tbl:infinite_b_closures:infinite_b0_closures} and \ref{tbl:infinite_b_closures:infinite_b1_closures}, and explicit formulas for the indices $[B_{0}(m) : B_{0}(n)]$ and $[B_{1}(m) : B_{1}(n)]$ given for instance in \cite[Section 3.9]{diamond2005first}, one obtains a simple procedure for computing the sets given in Theorem \ref{thm:infinite_degrees:infinite_fiber_degrees}. This procedure is implemented in the functions \texttt{PossibleX0Fibers} and \texttt{PossibleX1Fibers} defined in the file \texttt{degrees\_x0\_x1.m}.
	
	In addition, it is straightforward to deduce the degrees of the points on $X_{0}(n)$ and $X_{1}(n)$ with rational $j$-invariant which occur infinitely often from the degrees of the fibers given in Theorem \ref{thm:infinite_degrees:infinite_fiber_degrees}. This yields the statements of Theorems \ref{thm:introduction:infinite_x0_degrees} and \ref{thm:introduction:infinite_x1_degrees}.
	
	\section{Degrees of rational fibers on \texorpdfstring{$X_{0}(n)$}{X\_0(n)} and \texorpdfstring{$X_{1}(n)$}{X\_1(n)} which occur finitely often} \label{sec:finite_degrees}
	
	We now turn our attention to determining the degrees of the fibers with rational $j$-invariant on the modular curves $X_{0}(n)$ and $X_{1}(n)$ which occur finitely often. As before, we restrict our consideration to the fibers lying above non-CM $j$-invariants. The structure of the argument closely mirrors the one used in the infinite case. In particular, the core of the proof is to determine the conjugacy classes $\cG$ of open subgroups of $\GL(\Zhat)$ such that $\cG = \clo{\cB_{0}(n)}{\cG_{j}}$ or $\cG = \clo{\cB_{1}(n)}{\cG_{j}}$ for a non-zero, finite number of non-CM rational $j$-invariants $j \in \Q$.
	
	\subsection{The intersections \texorpdfstring{$\cG_{j} \cap \SL(\Zhat)$}{of G\_j with SL\_2(Zhat)}} \label{sec:finite_degrees:sl_intersections}
	
	While the corresponding classification in the infinite case was mostly straightforward thanks to the data stored in the Cummins-Pauli database and the LMFDB, the finite version is much harder to tackle unconditionally. Instead, we will rely on a conjectural classification by Zywina of the intersections $\cG_{j} \cap \SL(\Zhat)$, for $j$ a non-CM rational $j$-invariant. This classification relies on the notion of the agreeable closure of an open subgroup of $\GL(\Zhat)$, which we recall first.
	
	\begin{definition}[{\cite[Section 8.3]{zywina2024explicit}, \cite[Section 4.1]{zywina2025open}}]
		Let $G$ be an open subgroup of $\GL(\Zhat)$, and denote by $[G, G] \leq \SL(\Zhat)$ its commutator subgroup. Let $N$ be the product of primes dividing the $\SL$-level of $[G, G]$, and recall that $\GL(\Zhat) \cong \prod_{p} \GL(\Z_{p})$. The \textbf{agreeable closure of $G$} is defined to be the subgroup
		\[
			G_{\agr} = \Zhat* G \left( \prod_{p \nmid N} \GL(\Z_{p}) \right),
		\]
		where $\Zhat*$ denotes the set of scalar matrices of $\GL(\Zhat)$.
	\end{definition}
	
	\begin{remark}
		Since both $\Zhat*$ and $\prod_{p \nmid N} \GL(\Z_{p})$ are normal subgroups of $\GL(\Zhat)$, the agreeable closure $G_{\agr}$ is indeed a subgroup of $\GL(\Zhat)$. Moreover, for any $g \in \GL(\Zhat)$, we have that $(G^{g})_{\agr} = G_{\agr} {}^{g}$. Therefore, given a conjugacy class $\cG$ of open subgroups of $\GL(\Zhat)$, we may define the agreeable closure $\cG_{\agr}$ of $\cG$ to be the conjugacy class of the agreeable closure $G_{\agr}$, for any representative $G \in \cG$.
	\end{remark}
	
	The link between agreeable closures and the intersection $\cG_{j} \cap \SL(\Zhat)$ is illuminated by the following two lemmas.
	
	\begin{lemma}[{\cite[Lemma 1.7]{zywina2024explicit}}]
		Let $E$ be an elliptic curve defined over $\Q$. Then, we have
		\[
			\cG_{E} \cap \SL(\Zhat) = [\cG_{E}, \cG_{E}].
		\]
	\end{lemma}
	
	\begin{lemma}[{\cite[Proposition 8.1]{zywina2024explicit}, \cite[Lemma 4.1]{zywina2025open}}]
		Let $G$ be an open subgroup of $\GL(\Zhat)$. Then,
		\[
			[G, G] = [G_{\agr}, G_{\agr}].
		\]		
	\end{lemma}
	
	Combining both results shows that, for any elliptic curve $E / \Q$, the intersection $\cG_{E} \cap \SL(\Zhat)$ is equal to the commutator subgroup $[(\cG_{E})_{\agr}, (\cG_{E})_{\agr}]$. Therefore, in order to determine the possible intersections $\cG_{E} \cap \SL(\Zhat)$, as $E$ varies through all non-CM elliptic curves over $\Q$, it suffices to classify the possible agreeable closures $(\cG_{E})_{\agr}$, for the same elliptic curves $E$. This is precisely the work undertaken by Zywina in \cite{zywina2024explicit}, which gives a conjectural classification for such agreeable closures. We summarize this classification below.
	
	Note that, for non-CM elliptic curves $E / \Q$ with $j$-invariant $j$, we have that $\cG_{j} = \pm \cG_{E}$. Therefore, while the work of Zywina gives a classification of the intersections $\cG_{E} \cap \SL(\Zhat)$, it is straightforward to convert this to a classification of the intersections $\cG_{j} \cap \SL(\Zhat)$. We present the latter, as it will be easier to apply in later sections.
	
	\begin{conjecture}[{\cite[Section 14]{zywina2024explicit}}] \label{conj:finite_degrees:sl_intersections}
		Let $j \in \Q$ be a non-CM rational $j$-invariant. If the modular curve $X_{(\cG_{j})_{\agr}}$ has infinitely many rational points, then the intersection $\cG_{j} \cap \SL(\Zhat)$ is one of the conjugacy classes given in Table \ref{tbl:finite_b_closures:infinite_sl_intersections_zywina}. On the other hand, if the modular curve $X_{(\cG_{j})_{\agr}}$ has finitely many rational points, then $j$ and $\cG_{j}$ are given in Table \ref{tbl:finite_b_closures:exceptional_j_invariants}.
	\end{conjecture}
	
	\begin{remark}
		While we have stated this as a single conjecture, the validity of this result is dependent on two distinct conjectures. The first is Serre's uniformity conjecture \cite[Conjecture 1.2]{zywina2024explicit}, while the second is a conjecture on the set of rational points on a number of modular curves of high genus. We note that the first statement in the conjecture is solely dependent on Serre's uniformity conjecture, while the second statement is dependent on both. For more information, we refer the reader to \cite{zywina2024explicit}.
	\end{remark}
	
	\subsection{Agreeable closures with infinitely many rational points} \label{sec:finite_degrees:infinite_agreeable}
	
	Given this conjecture on the intersections $\cG_{j} \cap \SL(\Zhat)$, we can now turn our attention towards determining the $\cB_{0}(n)$- and $\cB_{1}(n)$-closures of the extended adelic Galois image $\cG_{j}$, for all non-CM $j$-invariants $j \in \Q$. To do so, we split into the two cases given in Conjecture \ref{conj:finite_degrees:sl_intersections}, namely whether the modular curve $X_{(\cG_{j})_{\agr}}$ contains infinitely many rational points.
	
	In this section, we tackle the case where the modular curve $X_{(\cG_{j})_{\agr}}$ contains infinitely many rational points. Throughout, we let $j \in \Q$ be a non-CM $j$-invariant, and suppose that the modular curve $X_{(\cG_{j})_{\agr}}$ has infinitely many rational points. Our goal is now to classify the possible $\cB_{0}(n)$- and $\cB_{1}(n)$-closures of the extended adelic Galois image $\cG_{j}$. 
	
	The case of $\cB_{0}(n)$-closures is much easier than the case of $\cB_{1}(n)$-closures, as the $\cB_{0}(n)$-closure must contain the agreeable closure, as the following result shows.
	
	\begin{lemma} \label{thm:finite_degrees:b0_contains_agreeable}
		Let $G$ be an open subgroup of $\GL(\Zhat)$. Then, we have
		\[
			G_{\agr} \leq \clo{\cB_{0}}{G}.
		\]
	\end{lemma}
	
	\begin{proof}
		Let $n$ be the $\SL$-level of $G$. By Lemma \ref{thm:closures:determinant_gives_sl2_level}, the $\cB_{0}$-closure $\clo{\cB_{0}}{G}$ is an open subgroup of $\GL$-level dividing $n$. In particular, we have that
		\[
			\ker(\pi_{n}) \leq \clo{\cB_{0}}{G}.
		\]
		Let $N$ be the product of the primes dividing the $\SL$-level of the commutator subgroup $[G, G]$. Since $[G, G] \leq G$, it follows that each prime dividing $n$ also divides $N$. In particular, we have
		\[
			\prod_{p \nmid N} \GL(\Z_{p}) \leq \prod_{p \nmid n} \GL(\Z_{p}) \leq \ker(\pi_{n}) \leq \clo{\cB_{0}}{G}.
		\]
		Denote by $\Zhat*$ the set of scalar matrices of $\GL(\Zhat)$, and note that $\Zhat*$ is a normal subgroup of $\GL(\Zhat)$. Since $\Zhat* \leq \cB_{0}$, it follows by Lemma \ref{thm:closures:normal_product_equivalent} that the groups $G$ and $\Zhat* G$ are $\cB_{0}$-equivalent. In particular, we have $\clo{\cB_{0}}{G} = \clo{\cB_{0}}{(\Zhat* G)}$, and so $\Zhat* G \leq \clo{\cB_{0}}{G}$. Thus, we obtain that
		\[
			G_{\agr} = \Zhat* G \left( \prod_{p \nmid N} \GL(\Z_{p}) \right) \leq \clo{\cB_{0}}{G},
		\]
		as required.
	\end{proof}
	
	As a corollary, we obtain that the $\cB_{0}(n)$-closure of $\cG_{j}$ must be one of the conjugacy classes computed in Section \ref{sec:infinite_degrees:infinite_closed_subgroups}.
	
	\begin{corollary} \label{thm:finite_degrees:b0_closure_infinite_agreeable}
		Let $j \in \Q$ be a non-CM $j$-invariant, and suppose that the modular curve $X_{(\cG_{j})_{\agr}}$ has infinitely many rational points. Then the conjugacy class $\clo{\cB_{0}(n)}{\cG_{j}}$ is one of the conjugacy classes given in Table \ref{tbl:infinite_b_closures:infinite_b0_closures}, for all $n \geq 1$.
	\end{corollary}
	
	\begin{proof}
		Fix $n \geq 1$. By Lemma \ref{thm:finite_degrees:b0_contains_agreeable}, we have that $(\cG_{j})_{\agr} \leq \clo{\cB_{0}}{\cG_{j}}$. Therefore, by Lemma \ref{thm:closures:inclusion_preserving_in_H}, we have that
		\[
			(\cG_{j})_{\agr} \leq \clo{\cB_{0}}{\cG_{j}} \leq \clo{\cB_{0}(n)}{\cG_{j}}.
		\]
		Thus, we obtain an inclusion morphism of modular curves $X_{(\cG_{j})_{\agr}} \to X_{\clo{\cB_{0}(n)}{\cG_{j}}}$. As the modular curve $X_{(\cG_{j})_{\agr}}$ has infinitely many rational points, it follows that the modular curve $X_{\clo{\cB_{0}(n)}{\cG_{j}}}$ also has infinitely many rational points. Therefore, we obtain that $\clo{\cB_{0}(n)}{\cG_{j}}$ is a conjugacy class of open $\cB_{0}(n)$-closed subgroups of $\GL(\Zhat)$ such that the modular curve $X_{\clo{\cB_{0}(n)}{\cG_{j}}}$ has infinitely many rational points. By Theorem \ref{thm:infinite_degrees:infinite_closed_subgroups}, it follows that the conjugacy class $\clo{\cB_{0}(n)}{\cG_{j}}$ is listed in Table \ref{tbl:infinite_b_closures:infinite_b0_closures}.
	\end{proof}
	
	The case of $\cB_{1}(n)$-closures is much more difficult, as the $\cB_{1}(n)$-closure of $\cG_{j}$ does not necessarily contain the agreeable closure, since $B_{1}(n)$ does not contain all of the scalar matrices $\Zhat*$. Therefore, it is possible for the modular curve $X_{\clo{\cB_{1}}{\cG_{j}}}$ to contain finitely many rational points while the modular curve $X_{(\cG_{j})_{\agr}}$ contains infinitely many.

	Instead, we aim to utilize the knowledge of $\cG_{j} \cap \SL(\Zhat)$ given in Conjecture \ref{conj:finite_degrees:sl_intersections} to classify the possible $\cB_{1}(n)$-closures $\clo{\cB_{1}(n)}{\cG_{j}}$. We obtain the following result, whose proof will occupy the remainder of this section.
	
	\begin{theorem} \label{thm:finite_degrees:b1_closure_infinite_agreeable}
		Fix $n \geq 1$. Let $j \in \Q$ be a non-CM $j$-invariant such that the intersection $\cG_{j} \cap \SL(\Zhat)$ is one of the conjugacy classes listed in Table \ref{tbl:finite_b_closures:infinite_sl_intersections_zywina}. Then, one of the following holds:
		\begin{itemize}
			\item The $\cB_{1}(n)$-closure of the conjugacy class $\cG_{j}$ is given in Table \ref{tbl:infinite_b_closures:infinite_b1_closures}.
			\item The $j$-invariant $j$ and the conjugacy class $\cG_{j}$ are given in one of the unshaded rows of Table \ref{tbl:finite_b_closures:exceptional_j_infinite_sl}.
		\end{itemize}
	\end{theorem}
	
	\begin{proof}
		By Lemma \ref{thm:closures:inclusion_preserving_in_H}, we know that $\clo{\cB_{1}(n)}{\cG_{j}}$ is a conjugacy class of $\cB_{1}$-closed subgroups of $\GL(\Zhat)$, and that $\clo{\cB_{1}}{\cG_{j}} \leq \clo{\cB_{1}(n)}{\cG_{j}}$. In addition, by Lemma \ref{thm:closures:inclusion_preserving_in_G}, we have
		\[
			\clo{\cB_{1}}{(\cG_{j} \cap \SL(\Zhat))} \leq \clo{\cB_{1}}{\cG_{j}} \leq \clo{\cB_{1}(n)}{\cG_{j}}.
		\]
		By Lemma \ref{thm:closures:determinant_gives_sl2_level}, the conjugacy class $\clo{\cB_{1}}{(\cG_{j} \cap \SL(\Zhat))}$ is a conjugacy class of open subgroups of $\GL(\Zhat)$. In particular, there are finitely many conjugacy classes $\cG$ of subgroups of $\GL(\Zhat)$ such that $\clo{\cB_{1}}{(\cG_{j} \cap \SL(\Zhat))} \leq \cG$. Therefore, we know that $\clo{\cB_{1}(n)}{\cG_{j}}$ is one of the finitely many $\cB_{1}$-closed conjugacy class of subgroups of $\GL(\Zhat)$ containing one of the conjugacy classes listed in Table \ref{tbl:finite_b_closures:infinite_sl_intersections_zywina}.
		
		We enumerate these conjugacy classes using \texttt{Magma}; the computation is implemented in the file \texttt{maximal\_finite\_b1\_closed.m}. We find 2651 such conjugacy classes, which are displayed in Figure \ref{fig:finite_b_closures:finite_b_closures}. Each node in the graph represents one of the 2651 conjugacy classes, while the edges indicate containment: there is an edge $\cG \to \cG'$ if and only if $\cG' \leq \cG$ and $\cG'$ is maximal in this property. The graph is topologically sorted, so that each edge goes from top to bottom, and the arrowheads are omitted. The vertices are grouped by the level of the associated conjugacy class, while the opacity of edges between conjugacy classes of differing levels has been reduced for clarity.
		
		\begin{figure}
			\centering
			\begin{sideways}
				\includegraphics[width=0.9\textheight, keepaspectratio]{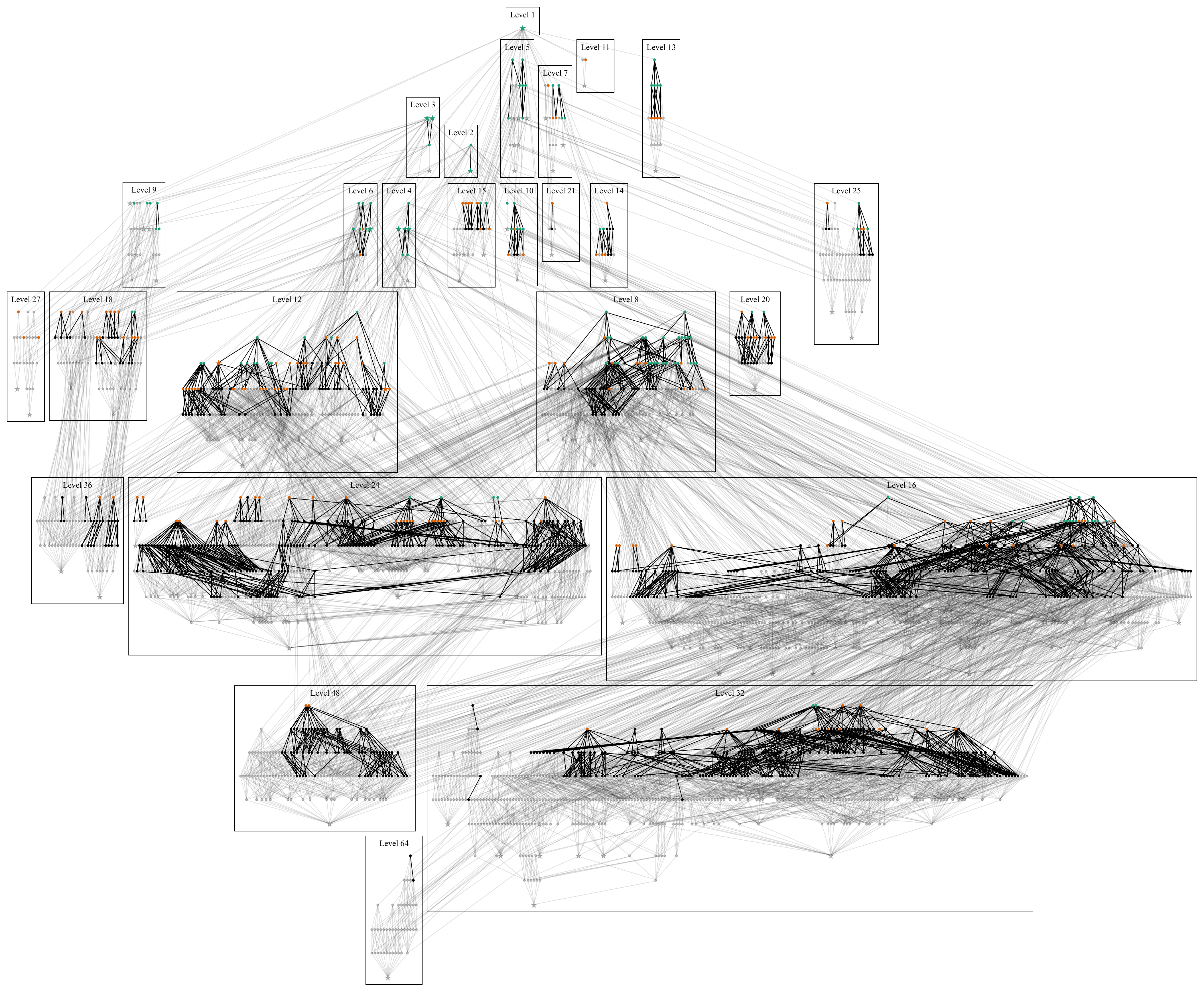}
			\end{sideways}
			\caption{The graph of the $\cB_{1}$-closed conjugacy classes of subgroups of $\GL(\Zhat)$ containing one of the conjugacy classes listed in Table \ref{tbl:finite_b_closures:infinite_sl_intersections_zywina}, ordered by inclusion.} \label{fig:finite_b_closures:finite_b_closures}
		\end{figure}
		
		We now know that the $\cB_{1}(n)$-closure of $\cG_{j}$ must be one of the 2651 conjugacy classes given in Figure \ref{fig:finite_b_closures:finite_b_closures}. Since the $\cB_{1}(n)$-closure of $\cG_{j}$ contains $\cG_{j}$ by definition, it follows from Corollary \ref{thm:preliminaries:modular_curve_rational_point} that the modular curve $X_{\clo{\cB_{1}(n)}{\cG_{j}}}$ has a rational point with $j$-invariant $j$. In particular, we must have $\det(\clo{\cB_{1}(n)}{\cG_{j}}) = \Zhat$. Computing the image of each of the 2651 aforementioned conjugacy classes under the determinant map, we find that 1573 conjugacy classes do not have full determinant. These conjugacy classes correspond to the gray vertices in Figure \ref{fig:finite_b_closures:finite_b_closures}. Thus, we know that the $\cB_{1}(n)$-closure of $\cG_{j}$ must be one of the 1078 remaining conjugacy classes.
		
		If the modular curve $X_{\clo{\cB_{1}(n)}{\cG_{j}}}$ has genus zero, then, as it has a rational point, it must be isomorphic to $\PP^{1}_{\Q}$. In particular, the modular curve $X_{\clo{\cB_{1}(n)}{\cG_{j}}}$ has infinitely many rational points. By Theorem \ref{thm:infinite_degrees:infinite_closed_subgroups}, it follows that the conjugacy class $\clo{\cB_{1}(n)}{\cG_{j}}$ is given in Table \ref{tbl:infinite_b_closures:infinite_b1_closures}, and so the first statement of Theorem \ref{thm:finite_degrees:b1_closure_infinite_agreeable} holds.
		
		Therefore, we may assume that the modular curve $X_{\clo{\cB_{1}(n)}{\cG_{j}}}$ has non-zero genus. Computing the genus of the modular curves associated to the conjugacy classes of Figure \ref{fig:finite_b_closures:finite_b_closures}, we find that 121 conjugacy classes have genus 0. These are drawn in green on Figure \ref{fig:finite_b_closures:finite_b_closures}. In particular, the $\cB_{1}(n)$-closure of $\cG_{j}$ must be one of the remaining 957 conjugacy classes.
		
		Since the modular curve $X_{\clo{\cB_{1}(n)}{\cG_{j}}}$ has a rational point with $j$-invariant $j$, it follows that $j$ must be the $j$-invariant of some non-cuspidal, non-CM rational point on the modular curve corresponding to one of the 957 conjugacy classes described above. Therefore, the result follows from the following classification of these rational points.
	\end{proof}
		
	\begin{theorem}
		Let $\cH$ be one of the 957 conjugacy classes of subgroups of $\GL(\Zhat)$ listed in Figure \ref{fig:finite_b_closures:finite_b_closures} such that $\det(\cH) = \Zhat*$ and the modular curve $X_{\cH}$ has positive genus. Let $x \in X_{\cH}$ be a non-cuspidal, non-CM rational point. Then $\cH$ and $j(x)$ are listed in Table \ref{tbl:finite_b_closures:exceptional_j_infinite_sl}.
	\end{theorem}
	
	We postpone the proof of this result to Section \ref{sec:rational_points}.
	
	\subsection{Computing degrees of rational fibers on \texorpdfstring{$X_{0}(n)$}{X\_0(n)} and \texorpdfstring{$X_{1}(n)$}{X\_1(n)}} \label{sec:finite_degrees:computing_degrees}
	
	Building on the two previous sections, we can now give a conjecturally complete classification of the degrees of the fibers with rational $j$-invariant on the modular curve $X_{0}(n)$ and $X_{1}(n)$. We split the cases of $X_{0}(n)$ and $X_{1}(n)$, as the argument proceeds slightly differently in each case.
	
	Firstly, for the modular curves $X_{0}(n)$, we obtain the following result.
	
	\begin{theorem} \label{thm:finite_degrees:x0_fiber_degrees}
		Suppose that Conjecture \ref{conj:finite_degrees:sl_intersections} holds. Fix $n \geq 1$. Then, for any conjugacy class $\cG$ of open subgroups of $\GL(\Zhat)$, there exists a non-CM rational $j$-invariant $j \in \Q \setminus \Qcm$ such that $\clo{\cB_{0}(n)}{\cG_{j}} = \cG$ if and only if $\cG$ is listed in Table \ref{tbl:infinite_b_closures:infinite_b0_closures} or \ref{tbl:finite_b_closures:finite_b0_closures}, and the level of $\cG$ divides $n$. In particular, we have that
		\begin{align*}
			&\{\ddeg_{X_{0}(n)}(j) : j \in \Q \setminus \Qcm\} \\
			& \quad = \left\{\ldblbrace [B_{0}(m) : B_{0}(n)] |\Omega| : \Omega \in B_{0}(m) \backslash {\GL(\Zhat)} / \cG \rdblbrace : \begin{array}{l}\cG \text{ in Table \ref{tbl:infinite_b_closures:infinite_b0_closures} or \ref{tbl:finite_b_closures:finite_b0_closures}} \\ \text{of level } m \mid n\end{array}\right\}.
		\end{align*}
	\end{theorem}
	
	\begin{proof}
		Let $j \in \Q$ be a non-CM rational $j$-invariant. If the modular curve $X_{(\cG_{j})_{\agr}}$ has infinitely many points, then, by Corollary \ref{thm:finite_degrees:b0_closure_infinite_agreeable}, the conjugacy class $\clo{\cB_{0}(n)}{\cG_{j}}$ is given in Table \ref{tbl:infinite_b_closures:infinite_b0_closures}.
		
		On the other hand, if the modular curve $X_{(\cG_{j})_{\agr}}$ has finitely many points, then, by Conjecture \ref{conj:finite_degrees:sl_intersections}, the conjugacy class $\cG_{j}$ is given in Table \ref{tbl:finite_b_closures:exceptional_j_invariants}. Note that, since $j$ is non-CM, $\cG_{j}$ is a conjugacy class of open subgroups of $\GL(\Zhat)$. Let $k$ be the level of $\cG_{j}$. Since $\cG_{j} \leq \clo{\cB_{0}(n)}{\cG_{j}}$, the closure $\clo{\cB_{0}(n)}{\cG_{j}}$ has level dividing $k$. In particular, by Lemma \ref{thm:closures:family_closures}, the conjugacy class $\clo{\cB_{0}(n)}{\cG_{j}}$ is equal to the $\cB_{0}((n, k))$-closure of $\cG_{j}$.
		
		We compute the $\cB_{0}(l)$-closure of $\cG_{j}$ for all $j$ given in Table \ref{tbl:finite_b_closures:exceptional_j_invariants} and all $l \geq 1$ dividing the level of $\cG_{j}$ in the file \texttt{finite\_b0\_closures}. The closures which are not given in Table \ref{tbl:infinite_b_closures:infinite_b0_closures} are listed in Table \ref{tbl:finite_b_closures:finite_b0_closures}.
		
		Therefore, we know that the group $\clo{\cB_{0}(n)}{\cG_{j}}$ is given in one of Tables \ref{tbl:infinite_b_closures:infinite_b0_closures} or \ref{tbl:finite_b_closures:finite_b0_closures}. Note that, by Lemma \ref{thm:closures:open_closures_preserve_level}, the level of $\clo{\cB_{0}(n)}{\cG_{j}}$ divides $n$. 
		
		For the opposite direction, let $\cG$ be a conjugacy class of open subgroups of $\GL(\Zhat)$ appearing in Table \ref{tbl:infinite_b_closures:infinite_b0_closures} or \ref{tbl:finite_b_closures:finite_b0_closures}, such that the level $m$ of $\cG$ divides $n$. If $\cG$ is listed in Table \ref{tbl:infinite_b_closures:infinite_b0_closures}, we know by Theorem \ref{thm:infinite_degrees:infinite_fiber_degrees} that there exist infinitely many rational $j$-invariants $j \in \Q$ such that $\clo{\cB_{0}(n)}{\cG_{j}} = \cG$. In particular, there exists a non-CM such $j$-invariant.
		
		On the other hand, if $\cG$ is listed in Table \ref{tbl:finite_b_closures:finite_b0_closures}, then by construction, there exists a non-CM $j$-invariant $j \in \Q$ such that $\clo{\cB_{0}(m)}{\cG_{j}} = \cG$. Since $m$ divides $n$, we have that $\cB_{0}(n) \leq \cB_{0}(m)$, and so, by Lemma \ref{thm:closures:inclusion_preserving_in_H}, we have
		\[
			\clo{\cB_{0}(n)}{\cG_{j}} \leq \clo{\cB_{0}(m)}{\cG_{j}} = \cG.
		\]
		By construction, the conjugacy class $\cG$ is not listed in Table \ref{tbl:infinite_b_closures:infinite_b0_closures}. Therefore, by Theorem \ref{thm:infinite_degrees:infinite_closed_subgroups}, the modular curve $X_{\cG}$ has finitely many rational points. By the above inclusion, we obtain an inclusion morphism $X_{\clo{\cB_{0}(n)}{\cG_{j}}} \to X_{\cG}$. Thus, the former modular curve also has finitely many rational points, and $\clo{\cB_{0}(n)}{\cG_{j}}$ is not listed in Theorem \ref{tbl:infinite_b_closures:infinite_b0_closures}. By construction, it follows that $\clo{\cB_{0}(n)}{\cG_{j}}$ is also given in Table \ref{tbl:finite_b_closures:finite_b0_closures}. However, it is straightforward to check that each $j$-invariant occurs only once in Table \ref{tbl:finite_b_closures:finite_b0_closures}. Thus, we must have that $\clo{\cB_{0}(n)}{\cG_{j}} = \cG$. This concludes the proof of the first assertion.
		
		For the second part, note that, by Theorem \ref{thm:closures:modular_curve_point_degrees},
		\begin{align*}
			&\{\ddeg_{X_{0}(n)}(j) : j \in \Q \setminus \Qcm\} \\
			& \quad = \{\ldblbrace \left| \Omega \right| : \Omega \in B_{0}(n) \backslash {\GL(\Zhat)} / \clo{\cB_{0}(n)}{\cG_{j}} \rdblbrace : j \in \Q \setminus \Qcm\}.
		\end{align*}
		Hence, by the above, we have
		\begin{align*}
			&\{\ddeg_{X_{0}(n)}(j) : j \in \Q \setminus \Qcm\} \\
			& \quad = \{\ldblbrace \left| \Omega \right| : \Omega \in B_{0}(n) \backslash {\GL(\Zhat)} / \cG \rdblbrace : \cG \text{ in Table \ref{tbl:infinite_b_closures:infinite_b0_closures} or \ref{tbl:finite_b_closures:finite_b0_closures} of level } m \mid n\}.
		\end{align*}
		The remainder of the proof now proceeds as for Theorem \ref{thm:infinite_degrees:infinite_fiber_degrees}.
	\end{proof}
	
	We obtain a similar statement for the modular curves $X_{1}(n)$.
	
	\begin{theorem} \label{thm:finite_degrees:x1_fiber_degrees}
		Suppose that Conjecture \ref{conj:finite_degrees:sl_intersections} holds. Fix $n \geq 1$. Then, for any conjugacy class $\cG$ of open subgroups of $\GL(\Zhat)$, there exists a non-CM rational $j$-invariant $j \in \Q \setminus \Qcm$ such that $\clo{\cB_{1}(n)}{\cG_{j}} = \cG$ if and only if $\cG$ is listed in Table \ref{tbl:infinite_b_closures:infinite_b1_closures} or \ref{tbl:finite_b_closures:finite_b1_closures}, and the level of $\cG$ divides $n$. In particular, we have
		\begin{align*}
			&\{\ddeg_{X_{1}(n)}(j) : j \in \Q \setminus \Qcm\} \\
			& \quad = \left\{\ldblbrace [B_{1}(m) : B_{1}(n)] |\Omega| : \Omega \in B_{1}(m) \backslash {\GL(\Zhat)} / \cG \rdblbrace : \begin{array}{l}\cG \text{ in Table \ref{tbl:infinite_b_closures:infinite_b1_closures} or \ref{tbl:finite_b_closures:finite_b1_closures}} \\ \text{of level } m \mid n\end{array}\right\}.
		\end{align*}
	\end{theorem}
	
	\begin{proof}
		Let $j \in \Q$ be a non-CM rational $j$-invariant. Suppose first that the intersection $\cG_{j} \cap \SL(\Zhat)$ is one of the conjugacy classes given in Table \ref{tbl:finite_b_closures:infinite_sl_intersections_zywina}. By Theorem \ref{thm:finite_degrees:b1_closure_infinite_agreeable}, either the $\cB_{1}(n)$-closure of $\cG_{j}$ is listed in Table \ref{tbl:infinite_b_closures:infinite_b1_closures}, or the conjugacy class $\cG_{j}$ is given in one of the unshaded rows of Table \ref{tbl:finite_b_closures:exceptional_j_infinite_sl}.
		
		On the other hand, suppose that the intersection $\cG_{j} \cap \SL(\Zhat)$ is not given in Table \ref{tbl:finite_b_closures:infinite_sl_intersections_zywina}. By Conjecture \ref{conj:finite_degrees:sl_intersections}, the conjugacy class $\cG_{j}$ is given in one of the unshaded rows of Table \ref{tbl:finite_b_closures:exceptional_j_invariants}.
		
		Therefore, we have shown that either the $\cB_{1}(n)$-closure of $\cG_{j}$ is given in Table \ref{tbl:infinite_b_closures:infinite_b1_closures}, or $\cG_{j}$ is given in one of the unshaded rows of Tables \ref{tbl:finite_b_closures:exceptional_j_invariants} or \ref{tbl:finite_b_closures:exceptional_j_infinite_sl}. As in the proof of Theorem \ref{thm:finite_degrees:x0_fiber_degrees}, we compute the $\cB_{1}(n)$-closure of all the latter subgroups $\cG_{j}$. This is done in the file \texttt{finite\_b1\_closures}, and the resulting closures not given in Table \ref{tbl:infinite_b_closures:infinite_b1_closures} are listed in Table \ref{tbl:finite_b_closures:finite_b1_closures}.
		
		Therefore, we know that the group $\clo{\cB_{1}(n)}{\cG_{j}}$ is given in one of Tables \ref{tbl:infinite_b_closures:infinite_b1_closures} or \ref{tbl:finite_b_closures:finite_b1_closures}. Note that, by Lemma \ref{thm:closures:open_closures_preserve_level}, the level of $\clo{\cB_{1}(n)}{\cG_{j}}$ divides $n$.
		
		The proof of the remainder of the first assertion, and of the second assertion, are now analogous to the proof of Theorem \ref{thm:finite_degrees:x0_fiber_degrees}, noting once again that each $j$-invariant appears only once in Table \ref{tbl:finite_b_closures:finite_b1_closures}.
	\end{proof}
	
	As was the case for Theorem \ref{thm:infinite_degrees:infinite_fiber_degrees}, Theorems \ref{thm:finite_degrees:x0_fiber_degrees} and \ref{thm:finite_degrees:x1_fiber_degrees} provide explicit procedures for computing the set of degrees of rational fibers on the modular curves $X_{0}(n)$ and $X_{1}(n)$ respectively. As before, these are implemented in the functions \texttt{PossibleX0Fibers} and \texttt{PossibleX1Fibers} defined in the file \texttt{degrees\_x0\_x1.m}.
	
	In addition, it is straightforward to deduce the degrees of the points on $X_{0}(n)$ and $X_{1}(n)$ with rational $j$-invariant from the degrees of the fibers given in Theorem \ref{thm:infinite_degrees:infinite_fiber_degrees}. This yields the statements of Theorems \ref{thm:introduction:finite_x0_degrees} and \ref{thm:introduction:finite_x1_degrees}.
	
	\section{Determining rational points on modular curves} \label{sec:rational_points}
	
	In this section, we tackle the last missing piece in the proof of Theorem \ref{thm:finite_degrees:b1_closure_infinite_agreeable}. Namely, we prove the following.
	
	\begin{theorem} \label{thm:rational_points:main_theorem}
		Let $\cH$ be one of the 957 conjugacy classes of subgroups of $\GL(\Zhat)$ listed in Figure \ref{fig:finite_b_closures:finite_b_closures} such that $\det(\cH) = \Zhat*$ and the modular curve $X_{\cH}$ has positive genus. Let $x \in X_{\cH}$ be a non-cuspidal, non-CM rational point. Then $\cH$ and $j(x)$ are listed in Table \ref{tbl:finite_b_closures:exceptional_j_infinite_sl}.
	\end{theorem}
	
	While the reader may have bristled at the thought of computing the rational points on 957 different, potentially high genus, modular curves, we make a simplification which greatly reduces the size of this problem. As evidenced in Figure \ref{fig:finite_b_closures:finite_b_closures}, many of the 957 aforementioned conjugacy classes are not maximal; that is to say, for such a conjugacy class $\cG$, there exists a different conjugacy class $\cG'$, in this same list of 957 classes, such that $\cG \leq \cG'$. As there is an associated inclusion morphism $X_{\cG} \to X_{\cG'}$ which commutes with the $j$-map, it follows that
	\[
		j(X_{\cG}(\Q)) \subseteq j(X_{\cG'}(\Q)).
	\]
	In particular, it suffices to compute the rational points on the modular curves corresponding to the maximal classes amongst the 957 conjugacy classes. There are 160 such maximal conjugacy classes, which are drawn in orange in Figure \ref{fig:finite_b_closures:finite_b_closures} and are listed in Table \ref{tbl:finite_b_closures:maximal_finite_b1_closed}. We have now reduced the proof of Theorem \ref{thm:rational_points:main_theorem} to the computation of rational points on 160 different modular curves.
	
	Rather than go through each modular curve individually, we detail all of the techniques used in this computation, most of which are straightforward to apply using \texttt{Magma} and the data in the LMFDB. Table \ref{tbl:finite_b_closures:maximal_finite_b1_closed} then summarizes the technique used for each curve.
	
	We note that, while finding rational points on curves is often a difficult task, most of the modular curves in Table \ref{tbl:finite_b_closures:maximal_finite_b1_closed} can be tackled using elementary and straightforward methods. This is explained by the fact that these modular curves are maximal $\cB_{1}$-closed of non-zero genus, and so often admit a non-trivial modular cover to a modular curve which is maximal of non-zero genus, but not necessarily $\cB_{1}$-closed. This cover commutes with the $j$-map, and so it suffices to prove the desired result for the covered curve. This latter curve often has very small genus and is straightforward to tackle using existing computational techniques.
	
	\subsection{Prior work} \label{sec:rational_points:prior_work}
	
	A number of the modular curves in Table \ref{tbl:finite_b_closures:maximal_finite_b1_closed} have been handled in previous work.
	
	In \cite{zywina2015possible}, Zywina shows that almost all of the rational points on the modular curves of Table \ref{tbl:finite_b_closures:maximal_finite_b1_closed} with prime level are either cuspidal or CM points. The only exception is the unique rational point on the modular curve \href{https://beta.lmfdb.org/ModularCurve/Q/7.56.1.b.1}{7.56.1.b.1} with $j$-invariant $\frac{2268945}{128}$.
	
	The modular curves of prime power level $2^n$ were considered by Rouse and Zureick-Brown in \cite{rouse2015elliptic}. Their work shows that the rational points on the modular curves of Table \ref{tbl:finite_b_closures:maximal_finite_b1_closed} with level $2^n$ are all cuspidal or CM points, apart from four rational points on the modular curve \href{https://beta.lmfdb.org/ModularCurve/Q/16.96.3.fa.1}{16.96.3.fa.1} with $j$-invariant $\frac{4097^{3}}{2^{4}}$ and four rational points on the modular curve \href{https://beta.lmfdb.org/ModularCurve/Q/16.96.3.fa.2}{16.96.3.fa.2} with $j$-invariant $\frac{16974593}{256}$.
	
	The modular curves of prime power level $\ell^n$, for $\ell \in \{3, 5, 7, 11\}$, were considered by Rouse, Sutherland and Zureick-Brown in \cite{rouse2022ell-adic}. In our case, this work shows that the rational points on the modular curves of level 25 and 27 in Table \ref{tbl:finite_b_closures:maximal_finite_b1_closed} are all cuspidal.
	
	\subsection{Local solubility} \label{sec:rational_points:local_solubility}
	
	Given a modular curve $X_{H}$, for some open subgroup $H \leq \GL(\Zhat)$, it may be that there is a local obstruction to the existence of rational points on $X_{H}$; that is to say, the sets $X_{H}(\RR)$ or $X_{H}(\Q_{p})$, for some prime $p$, may be empty. This can be checked in a few different ways.
	
	The existence of real or $\Q_{p}$-points on the modular curve $X_{H}$, with $p$ not dividing the level of $H$, can be determined directly from the subgroup $H$, without computing a model for the curve $X_{H}$. For the case of real points, this is determined in \cite[Proposition 3.5]{zywina2022possible}, while the case of $\Q_{p}$-points is covered by \cite[Section 5.1]{rouse2022ell-adic}. The result of these computations is stored in the LMFDB. In particular, this shows that the modular curves \href{https://beta.lmfdb.org/ModularCurve/Q/12.64.1.a.1}{12.64.1.a.1}, \href{https://beta.lmfdb.org/ModularCurve/Q/12.64.1.a.2}{12.64.1.a.2} and \href{https://beta.lmfdb.org/ModularCurve/Q/14.144.4.d.1}{14.144.4.d.1} have no rational points.
	
	To check the existence of $\Q_{p}$-points on the modular curve $X_{H}$, where $p$ divides the level of $H$, we make use of the model for $X_{H}$ given in the LMFDB. Given such a model, we may use the \texttt{Magma} intrinsic \texttt{IsLocallySolvable} to test whether the set $X_{H}(\Q_{p})$ is empty, for all primes $p$ dividing the level of $H$. This method shows that the modular curves \href{https://beta.lmfdb.org/ModularCurve/Q/12.96.1.d.1}{12.96.1.d.1}, \href{https://beta.lmfdb.org/ModularCurve/Q/12.96.1.f.2}{12.96.1.f.2}, \href{https://beta.lmfdb.org/ModularCurve/Q/12.72.4.t.1}{12.72.4.t.1}, \href{https://beta.lmfdb.org/ModularCurve/Q/14.144.4.c.1}{14.144.4.c.1}, \href{https://beta.lmfdb.org/ModularCurve/Q/14.144.4.c.2}{14.144.4.c.2}, \href{https://beta.lmfdb.org/ModularCurve/Q/15.60.3.f.1}{15.60.3.f.1}, \href{https://beta.lmfdb.org/ModularCurve/Q/15.72.3.i.1}{15.72.3.i.1}, \href{https://beta.lmfdb.org/ModularCurve/Q/15.72.3.i.2}{15.72.3.i.2}, \href{https://beta.lmfdb.org/ModularCurve/Q/15.90.4.f.1}{15.90.4.f.1}, \href{https://beta.lmfdb.org/ModularCurve/Q/18.108.4.b.1}{18.108.4.b.1}, \href{https://beta.lmfdb.org/ModularCurve/Q/21.126.4.a.1}{21.126.4.a.1}, \href{https://beta.lmfdb.org/ModularCurve/Q/24.96.1.dk.1}{24.96.1.dk.1}, \href{https://beta.lmfdb.org/ModularCurve/Q/24.96.1.dk.2}{24.96.1.dk.2}, \href{https://beta.lmfdb.org/ModularCurve/Q/24.96.3.iu.1}{24.96.3.iu.1}, \href{https://beta.lmfdb.org/ModularCurve/Q/24.72.4.jc.1}{24.72.4.jc.1} and \href{https://beta.lmfdb.org/ModularCurve/Q/24.96.5.jb.1}{24.96.5.jb.1} have no rational points.
	
	\subsection{Rank 0 elliptic curves} \label{sec:rational_points:rank_0_ec}
	
	Suppose that the modular curve $X_{H}$ is a rank 0 elliptic curve. Using a model for the curve given in the LMFDB, we can compute the finite set $X_{H}(\Q)$ using the \texttt{MordellWeilGroup} intrinsic in \texttt{Magma}. Alternatively, the structure and generators of the Mordell-Weil group $X_{H}(\Q)$ can be read off of the elliptic curve database of the LMFDB. Using the explicit equations for the $j$-map $j : X_{H} \to \PP^{1}_{\Q}$ stored in the LMFDB, one can then compute the $j$-invariant of each rational point of $X_{H}$.
	
	We note that, as the LMFDB stores a list of known rational points on the modular curve $X_{H}$ and their corresponding $j$-invariants, it is often sufficient to check that the size of this list is equal to the size of $X_{H}(\Q)$ as determined above.
	
	This method can be used to find the rational points on all genus 1 modular curves in Table \ref{tbl:finite_b_closures:maximal_finite_b1_closed} not handled by one of the previous methods. The only non-cuspdial, non-CM rational points found on these curves are as follows:
	\begin{itemize}
		\item Four rational points on the modular curve \href{https://beta.lmfdb.org/ModularCurve/Q/12.32.1.b.1}{12.32.1.b.1}; two with $j$-invariant $-\frac{35937}{4}$ and two with $j$-invariant $\frac{109503}{64}$.
		\item Four rational points on the modular curve \href{https://beta.lmfdb.org/ModularCurve/Q/15.36.1.b.1}{15.36.1.b.1}; two with $j$-invariant $\frac{1331}{8}$ and two with $j$-invariant $-\frac{1680914269}{32768}$.
	\end{itemize}
	In addition, a number of higher genus modular curves in Table \ref{tbl:finite_b_closures:maximal_finite_b1_closed} admit a modular cover to a rank 0 elliptic curve. Applying the same method to the latter modular curve may then be used to show that the original modular curve has no non-cuspidal, non-CM rational points. The complete list of modular curves for which this method was used can be found in Table \ref{tbl:finite_b_closures:maximal_finite_b1_closed}.
	
	\subsection{Chabauty for genus 2 curves} \label{sec:rational_points:chabauty}
	
	Let $X_{H}$ be a modular curve such that $r < g(X_{H})$, where $r$ is the rank of the Mordell-Weil group $J_{X_{H}}(\Q)$ of the Jacobian of $X_{H}$. In this case, the Chabauty-Coleman method provides an explicit for computing the set $X_{H}(\Q)$ via $p$-adic integration. See \cite{mccallum2012method} for further details.
	
	Chabauty's method is implemented in \texttt{Magma} for genus 2 curves via two intrinsics: \texttt{Chabauty0} provides the functionality for curves whose Jacobian has rank 0, while \texttt{Chabauty} covers the case of rank 1 Jacobians. In both cases, by making use of additional techniques such as the Mordell-Weil sieve, the intrinsic returns the complete set of rational points on the input curve.
	
	To apply this to the modular curves in Table \ref{tbl:finite_b_closures:maximal_finite_b1_closed}, we proceed much as in the previous section. We first retrieve a model for the modular curve from the LMFDB, and use the \texttt{RankBounds} intrinsic to compute the rank of the Jacobian. This intrinsic is not guaranteed to compute the exact rank of the Jacobian; however, for all of our curves, an exact value for the rank is found. We then apply one of the aforementioned Chabauty intrinsics to compute the set of rational points on the curve, before using the $j$-map given in the LMFDB to compute the $j$-invariant of each rational point. As before, we also apply this method to a number of higher genus modular curves which admit a modular cover to a genus 2 curve with rank 0 or 1.
	
	The \texttt{Chabauty0} intrinsic was used for the modular curves \href{https://beta.lmfdb.org/ModularCurve/Q/12.72.2.a.1}{12.72.2.a.1}, \href{https://beta.lmfdb.org/ModularCurve/Q/18.72.2.c.1}{18.72.2.c.1}, \href{https://beta.lmfdb.org/ModularCurve/Q/18.72.2.c.2}{18.72.2.c.2}, \href{https://beta.lmfdb.org/ModularCurve/Q/18.108.2.a.1}{18.108.2.a.1}, \href{https://beta.lmfdb.org/ModularCurve/Q/18.108.2.d.1}{18.108.2.d.1}, \href{https://beta.lmfdb.org/ModularCurve/Q/18.108.2.d.2}{18.108.2.d.2}, \href{https://beta.lmfdb.org/ModularCurve/Q/18.108.4.g.1}{18.108.4.g.1}, \href{https://beta.lmfdb.org/ModularCurve/Q/18.108.4.g.2}{18.108.4.g.2}, \href{https://beta.lmfdb.org/ModularCurve/Q/24.72.2.e.1}{24.72.2.e.1}, \href{https://beta.lmfdb.org/ModularCurve/Q/24.72.2.hl.1}{24.72.2.hl.1} and \href{https://beta.lmfdb.org/ModularCurve/Q/24.72.2.hl.2}{24.72.2.hl.2}, while the \texttt{Chabauty} intrinsic was used for the modular curve \href{https://beta.lmfdb.org/ModularCurve/Q/15.90.4.d.1}{15.90.4.d.1}. We find that all rational points on these modular curves are cusps or CM points, except for the following:
	\begin{itemize}
		\item Three rational points on the modular curve \href{https://beta.lmfdb.org/ModularCurve/Q/18.72.2.c.1}{18.72.2.c.1} with $j$-invariant $406749952$.
		\item Three rational points on the modular curve \href{https://beta.lmfdb.org/ModularCurve/Q/18.72.2.c.2}{18.72.2.c.2} with $j$-invariant $1792$.
		\item Two rational points on the modular curve \href{https://beta.lmfdb.org/ModularCurve/Q/24.72.2.hl.1}{24.72.2.hl.1} with $j$-invariant $4913$.
		\item Two rational points on the modular curve \href{https://beta.lmfdb.org/ModularCurve/Q/24.72.2.hl.2}{24.72.2.hl.2} with $j$-invariant $16974593$.
	\end{itemize}
	
	\subsection{The curve 15.90.3.c.1} \label{sec:rational_points:15.90.3.c.1}
	
	The methods described above allow us to determine the rational points on all but one of the modular curves in Table \ref{tbl:finite_b_closures:maximal_finite_b1_closed}, namely the modular curve \href{https://beta.lmfdb.org/ModularCurve/Q/15.90.3.c.1}{15.90.3.c.1}. This is a genus 3 curve with two rational cusps. We will show in this section that only rational points on the curve are the two rational cusps.
	
	At the time of writing, there is no model for the curve available in the LMFDB. To obtain a model for the curve, we utilize Zywina's algorithm for computing the Atkin-Lehner action on cusp forms \cite{zywina2021computing, zywina2021computing-implementation}. This provides a function, \texttt{FindCanonicalModel}, for computing the canonical model of a modular curve $X_{H}$, where $H$ is an open subgroup of $\GL(\Zhat)$ with full determinant and containing $-I$. This gives a model for the modular curve \href{https://beta.lmfdb.org/ModularCurve/Q/15.90.3.c.1}{15.90.3.c.1} in $\PP^{2}$, given by
	\[
		C: (x + y + z) y^3 = (x^2 - x z + z^2) (x z - y^2).
	\]
	A search for points on $C$ reveals only the two rational points $(0 : 0 : 1)$ and $(1 : 0 : 0)$.
	
	We now compute the rank of the Jacobian $J$ of the modular curve \href{https://beta.lmfdb.org/ModularCurve/Q/15.90.3.c.1}{15.90.3.c.1}. To do so, we make use of the isogeny decomposition of $J$ computed using \cite[Section 6]{rouse2022ell-adic} and stored in the LMFDB. In our case, the Jacobian $J$ is isogenous to the product $E_{1} \times E_{2} \times E_{3}$ of three elliptic curves, where we may take $E_{1}$, $E_{2}$ and $E_{3}$ to be the elliptic curves 15.a7, 75.c1 and 225.c2 respectively. The elliptic curves $E_{1}$ and $E_{2}$ have (algebraic) rank 0, while the elliptic curve $E_{3}$ has rank 1. In particular, the Jacobian $J$ has rank 1.
	
	Since the Jacobian $J$ has rank strictly smaller than the genus of $C$, we can attempt to use Chabauty's method to determine the set of rational points on $C$. As Chabauty's method is not, at present, implemented in \texttt{Magma} for genus 3 curves, we undertake this computation by hand.
	
	Throughout the remainder, we work with the prime 2. This is a prime of good reduction for the modular curve \href{https://beta.lmfdb.org/ModularCurve/Q/15.90.3.c.1}{15.90.3.c.1}, as it does not divide the level, and it is straightforward to check that it is also a prime of good reduction for the model $C$. Moreover, we have that
	\[
		C(\FF_{2}) = \{(0 : 0 : 1), (1 : 0 : 0)\}.
	\]
	It therefore suffices to show that each residue disk, that is to say, each fiber of the map $C(\Q_{2}) \to C(\FF_{2})$, contains a single rational point.
	
	The curve $C$ has an involution defined by $\iota : (x : y : z) \mapsto (z : y : x)$. Using the \texttt{CurveQuotient} intrinsic in \texttt{Magma}, we find that the quotient curve $C / \langle \iota \rangle$ is isomorphic to the elliptic curve $E_{3}$. Therefore, denoting by $f : C \to C / \langle \iota \rangle$ the quotient map, we obtain a pushforward map $f_{\ast} : J \to E_{3}$. As $J$ is isogenous to the product $E_{1} \times E_{2} \times E_{3}$, it follows that the kernel of $f_{\ast}$ is isogenous to $E_{1} \times E_{2}$. In particular, the group $(\ker f_{\ast})(\Q)$ is torsion.
	
	Consider the $\Q_{2}$-vector space $H^{0}(J_{\Q_{2}}, \Omega^{1})$ of regular 1-forms on the Jacobian $J$, considered over $\Q_{2}$. The involution $\iota$ defines a trace map given by
	\begin{align*}
		\operatorname{Tr} : H^{0}(J_{\Q_{2}}, \Omega^{1}) & \to H^{0}(J_{\Q_{2}}, \Omega^{1}) \\
		\omega & \mapsto \omega + \iota^{\ast} \omega.
	\end{align*}
	Let $\omega$ be an element of the kernel of the trace map $\operatorname{Tr}$. Let $[D] \in J(\Q)$, for some divisor $D$ on $C / \Q_{2}$. Since $\omega$ is in the kernel of $\operatorname{Tr}$, we have that $\omega = -\iota^{\ast} \omega$. Therefore, by the functoriality of Coleman integration \cite[Proposition 2.4]{coleman1985torsion}, we have
	\begin{align*}
		\int_{0}^{[D]} \omega & = \frac{1}{2} \left(\int_{0}^{[D]} \omega + \int_{0}^{[D]} \omega\right) \\
		& = \frac{1}{2} \left(\int_{0}^{[D]} \omega - \int_{0}^{[D]} \iota^{\ast} \omega\right) \\
		& = \frac{1}{2} \left(\int_{0}^{[D]} \omega - \int_{0}^{[\iota(D)]} \omega\right) \\
		& = \frac{1}{2} \int_{[\iota(D)]}^{[D]} \omega \\
		& = \frac{1}{2} \int_{0}^{[D - \iota(D)]} \omega.
	\end{align*}
	By definition of $f$, we have that
	\[
		f_{\ast}([D - \iota(D)]) = [f(D) - f(\iota(D))] = [f(D) - f(D)] = 0.
	\]
	Therefore, the class $[D - \iota(D)]$ belongs to the intersection $(\ker f_{\ast}) \cap J(\Q) = (\ker f_{\ast})(\Q)$. By the above, the latter is torsion, and so, again by the functoriality of integration, we have
	\[
		\int_{0}^{[D]} \omega = \frac{1}{2} \int_{0}^{[D - \iota(D)]} \omega = 0.
	\]
	
	By \cite[Proposition 2.2]{milne1986jacobian}, there is an isomorphism of $\Q_{2}$-vector spaces $H^{0}(J_{\Q_{2}}, \Omega^{1}) \cong H^{0}(C_{\Q_{2}}, \Omega^{1})$. Throughout the remainder, we implicitly work with this isomorphism. In particular, setting
	\[
		\int_{P}^{Q} \omega := \int_{0}^{[Q - P]} \omega_{J},
	\]
	for any $P, Q \in C(\Q_{2})$ and $\omega \in H^{0}(C_{\Q_{2}}, \Omega^{1})$ corresponding to $\omega_{J} \in H^{0}(J_{\Q_{2}}, \Omega^{1})$, the above shows that
	\[
		\int_{P}^{Q} \omega = 0,
	\]
	for all $P, Q \in C(\Q)$ and $\omega \in \ker \operatorname{Tr}$.
	
	Consider the smooth affine model of $C$ on the affine chart $z \neq 0$ given by
	\[
		C_{z} : (X + Y + 1) Y^{3} - (X^2 - X + 1) (X - Y^{2}) = 0,
	\]
	where $X = \frac{x}{z}$ and $Y = \frac{y}{z}$. A basis for the vector space of regular 1-forms on $C_{z}$ is given by the 1-forms $\omega$, $X \omega$ and $Y \omega$, where
	\begin{align*}
		\omega &= \frac{dX}{4Y^{3} + 3(X + 1)Y^{2} + 2(X^{2} - X + 1)Y} \\
		&= -\frac{dY}{Y^{3} + (2X - 1)Y^{2} - 3X^{2} + 2X - 1}.
	\end{align*}
	The action of $\iota$ on $C_{z}$ is given by $\iota(X, Y) = (1/X, Y/X)$, from which we compute that
	\begin{align*}
		\iota^{\ast} \omega & = - \frac{X dX}{4 Y^{3} + 3(X + 1)Y^2 + 2(X^2 - X + 1)Y} = - X \omega, \\
		\iota^{\ast} (X \omega) & = - \omega, \\
		\iota^{\ast} (Y \omega) & = - Y \omega.
	\end{align*}
	Therefore, the kernel of the trace map $\operatorname{Tr}$ is two-dimensional, with basis $\{\omega + X \omega, Y \omega\}$.
	
	Let $P \in C(\Q)$ be in the residue disk of $(0 : 0 : 1) \in C(\FF_{2})$. Note that $P$ must belong to the affine chart $z \neq 0$, and so we may write $P = (u, v) \in C_{z}$, for some $u, v \in \Q$ such that $(u, v) \equiv (0, 0) \pmod{2}$. In particular, we have that $u, v \in 2 \Z_{2}$.
	
	By the argument above, we have that
	\[
		\int_{(0, 0)}^{(u, v)} \omega + X \omega = 0.
	\]
	As the points $(0, 0)$ and $(u, v)$ belong to the same residue disk, we can evaluate the integral on the left side by expanding in a power series in a local parameter on $C_{z}$. The function $Y$ is a uniformizer for $C_{z}$ at $(0, 0)$, both over $\Q_{2}$ and $\FF_{2}$. Moreover, using the equation for $C_{z}$ given above, we can write the function $X$ on $C_z$ as the power series
	\[
		X = Y^{2} + Y^{3} + Y^{4} + 2Y^{5} + O(Y^{6}) \in \Z_{2}[[Y]].
	\]
	Therefore, we have that
	\begin{align*}
		\omega + X \omega &= -\frac{(1 + X) dY}{Y^{3} + (2X - 1)Y^{2} - 3X^{2} + 2X - 1} \\
		&= -\frac{1 + Y^{2} + Y^{3} + Y^{4} + 2Y^{5} + O(Y^{6})}{-1 + Y^{2} + 3Y^{3} + Y^{4} + O(Y^{6})} dY \\
		&= \left(1 + 2Y^{2} + 4Y^{3} + 4Y^{4} + 12Y^{5} + O(Y^{6})\right) dY.
	\end{align*}
	Thus, we obtain that
	\begin{align*}
		0 = \int_{(0, 0)}^{(u, v)} \omega + X \omega &= \int_{0}^{v} \left(1 + 2Y^{2} + 4Y^{3} + 4Y^{4} + 12Y^{5} + O(Y^{6})\right) dY \\
		&= \left[Y + \frac{2}{3} Y^{3} + Y^{4} + \frac{4}{5} Y^{5} + 2 Y^{6} + O(Y^{7})\right]_{Y = 0}^{Y = v} \\
		&= v \left(1 + \frac{2}{3} v^{2} + v^{3} + \frac{4}{5} v^{4} + 2 v^{5} + O(v^{6})\right).
	\end{align*}
	Let $f(t) = 1 + \frac{2}{3} t^{2} + t^{3} + \frac{4}{5} t^{4} + 2 t^{5} + O(t^{6}) \in \Q_{2}[[t]]$ be the power series constructed above. Writing $f(t) = \sum_{i = 0}^{\infty} a_{i} t^{i}$, by construction, we have that $a_{i} = \frac{b_{i}}{i + 1}$ for some $b_{i} \in \Z_{2}$, for all $i$. In particular, for $i \geq 2$, we have that
	\[
		v_{2}(a_{i}) \geq - v_{2}(i + 1) > -i.
	\]
	Moreover, the above expansion shows that $v_{2}(a_{1}) = \infty > -1$. Therefore, the Newton polygon of $f$ has no segments with slope less than or equal to $-1$. By the theory of Newton polygons \cite[Chapter IV.4]{koblitz1984p-adic}, it follows that $f$ has no roots in $2\Z_{2}$. In particular, since $v \in 2\Z_{2}$, it follows that we must have $v = 0$. The equation of $C_{z}$ then shows that $u = 0$ as well, since $u \in 2\Z_{2}$, and so $P = (0, 0)$. In other words, we have shown that the only rational point in the residue disk of $(0 : 0 : 1) \in C(\FF_{2})$ is the point $(0 : 0 : 1) \in C(\Q)$.
	
	The involution $\iota$ maps the residue disk of $(0 : 0 : 1) \in C(\FF_{2})$ to the residue disk of $(1 : 0 : 0) \in C(\FF_{2})$. Therefore, there is also a single rational point in the residue disk of $(1 : 0 : 0) \in C(\FF_{2})$, namely $(1 : 0 : 0) \in C(\Q)$. Thus, the modular curve \href{https://beta.lmfdb.org/ModularCurve/Q/15.90.3.c.1}{15.90.3.c.1} has exactly two rational points, both of which are cuspidal.
	
	\section{Isolated points with rational \texorpdfstring{$j$}{j}-invariant on \texorpdfstring{$X_1(n)$}{X\_1(n)}} \label{sec:isolated_points}
	
	In this final section, we will aim to use the results of this paper, and particularly the notion of $H$-closures, to classify, conditionally on Conjecture \ref{conj:finite_degrees:sl_intersections}, the $j$-invariants of all non-cuspidal, non-CM isolated points with rational $j$-invariant on the modular curves $X_{1}(n)$. In particular, we show that the conjectural set of such $j$-invariants given by Bourdon, Hashimoto, Keller, Klagsbrun, Lowry-Duda, Morrison, Najman and Shukla in \cite[Conjecture 4]{bourdon2025towards} is complete assuming Conjecture \ref{conj:finite_degrees:sl_intersections}.
	
	Following \cite{terao2025isolated}, we set the following piece of notation. Given an open subgroup $H$ of $\GL(\Zhat)$, and a closed point $x \in X_{H}$, we say that a pair $(E, [\alpha]_{H})$ is a minimal representative for $x$ if $E$ is an elliptic curve defined over $\Q(j(x))$, $\alpha$ a profinite level structure on $E$, and the closed point $x$ corresponds to the Galois orbit of the geometric point $[(E, [\alpha]_{H})] \in X_{H}(\Qbar)$. By \cite[Lemma 4.9 and Theorem 4.20]{terao2025isolated}, such a minimal representative exists for any closed point $x \in X_{H}$.
	
	The central result which allows us to relate the work done in this paper to the question of isolated points is the following. An existing argument, used for instance in \cite[Theorem 38]{bourdon2025towards} and \cite[Theorem 5.4]{terao2025isolated}, allows one to construct, from an isolated point on any modular curve $x \in X_{H}$, another isolated point on a specific modular curve closely related to the image of the Galois representation associated to a minimal representative for $x$. We show that this same argument can be applied to a large set of subgroups of $\GL(\Zhat)$ characterized by the notion of $H$-equivalence defined in Section \ref{sec:closures}. The precise statement is as follows.
	
	\begin{theorem} \label{thm:isolated_points:equivalent_isolated}
		Let $H$ be an open subgroup of $\GL(\Zhat)$, and let $x \in X_{H}$ be a non-cuspidal closed point with minimal representative $(E, [\alpha]_{H})$ and $j(x) \notin \{0, 1728\}$. Let $G$ be an open subgroup of $\GL(\Zhat)$ such that $G_{E, \alpha} \leq G$ and the groups $G$ and $G_{E, \alpha}$ are $\pm H$-equivalent. Suppose that $x$ is isolated. Then the closed point $y \in X_{G}$ corresponding to the $G_{\Q}$-orbit of the geometric point $[(E, [\alpha]_{G})] \in X_{G}(\Qbar)$ is also isolated.
	\end{theorem}
	
	\begin{proof}
		Since both $G$ and $H$ are open subgroups of $\GL(\Zhat)$, the intersection $G \cap H$ is also an open subgroup of $\GL(\Zhat)$. Therefore, consider the pair of maps of modular curves
		\[\begin{tikzcd}
			X_{H} & X_{G \cap H} \ar[l, "f"'] \ar[r, "g"] & X_{G},
		\end{tikzcd}\]
		each induced by the natural inclusion of subgroups. Let $z \in X_{G \cap H}$ be the closed point corresponding to the $G_{\Q}$-orbit of the geometric point $[(E, [\alpha]_{G \cap H})] \in X_{G \cap H}(\Qbar)$. By the definition of the inclusion morphisms, we have that $f(z) = x$ and $g(z) = y$.
		
		Since $A_{E, \alpha} \leq G_{E, \alpha} \leq G$, we have that
		\begin{align*}
			[G_{E, \alpha} \cap A_{E, \alpha} H : G_{E, \alpha} \cap A_{E, \alpha}(G \cap H)] &= [G_{E, \alpha} \cap A_{E, \alpha} H : G_{E, \alpha} \cap G \cap A_{E, \alpha} H] \\
			&= [G_{E, \alpha} \cap A_{E, \alpha} H : G_{E, \alpha} \cap A_{E, \alpha} H] \\
			&= 1.
		\end{align*}
		Therefore, since $x = f(z)$ is isolated, by \cite[Theorem 5.2]{terao2025isolated}, the closed point $z \in X_{G \cap H}$ is also isolated.
		
		As $j(E) = j(x) \notin \{0, 1728\}$, we have that $A_{E, \alpha} = \{\pm I\}$. Thus,
		\begin{align*}
			[G_{E, \alpha} \cap A_{E, \alpha} G : G_{E, \alpha} \cap A_{E, \alpha} (G \cap H)] = [G_{E, \alpha} : G_{E, \alpha} \cap (\pm H)].
		\end{align*}
		Since the groups $G$ and $G_{E, \alpha}$ are $\pm H$-equivalent, we have that $G \subseteq (\pm H) G = (\pm H) G_{E, \alpha}$. Moreover, since $G_{E, \alpha} \leq G$, we obtain that $G = (\pm H \cap G) G_{E, \alpha}$. Consider now the map between the set of right cosets $(\pm H \cap G_{E, \alpha}) \backslash G_{E, \alpha}$ and the set of right cosets $(\pm H \cap G) \backslash G$ given by
		\begin{align*}
			\varphi : (\pm H \cap G_{E, \alpha}) \backslash G_{E, \alpha} & \to (\pm H \cap G) \backslash G \\
			(\pm H \cap G_{E, \alpha}) g & \mapsto (\pm H \cap G) g.
		\end{align*}
		Given any two elements $g, g' \in G_{E, \alpha}$, we have
		\begin{align*}
			(\pm H \cap G) g = (\pm H \cap G) g' & \iff g' g^{-1} \in \pm H \cap G \\
			& \iff g' g^{-1} \in \pm H \cap G \cap G_{E, \alpha} \\
			& \iff g' g^{-1} \in \pm H \cap G_{E, \alpha} \\
			& \iff (\pm H \cap G_{E, \alpha}) g = (\pm H \cap G_{E, \alpha}) g'.
		\end{align*}
		Therefore, the map $\varphi$ is both well-defined and injective. Moreover, since $G = (\pm H \cap G) G_{E, \alpha}$, the map $\varphi$ is also surjective, and hence a bijection. Thus, we obtain that
		\begin{align*}
			[G_{E, \alpha} \cap A_{E, \alpha} G : G_{E, \alpha} \cap A_{E, \alpha} (G \cap H)] &= [G_{E, \alpha} : G_{E, \alpha} \cap (\pm H)] \\
			&= [G : G \cap (\pm H)] \\
			&= [\pm G : \pm(G \cap H)].
		\end{align*}
		By \cite[Theorem 5.1]{terao2025isolated}, since the closed point $z \in X_{G \cap H}$ is isolated, the closed point $y = g(z)$ is isolated as well.
	\end{proof}
	
	\pagebreak
	
	In addition to providing a connection between the notions defined in Section \ref{sec:closures} and isolated points on modular curves, this result provides a convenient framework for understanding existing work on isolated points on the modular curves $X_{0}(n)$ and $X_{1}(n)$. Fix $n \geq 1$, and consider the following inclusions of subgroups of $\GL(\Zhat)$:
	\[\begin{tikzcd}[row sep=tiny]
		\clo{B_{0}(n)}{G_{E, \alpha}} \\
		& \clo{B_{1}(n)}{G_{E, \alpha}} \ar[ul] \\
		\clo{\cB_{0}(n)}{G_{E, \alpha}} \ar[uu] \\
		& \clo{\cB_{1}(n)}{G_{E, \alpha}} \ar[ul] \ar[uu] \\
		(G_{E, \alpha})_{\agr} \ar[uu] \\
		& G_{E, \alpha}(n) \ar[uu] \\
		G_{E, \alpha} \ar[uu] \ar[ur]
	\end{tikzcd}\]
	By Theorem \ref{thm:isolated_points:equivalent_isolated}, if a closed point $x \in X_{0}(n)$ with minimal representative $(E, [\alpha]_{B_{0}(n)})$ is isolated, then we obtain an isolated point on the modular curve corresponding to any subgroup $G$ lying between $G_{E, \alpha}$ and $\clo{B_{0}(n)}{G_{E, \alpha}}$. Similarly, if there is an isolated point $x \in X_{1}(n)$ with minimal representative $(E, [\alpha]_{B_{1}(n)})$, then we obtain an isolated point on the modular curve corresponding to any subgroup $G$ lying between $G_{E, \alpha}$ and $\clo{B_{1}(n)}{G_{E, \alpha}}$.
	
	Using these facts, and the diagram above, we can recontextualize a number of existing results on isolated points on modular curves. For instance, the single-source theorem for isolated points given in \cite[Theorem 5.3]{terao2025isolated} corresponds to the case of $G = G_{E, \alpha}$. Similarly, \cite[Theorem 38]{bourdon2025towards} and \cite[Theorem 5.5]{terao2025isolated} both correspond to the case of $G = G_{E, \alpha}(n)$. In \cite[Section 7]{terao2025isolated}, the set of $j$-invariants of non-cuspidal, non-CM isolated points with rational $j$-invariant on the modular curves $X_{0}(n)$ is determined, under the same conjecture of Zywina. The key step in this classification is \cite[Theorem 7.3]{terao2025isolated}, which corresponds to the case $G = (G_{E, \alpha})_{\agr}$ above.
	
	The above diagram also illustrates why the case of the modular curves $X_{1}(n)$ is more difficult than the case of the modular curves $X_{0}(n)$. Indeed, while the $B_{0}(n)$-closure of $G_{E, \alpha}$ contains the agreeable closure $(G_{E, \alpha})_{\agr}$, the same is not true of the $B_{1}(n)$-closure. As such, while the classification of Zywina of the agreeable closures $(G_{E, \alpha})_{\agr}$ associated to elliptic curves $E/\Q$ can be directly leveraged to understand the isolated points on the modular curves $X_{0}(n)$, the same is not true of the modular curves $X_{1}(n)$.
	
	In this present work however, we have conjecturally classified the set of $\cB_{1}(n)$-closures $\clo{\cB_{1}}{G_{E, \alpha}}$ for elliptic curves $E/\Q$. Therefore, the above diagram suggests that one can apply the same approach to determining the isolated points with rational $j$-invariant on the modular curves $X_{1}(n)$ as was done for the modular curves $X_{0}(n)$ above, with the role of the agreeable closure $(G_{E, \alpha})_{\agr}$ played by the $\cB_{1}$-closure $\clo{\cB_{1}}{G_{E, \alpha}}$.
	
	This is precisely the strategy we employ to classify the $j$-invariants of the non-cuspidal, non-CM isolated points with rational $j$-invariants on the modular curves $X_{1}(n)$. As mentioned at the start of this section, we show that, assuming Conjecture \ref{conj:finite_degrees:sl_intersections}, the conjectural set of such $j$-invariants given in \cite[Conjecture 4]{bourdon2025towards} is complete.
	
	\begin{theorem} \label{thm:isolated_points:x1_isolated_j_invariants}
		Suppose that Conjecture \ref{conj:finite_degrees:sl_intersections} holds. Fix $n \geq 1$, and let $x \in X_{1}(n)$ be a non-cuspidal, non-CM isolated closed point with $j(x) \in \Q$. Then
		\[
			j(x) \in \left\{-\frac{140625}{8}, \frac{351}{4}, -9317, -162677523113838677\right\}.
		\]
		In particular, \cite[Conjecture 4]{bourdon2025towards} holds.
	\end{theorem}
	
	\begin{proof}
		Let $(E, [\alpha]_{B_1(n)})$ be a minimal representative for the closed point $x \in X_{1}(n)$. By Theorem \ref{thm:finite_degrees:x1_fiber_degrees}, the conjugacy class of the $\cB_{1}(n)$-closure of the extended adelic Galois image $G_{E, \alpha}$ is given in Table \ref{tbl:infinite_b_closures:infinite_b1_closures} or \ref{tbl:finite_b_closures:finite_b1_closures}. Moreover, by definition, the groups $\clo{\cB_{1}(n)}{G_{E, \alpha}}$ and $G_{E, \alpha}$ are $B_{1}(n)$-equivalent. Therefore, by Theorem \ref{thm:isolated_points:equivalent_isolated}, the closed point $y \in X_{\clo{\cB_{1}(n)}{G_{E, \alpha}}}$ corresponding to the Galois orbit of the geometric point $[(E, [\alpha]_{\clo{\cB_{1}(n)}{G_{E, \alpha}}})] \in X_{\clo{\cB_{1}(n)}{G_{E, \alpha}}}(\Qbar)$ is also isolated.
		
		Suppose that the conjugacy class of the subgroup $\clo{\cB_{1}(n)}{G_{E, \alpha}}$ is given in Table \ref{tbl:infinite_b_closures:infinite_b1_closures}. By Remark \ref{rmk:infinite_degrees:all_have_genus_0}, the modular curve $X_{\clo{\cB_{1}(n)}{G_{E, \alpha}}}$ thus has genus 0. By \cite[Theorem 2.17]{terao2025isolated}, it follows that the closed point $y \in X_{\clo{\cB_{1}(n)}{G_{E, \alpha}}}$ is $\PP^{1}$-parametrized, contradicting the fact that $y$ is isolated.
		
		Thus, the conjugacy class of the subgroup $\clo{\cB_{1}(n)}{G_{E, \alpha}}$ is given in Table \ref{tbl:finite_b_closures:finite_b1_closures}. Moreover, the $j$-invariant $j(x)$ must be one of the 33 $j$-invariants listed in Table \ref{tbl:finite_b_closures:finite_b1_closures}. The work of Bourdon, Hashimoto, Keller, Klagsbrun, Lowry-Duda, Morrison, Najman and Shukla in \cite{bourdon2025towards} provides an algorithm for checking whether there exists an isolated point on a modular curve $X_{1}(n)$ with given rational $j$-invariant $j \in \Q$ . Running this algorithm on the 33 $j$-invariants listed in Table \ref{tbl:finite_b_closures:finite_b1_closures}, we find that the only $j$-invariants which can give rise to an isolated point on a modular curve $X_{1}(n)$ are
		\[
			 -\frac{882216989}{131072}, -\frac{140625}{8}, \frac{351}{4}, -9317 \text{ and } {-162677523113838677}.
		\]
		The latter four $j$-invariants are those which appear in the statement of the theorem, and are known to correspond to non-CM isolated points on some modular curve $X_{1}(n)$, see \cite{bourdon2025towards}. The first $j$-invariant also appears in the cited work, where it is shown \cite[Section 9.0.2]{bourdon2025towards} that there are no isolated points on any modular curve $X_{1}(n)$ with $j$-invariant $-\frac{882216989}{131072}$, thus concluding the proof.
	\end{proof}
	
	\begin{remark}
		Through the computations carried out in \cite{bourdon2025towards}, the authors find six $j$-invariants which might potentially correspond to isolated points on some modular curve $X_{1}(n)$. Four of these are the ones which appear in Theorem \ref{thm:isolated_points:x1_isolated_j_invariants}, and do indeed correspond to isolated points. The two others are spurious, and while their algorithm cannot rule out the contrary, do not in fact correspond to isolated points on some modular curve $X_{1}(n)$.
		
		The first is the $j$-invariant $-\frac{882216989}{131072}$ mentioned in the proof of Theorem \ref{thm:isolated_points:x1_isolated_j_invariants}, which the algorithm given in \cite{bourdon2025towards} suggests may correspond to an isolated point on the modular curve $X_{1}(17)$. This can be somewhat explained by a combination of two facts. Firstly, it gives rise to a degree 4 point on the modular curve $X_{1}(17)$, which is smaller than the genus of $X_{1}(17)$ and thus may be isolated. Moreover, this $j$-invariant appears in Table \ref{tbl:finite_b_closures:finite_b1_closures}: it is the unique non-CM $j$-invariant with $\cB_{1}(17)$-closure $\clo{\cB_{1}(17)}{\cG_{j}}$ equal to \href{https://beta.lmfdb.org/ModularCurve/Q/17.36.1.a.1}{17.36.1.a.1}. In fact, the same is true of the $B_{1}(17)$-closures of $\cG_{j}$. It is thus impossible to use an argument as in Theorem \ref{thm:isolated_points:equivalent_isolated} to show that it is not isolated.
		
		The second spurious $j$-invariant is $\frac{16778985534208729}{81000}$, which gives rise to a potentially isolated degree 4 point on the modular curve $X_{1}(24)$. However, this $j$-invariant does not appear in Table \ref{tbl:finite_b_closures:finite_b1_closures}. In fact, the $\cB_{1}(24)$-closure of $\cG_{j}$ is the conjugacy class \href{https://beta.lmfdb.org/ModularCurve/Q/24.48.0.bt.1}{24.48.0.bt.1}, which has genus 0. Therefore, Theorem \ref{thm:isolated_points:equivalent_isolated}, the $j$-invariant $\frac{16778985534208729}{81000}$ does not correspond to an isolated point on a modular curve $X_{1}(n)$. However, the mod-24 Galois image $\cG_{j}(24)$ of $j$ is the conjugacy class \href{https://beta.lmfdb.org/ModularCurve/Q/24.96.1.dj.3/}{24.96.1.dj.3}, whose modular curve has genus 1 and rank 1. Therefore, there are infinitely many non-CM $j$-invariants with the same mod-24 Galois image, and each will give rise to a spurious potentially isolated $j$-invariant when applying the algorithm in \cite{bourdon2025towards}. The second smallest such $j$-invariant is
		\[
			\frac{3760942468798987805603108263239159580334655361}{77436872925988468414306640625000000000}.
		\]
	\end{remark}
	
	\clearpage
	
	\appendix
	
	\section{Tables}
	
\begin{longtable}[p]{rll}
		\caption{The forty-four $\cB_{0}$-closed conjugacy classes $\cG$ of open subgroups of $\GL(\Zhat)$ such that the modular curve $X_{\cG}$ has infinitely many rational points. For each conjugacy class $\cG$, the LMFDB label of $\cG$ is given, as well as smallest $n \geq 1$ such that $\cG$ is $\cB_{0}(n)$-closed, equal to the level of $\cG$. The sizes of the orbits of $B_{0}(n) \backslash {\GL(\Zhat)} / \cG$ are also given.} \label{tbl:infinite_b_closures:infinite_b0_closures} \\
		$n$ & LMFDB & $B_{0}(n) \backslash {\GL(\Zhat)} / \cG$ \\ \toprule
		\endfirsthead
		$n$ & LMFDB & $B_{0}(n) \backslash {\GL(\Zhat)} / \cG$ \\ \specialrule{\heavyrulewidth}{\abovetopsep}{0pt}
		\endhead
		1 & \href{https://beta.lmfdb.org/ModularCurve/Q/1.1.0.a.1}{1.1.0.a.1} & 1 \\ \midrule
		2 & \href{https://beta.lmfdb.org/ModularCurve/Q/2.6.0.a.1}{2.6.0.a.1} & 1, 1, 1 \\*
		& \href{https://beta.lmfdb.org/ModularCurve/Q/2.3.0.a.1}{2.3.0.a.1} & 1, 2 \\ \midrule
		3 & \href{https://beta.lmfdb.org/ModularCurve/Q/3.12.0.a.1}{3.12.0.a.1} & 1, 1, 2 \\*
		& \href{https://beta.lmfdb.org/ModularCurve/Q/3.4.0.a.1}{3.4.0.a.1} & 1, 3 \\*
		& \href{https://beta.lmfdb.org/ModularCurve/Q/3.6.0.b.1}{3.6.0.b.1} & 2, 2 \\ \midrule
		4 & \href{https://beta.lmfdb.org/ModularCurve/Q/4.24.0.b.1}{4.24.0.b.1} & 1, 1, 1, 1, 2 \\*
		& \href{https://beta.lmfdb.org/ModularCurve/Q/4.12.0.b.1}{4.12.0.b.1} & 1, 1, 2, 2 \\
\bottomrule
	\end{longtable}
	
	\clearpage
	
	\begin{longtable}{rll}
		\caption{The one hundred eighteen $\cB_{1}$-closed conjugacy classes $\cG$ of open subgroups of $\GL(\Zhat)$ such that the modular curve $X_{\cG}$ has infinitely many rational points. For each conjugacy class $\cG$, the LMFDB label of $\cG$ is given, as well as smallest $n \geq 1$ such that $\cG$ is $\cB_{1}(n)$-closed, equal to the level of $\cG$. The sizes of the orbits of $B_{1}(n) \backslash {\GL(\Zhat)} / \cG$ are also given.} \label{tbl:infinite_b_closures:infinite_b1_closures} \\
		$n$ & LMFDB & $B_{1}(n) \backslash {\GL(\Zhat)} / \cG$ \\ \toprule
		\endfirsthead
		$n$ & LMFDB & $B_{1}(n) \backslash {\GL(\Zhat)} / \cG$ \\ \specialrule{\heavyrulewidth}{\abovetopsep}{0pt}
		\endhead
		1 & \href{https://beta.lmfdb.org/ModularCurve/Q/1.1.0.a.1}{1.1.0.a.1} & 1 \\ \midrule
		2 & \href{https://beta.lmfdb.org/ModularCurve/Q/2.6.0.a.1}{2.6.0.a.1} & 1, 1, 1 \\*
		& \href{https://beta.lmfdb.org/ModularCurve/Q/2.3.0.a.1}{2.3.0.a.1} & 1, 2 \\ \midrule
		3 & \href{https://beta.lmfdb.org/ModularCurve/Q/3.12.0.a.1}{3.12.0.a.1} & 1, 1, 2 \\*
		& \href{https://beta.lmfdb.org/ModularCurve/Q/3.4.0.a.1}{3.4.0.a.1} & 1, 3 \\*
		& \href{https://beta.lmfdb.org/ModularCurve/Q/3.6.0.b.1}{3.6.0.b.1} & 2, 2 \\ \midrule
		4 & \href{https://beta.lmfdb.org/ModularCurve/Q/4.24.0.b.1}{4.24.0.b.1} & 1, 1, 1, 1, 2 \\*
		& \href{https://beta.lmfdb.org/ModularCurve/Q/4.12.0.b.1}{4.12.0.b.1} & 1, 1, 2, 2 \\
\bottomrule
	\end{longtable}
	
	\clearpage
	
	\renewcommand*{\arraystretch}{1.1}
	\begin{longtable}{lrrrl}
		\caption{The one hundred eight known $\GL(\Zhat)$-conjugacy classes of open subgroups of $\SL(\Zhat)$ which occur as the commutator $\pm [G, G]$ for some agreeable subgroup $G$ such that the modular curve $X_{G}$ has infinitely many rational points, as computed by Zywina in \cite{zywina2024explicit}. For each conjugacy class $\cG$, the level $n$ and index $i = [\SL(\Zhat) : \cG]$ are given, as well as the genus $g$ of the Riemann surface $\cG \backslash \HH^{\ast}$. If the latter is at most 24, the conjugacy class has a CP label, which we state as well. Finally, for each conjugacy class, we give generators for a representative of the conjugacy class, as a subgroup of $\SL(\Z[n])$.} \label{tbl:finite_b_closures:infinite_sl_intersections_zywina} \\
		CP & $n$ & $i$ & $g$ & Generators \\ \toprule
		\endfirsthead
		CP & $n$ & $i$ & $g$ & Generators \\ \toprule
		\endhead
		\href{https://math-sites.uncg.edu/sites/pauli/congruence/csg0.html#group2A0}{2A0} & 2 & 2 & 0 & $\begin{bsmallmatrix} 0 & 1 \\ 1 & 1 \end{bsmallmatrix}$ \\
		\href{https://math-sites.uncg.edu/sites/pauli/congruence/csg0.html#group2C0}{2C0} & 2 & 6 & 0 &  \\
		\href{https://math-sites.uncg.edu/sites/pauli/congruence/csg0.html#group4D0}{4D0} & 4 & 8 & 0 & $\begin{bsmallmatrix} 2 & 1 \\ 1 & 3 \end{bsmallmatrix}$ \\
		\href{https://math-sites.uncg.edu/sites/pauli/congruence/csg0.html#group4E0}{4E0} & 4 & 12 & 0 & $\begin{bsmallmatrix} 3 & 2 \\ 2 & 3 \end{bsmallmatrix}$, $\begin{bsmallmatrix} 3 & 0 \\ 0 & 3 \end{bsmallmatrix}$ \\
		\href{https://math-sites.uncg.edu/sites/pauli/congruence/csg0.html#group4G0}{4G0} & 4 & 24 & 0 & $\begin{bsmallmatrix} 3 & 0 \\ 0 & 3 \end{bsmallmatrix}$ \\
		\href{https://math-sites.uncg.edu/sites/pauli/congruence/csg0.html#group6A0}{6A0} & 6 & 6 & 0 & $\begin{bsmallmatrix} 3 & 4 \\ 2 & 3 \end{bsmallmatrix}$, $\begin{bsmallmatrix} 1 & 5 \\ 3 & 4 \end{bsmallmatrix}$ \\
		\href{https://math-sites.uncg.edu/sites/pauli/congruence/csg0.html#group6C0}{6C0} & 6 & 8 & 0 & $\begin{bsmallmatrix} 2 & 3 \\ 3 & 5 \end{bsmallmatrix}$, $\begin{bsmallmatrix} 3 & 2 \\ 4 & 5 \end{bsmallmatrix}$ \\
		\href{https://math-sites.uncg.edu/sites/pauli/congruence/csg1.html#group6B1}{6B1} & 6 & 12 & 1 & $\begin{bsmallmatrix} 0 & 1 \\ 5 & 3 \end{bsmallmatrix}$ \\
\bottomrule
	\end{longtable}
	
	\clearpage
	
	\begin{landscape}
	\renewcommand*{\arraystretch}{1.25}
	\LTcapwidth=0.9\linewidth
	\begin{longtable}{lrrrrl}
		\caption{The 81 known non-CM $j$-invariants $j \in \Q$ such that the modular curve $X_{(\cG_{j})_{\agr}}$ has finitely many rational points. For each $j$ invariant, the $\GL$-level $n$ of $\cG_{j}$, the index $i = [\GL(\Zhat) : \cG_{j}]$, the genus $g$ of the modular curve $X_{\cG}$ and the $\SL$-level of $\cG_{j}$ are given. We also give generators, with small entries, for a representative of the conjugacy class $\cG_{j}$ as a subgroup of $\GL(\Z[n])$, computed in the file \texttt{preprocess\_exceptional\_images.m}. The shaded rows correspond to $j$-invariants $j$ such that the intersection $\cG_{j} \cap \SL(\Zhat)$ is given in Table \ref{tbl:finite_b_closures:infinite_sl_intersections_zywina}.} \label{tbl:finite_b_closures:exceptional_j_invariants} \\
		$j$-invariant & $n$ & $i$ & $g$ & $m$ & Generators \\ \toprule
		\endfirsthead
		$j$-invariant & $n$ & $i$ & $g$ & $m$ & Generators \\ \toprule
		\endhead
		\rowcolor{lightgray} $-13824$ & 12 & 12 & 1 & 6 & $\begin{bsmallmatrix} 9 & 7 \\ 10 & 9 \end{bsmallmatrix}$, $\begin{bsmallmatrix} 10 & 11 \\ 1 & 4 \end{bsmallmatrix}$, $\begin{bsmallmatrix} 0 & 1 \\ 1 & 0 \end{bsmallmatrix}$ \\
		$-138240$ & 12 & 24 & 1 & 12 & $\begin{bsmallmatrix} 1 & 3 \\ 7 & 4 \end{bsmallmatrix}$, $\begin{bsmallmatrix} 7 & 7 \\ 4 & 9 \end{bsmallmatrix}$ \\
		$-82944$ & 12 & 24 & 1 & 12 & $\begin{bsmallmatrix} 6 & 7 \\ 7 & 6 \end{bsmallmatrix}$, $\begin{bsmallmatrix} 3 & 11 \\ 8 & 7 \end{bsmallmatrix}$ \\
		$-\frac{35937}{4}$ & 12 & 64 & 1 & 12 & $\begin{bsmallmatrix} 9 & 5 \\ 8 & 3 \end{bsmallmatrix}$, $\begin{bsmallmatrix} 10 & 1 \\ 5 & 3 \end{bsmallmatrix}$, $\begin{bsmallmatrix} 1 & 0 \\ 9 & 7 \end{bsmallmatrix}$ \\
		$\frac{109503}{64}$ & 12 & 64 & 1 & 12 & $\begin{bsmallmatrix} 6 & 5 \\ 5 & 3 \end{bsmallmatrix}$, $\begin{bsmallmatrix} 9 & 8 \\ 5 & 3 \end{bsmallmatrix}$, $\begin{bsmallmatrix} 11 & 8 \\ 3 & 1 \end{bsmallmatrix}$ \\
		\rowcolor{lightgray} $-44789760$ & 18 & 54 & 2 & 18 & $\begin{bsmallmatrix} 9 & 7 \\ 10 & 7 \end{bsmallmatrix}$, $\begin{bsmallmatrix} 16 & 15 \\ 1 & 1 \end{bsmallmatrix}$, $\begin{bsmallmatrix} 5 & 4 \\ 6 & 5 \end{bsmallmatrix}$ \\
		\rowcolor{lightgray} $-36$ & 20 & 40 & 2 & 20 & $\begin{bsmallmatrix} 1 & 17 \\ 12 & 15 \end{bsmallmatrix}$, $\begin{bsmallmatrix} 11 & 12 \\ 3 & 9 \end{bsmallmatrix}$, $\begin{bsmallmatrix} 10 & 13 \\ 13 & 11 \end{bsmallmatrix}$ \\
		\rowcolor{lightgray} $-\frac{64278657}{1024}$ & 20 & 40 & 2 & 20 & $\begin{bsmallmatrix} 2 & 11 \\ 3 & 2 \end{bsmallmatrix}$, $\begin{bsmallmatrix} 14 & 15 \\ 3 & 17 \end{bsmallmatrix}$, $\begin{bsmallmatrix} 15 & 16 \\ 3 & 13 \end{bsmallmatrix}$ \\
\bottomrule
	\end{longtable}
	\end{landscape}
	
	\clearpage
	
	\begin{longtable}[p]{r>{\raggedright}p{6cm}rl}
		\caption{The one hundred sixty $\cB_{1}$-closed conjugacy classes $\cG$ of open subgroups of $\GL(\Zhat)$ with full determinant, non-zero genus, containing a subgroup of Table \ref{tbl:finite_b_closures:infinite_sl_intersections_zywina}, and maximal in this property. For each conjugacy class $\cG$, the level $n$, genus $g$ and LMFDB label of $\cG$ is given, as well as the method used to determine the set of non-cuspidal non-CM rational points on the modular curve $X_{\cG}$.} \label{tbl:finite_b_closures:maximal_finite_b1_closed} \\
		$n$ & LMFDB label & $g$ & Method \\ \toprule
		\endfirsthead
		$n$ & LMFDB label & $g$ & Method \\ \specialrule{\heavyrulewidth}{\abovetopsep}{0pt}
		\endhead
		6 & \href{https://beta.lmfdb.org/ModularCurve/Q/6.36.1.b.1}{6.36.1.b.1}, \href{https://beta.lmfdb.org/ModularCurve/Q/6.72.1.b.1}{6.72.1.b.1} & 1 & Rank 0 elliptic curve \mbox{(\ref{sec:rational_points:rank_0_ec})} \\ \midrule
		7 & \href{https://beta.lmfdb.org/ModularCurve/Q/7.42.1.b.1}{7.42.1.b.1}, \href{https://beta.lmfdb.org/ModularCurve/Q/7.56.1.a.1}{7.56.1.a.1}, \href{https://beta.lmfdb.org/ModularCurve/Q/7.56.1.b.1}{7.56.1.b.1} & 1 & Prior work \mbox{(\ref{sec:rational_points:prior_work})} \\ \midrule
		8 & \href{https://beta.lmfdb.org/ModularCurve/Q/8.24.1.x.1}{8.24.1.x.1}, \href{https://beta.lmfdb.org/ModularCurve/Q/8.32.1.b.1}{8.32.1.b.1}, \href{https://beta.lmfdb.org/ModularCurve/Q/8.48.1.h.1}{8.48.1.h.1}, \href{https://beta.lmfdb.org/ModularCurve/Q/8.48.1.k.1}{8.48.1.k.1}, \href{https://beta.lmfdb.org/ModularCurve/Q/8.48.1.k.2}{8.48.1.k.2}, \href{https://beta.lmfdb.org/ModularCurve/Q/8.48.1.n.1}{8.48.1.n.1}, \href{https://beta.lmfdb.org/ModularCurve/Q/8.48.1.bj.1}{8.48.1.bj.1}, \href{https://beta.lmfdb.org/ModularCurve/Q/8.48.1.bt.1}{8.48.1.bt.1}, \href{https://beta.lmfdb.org/ModularCurve/Q/8.48.1.bv.1}{8.48.1.bv.1}, \href{https://beta.lmfdb.org/ModularCurve/Q/8.96.1.d.1}{8.96.1.d.1}, \href{https://beta.lmfdb.org/ModularCurve/Q/8.96.1.k.1}{8.96.1.k.1}, \href{https://beta.lmfdb.org/ModularCurve/Q/8.96.1.m.1}{8.96.1.m.1}, \href{https://beta.lmfdb.org/ModularCurve/Q/8.96.1.n.1}{8.96.1.n.1} & 1 & Prior work \mbox{(\ref{sec:rational_points:prior_work})} \\ \midrule
		10 & \href{https://beta.lmfdb.org/ModularCurve/Q/10.36.1.a.1}{10.36.1.a.1}, \href{https://beta.lmfdb.org/ModularCurve/Q/10.72.1.b.1}{10.72.1.b.1}, \href{https://beta.lmfdb.org/ModularCurve/Q/10.72.1.b.2}{10.72.1.b.2} & 1 & Rank 0 elliptic curve \mbox{(\ref{sec:rational_points:rank_0_ec})} \\ \midrule
		11 & \href{https://beta.lmfdb.org/ModularCurve/Q/11.110.4.b.1}{11.110.4.b.1} & 4 & Prior work \mbox{(\ref{sec:rational_points:prior_work})} \\
\bottomrule
	\end{longtable}
	
	\clearpage
	
	\begin{landscape}
	\begin{table}
		\caption{The $j$-invariants of the non-cuspidal, non-CM rational points on the modular curves given in Table \ref{tbl:finite_b_closures:maximal_finite_b1_closed}. For each $j$-invariant $j$, the label of the modular curve of Table \ref{tbl:finite_b_closures:maximal_finite_b1_closed} on which it lies is given, alongside the $\GL$-level $n$ of $\cG_{j}$, the index $i = [\GL(\Zhat) : \cG_{j}]$, the genus of the modular curve $X_{\cG_{j}}$ and the $\SL$-level $m$ of $\cG_{j}$. We also provide generators for a representative of the conjugacy class $\cG_{j}$ as a subgroup of $\GL(\Z[n])$. The shaded rows indicate $j$-invariants $j$ such that the intersection $\cG_{j} \cap \SL(\Zhat)$ is not given in Table \ref{tbl:finite_b_closures:infinite_sl_intersections_zywina}.} \label{tbl:finite_b_closures:exceptional_j_infinite_sl}
		\renewcommand*{\arraystretch}{1.25}
		\begin{tabular}{llrrrrl}
			LMFDB & $j$-invariant & $n$ & $i$ & $g$ & $m$ & Generators \\ \toprule
			\href{https://beta.lmfdb.org/ModularCurve/Q/7.56.1.b.1}{7.56.1.b.1} & $\frac{2268945}{128}$ & 56 & 112 & 5 & 14 & $\begin{bsmallmatrix} 1 & 28 \\ 41 & 55 \end{bsmallmatrix}$, $\begin{bsmallmatrix} 49 & 22 \\ 41 & 27 \end{bsmallmatrix}$, $\begin{bsmallmatrix} 20 & 21 \\ 1 & 1 \end{bsmallmatrix}$, $\begin{bsmallmatrix} 0 & 19 \\ 23 & 16 \end{bsmallmatrix}$ \\
			\rowcolor{lightgray} \href{https://beta.lmfdb.org/ModularCurve/Q/12.32.1.b.1}{12.32.1.b.1} & $-\frac{35937}{4}$ & 12 & 64 & 1 & 12 & $\begin{bsmallmatrix} 9 & 5 \\ 8 & 3 \end{bsmallmatrix}$, $\begin{bsmallmatrix} 10 & 1 \\ 5 & 3 \end{bsmallmatrix}$, $\begin{bsmallmatrix} 1 & 0 \\ 9 & 7 \end{bsmallmatrix}$ \\
			\rowcolor{lightgray} \href{https://beta.lmfdb.org/ModularCurve/Q/12.32.1.b.1}{12.32.1.b.1} & $\frac{109503}{64}$ & 12 & 64 & 1 & 12 & $\begin{bsmallmatrix} 6 & 5 \\ 5 & 3 \end{bsmallmatrix}$, $\begin{bsmallmatrix} 9 & 8 \\ 5 & 3 \end{bsmallmatrix}$, $\begin{bsmallmatrix} 11 & 8 \\ 3 & 1 \end{bsmallmatrix}$ \\
			\rowcolor{lightgray} \href{https://beta.lmfdb.org/ModularCurve/Q/15.36.1.b.1}{15.36.1.b.1} & $\frac{1331}{8}$ & 1560 & 288 & 17 & 30 & $\begin{bsmallmatrix} 439 & 117 \\ 15 & 4 \end{bsmallmatrix}$, $\begin{bsmallmatrix} 71 & 27 \\ \smash[b]{\shortminus}405 & \smash[b]{\shortminus}154 \end{bsmallmatrix}$, $\begin{bsmallmatrix} 18 & 31 \\ 5 & 9 \end{bsmallmatrix}$, $\begin{bsmallmatrix} \smash[b]{\shortminus}57 & \smash[b]{\shortminus}62 \\ 25 & 27 \end{bsmallmatrix}$, $\begin{bsmallmatrix} 27 & 158 \\ 25 & 147 \end{bsmallmatrix}$, $\begin{bsmallmatrix} \smash[b]{\shortminus}131 & \smash[b]{\shortminus}150 \\ 15 & 17 \end{bsmallmatrix}$, $\begin{bsmallmatrix} 176 & 15 \\ 45 & 4 \end{bsmallmatrix}$ \\
			\rowcolor{lightgray} \href{https://beta.lmfdb.org/ModularCurve/Q/15.36.1.b.1}{15.36.1.b.1} & $-\frac{1680914269}{32768}$ & 1560 & 288 & 17 & 30 & $\begin{bsmallmatrix} 11 & 261 \\ 15 & 356 \end{bsmallmatrix}$, $\begin{bsmallmatrix} 49 & \smash[b]{\shortminus}567 \\ 15 & 304 \end{bsmallmatrix}$, $\begin{bsmallmatrix} \smash[b]{\shortminus}6 & \smash[b]{\shortminus}1 \\ 25 & 3 \end{bsmallmatrix}$, $\begin{bsmallmatrix} 117 & 17 \\ 20 & 3 \end{bsmallmatrix}$, $\begin{bsmallmatrix} 27 & 7 \\ \smash[b]{\shortminus}130 & \smash[b]{\shortminus}33 \end{bsmallmatrix}$, $\begin{bsmallmatrix} 82 & 15 \\ 75 & 14 \end{bsmallmatrix}$, $\begin{bsmallmatrix} 29 & 0 \\ 15 & 1 \end{bsmallmatrix}$ \\
			\href{https://beta.lmfdb.org/ModularCurve/Q/16.96.3.fa.1}{16.96.3.fa.1} & $\frac{4097^{3}}{2^{4}}$ & 656 & 192 & 9 & 16 & $\begin{bsmallmatrix} 44 & 133 \\ 201 & 442 \end{bsmallmatrix}$, $\begin{bsmallmatrix} 40 & 91 \\ 187 & 550 \end{bsmallmatrix}$, $\begin{bsmallmatrix} 347 & 64 \\ 180 & 655 \end{bsmallmatrix}$, $\begin{bsmallmatrix} 135 & 546 \\ 552 & 37 \end{bsmallmatrix}$ \\
			\href{https://beta.lmfdb.org/ModularCurve/Q/16.96.3.fa.2}{16.96.3.fa.2} & $\frac{16974593}{256}$ & 656 & 192 & 9 & 16 & $\begin{bsmallmatrix} 235 & 164 \\ 552 & 55 \end{bsmallmatrix}$, $\begin{bsmallmatrix} 393 & 392 \\ 202 & 167 \end{bsmallmatrix}$, $\begin{bsmallmatrix} 395 & 612 \\ 400 & 63 \end{bsmallmatrix}$, $\begin{bsmallmatrix} 578 & 395 \\ 577 & 382 \end{bsmallmatrix}$ \\
			\href{https://beta.lmfdb.org/ModularCurve/Q/18.72.2.c.1}{18.72.2.c.1} & $406749952$ & 252 & 432 & 28 & 36 & $\begin{bsmallmatrix} 163 & 90 \\ 216 & 155 \end{bsmallmatrix}$, $\begin{bsmallmatrix} 221 & 96 \\ 122 & 235 \end{bsmallmatrix}$, $\begin{bsmallmatrix} 121 & 99 \\ 183 & 122 \end{bsmallmatrix}$, $\begin{bsmallmatrix} 20 & 135 \\ 69 & 173 \end{bsmallmatrix}$ \\
			\href{https://beta.lmfdb.org/ModularCurve/Q/18.72.2.c.2}{18.72.2.c.2} & $1792$ & 252 & 432 & 28 & 36 & $\begin{bsmallmatrix} 164 & 129 \\ 79 & 199 \end{bsmallmatrix}$, $\begin{bsmallmatrix} 116 & 57 \\ 1 & 1 \end{bsmallmatrix}$, $\begin{bsmallmatrix} 203 & 141 \\ 173 & 196 \end{bsmallmatrix}$, $\begin{bsmallmatrix} 133 & 18 \\ 156 & 107 \end{bsmallmatrix}$ \\
			\href{https://beta.lmfdb.org/ModularCurve/Q/24.72.2.hl.1}{24.72.2.hl.1} & $4913$ & 3120 & 576 & 41 & 48 & $\begin{bsmallmatrix} 117 & 188 \\ 28 & 45 \end{bsmallmatrix}$, $\begin{bsmallmatrix} 13 & 12 \\ 300 & 277 \end{bsmallmatrix}$, $\begin{bsmallmatrix} 9 & 10 \\ 170 & 189 \end{bsmallmatrix}$, $\begin{bsmallmatrix} 27 & 50 \\ 34 & 63 \end{bsmallmatrix}$, $\begin{bsmallmatrix} 19 & 32 \\ \smash[b]{\shortminus}154 & \smash[b]{\shortminus}259 \end{bsmallmatrix}$, $\begin{bsmallmatrix} \smash[b]{\shortminus}2 & 43 \\ \smash[b]{\shortminus}1 & 16 \end{bsmallmatrix}$, $\begin{bsmallmatrix} 160 & 27 \\ 189 & 32 \end{bsmallmatrix}$, $\begin{bsmallmatrix} 262 & 35 \\ 239 & 32 \end{bsmallmatrix}$, $\begin{bsmallmatrix} 10 & 3 \\ 99 & 32 \end{bsmallmatrix}$ \\
			\href{https://beta.lmfdb.org/ModularCurve/Q/24.72.2.hl.2}{24.72.2.hl.2} & $16974593$ & 3120 & 576 & 41 & 48 & $\begin{bsmallmatrix} 41 & 42 \\ \smash[b]{\shortminus}534 & \smash[b]{\shortminus}547 \end{bsmallmatrix}$,\,$\begin{bsmallmatrix} 25 & 18 \\ 18 & 13 \end{bsmallmatrix}$,\,$\begin{bsmallmatrix} 9 & \smash[b]{\shortminus}152 \\ 8 & \smash[b]{\shortminus}135 \end{bsmallmatrix}$,\,$\begin{bsmallmatrix} 11 & 42 \\ 138 & 527 \end{bsmallmatrix}$,\,$\begin{bsmallmatrix} \smash[b]{\shortminus}469 & \smash[b]{\shortminus}434 \\ 40 & 37 \end{bsmallmatrix}$,\,$\begin{bsmallmatrix} 34 & \smash[b]{\shortminus}101 \\ 23 & \smash[b]{\shortminus}68 \end{bsmallmatrix}$,\,$\begin{bsmallmatrix} 32 & 165 \\ 3 & 16 \end{bsmallmatrix}$,\,$\begin{bsmallmatrix} \smash[b]{\shortminus}14 & \smash[b]{\shortminus}13 \\ 23 & 20 \end{bsmallmatrix}$,\,$\begin{bsmallmatrix} \smash[b]{\shortminus}752 & 45 \\ \smash[b]{\shortminus}435 & 26 \end{bsmallmatrix}$ \\ \bottomrule
		\end{tabular}
	\end{table}
	\end{landscape}
	
	\clearpage
	
	\begin{table}
		\caption{The known conjugacy classes $\cG$ of open subgroups of $\GL(\Zhat)$ which occur as the $\cB_{0}(n)$-closure, for some $n \geq 1$, of $\cG_{j}$ for finitely many non-CM $j \in \Q$. For each conjugacy class $\cG$, the level $n$ and LMFDB label of $\cG$ is given, as well as the set of $j$-invariants $j \in \Q \setminus \Qcm$ such that $\cG = \clo{\cB_{0}(n)}{\cG_{j}}$. The sizes of the orbits of $B_{0}(n) \backslash {\GL(\Zhat)} / \cG$ are also given.} \label{tbl:finite_b_closures:finite_b0_closures}
		\renewcommand*{\arraystretch}{1.25}
		\begin{tabular}{rlll}
			$n$ & LMFDB & $j$-invariants & $B_{0}(n) \backslash {\GL(\Zhat)} / \cG$ \\ \toprule
			7 & \href{https://beta.lmfdb.org/ModularCurve/Q/7.56.1.b.1}{7.56.1.b.1} & $\frac{2268945}{128}$ & 2, 3, 3 \\
			11 & \href{https://beta.lmfdb.org/ModularCurve/Q/11.12.1.a.1}{11.12.1.a.1} & $-121$, $-24729001$ & 1, 11 \\
			12 & \href{https://beta.lmfdb.org/ModularCurve/Q/12.32.1.b.1}{12.32.1.b.1} & $-\frac{35937}{4}$, $\frac{109503}{64}$ & 3, 3, 9, 9 \\
			& \href{https://beta.lmfdb.org/ModularCurve/Q/12.48.1.q.1}{12.48.1.q.1} & $\frac{3375}{64}$ & 6, 6, 12 \\
			13 & \href{https://beta.lmfdb.org/ModularCurve/Q/13.91.3.a.1}{13.91.3.a.1} & $-\frac{143 \cdot 1040^{3}}{3^{13}}$, $\frac{130 \cdot 442^{3}}{3^{13}}$, $\frac{12077 \cdot 1957713745728^{3}}{305^{13}}$ & 6, 8 \\
			15 & \href{https://beta.lmfdb.org/ModularCurve/Q/15.24.1.a.1}{15.24.1.a.1} & $-\frac{25}{2}$, $-\frac{121945}{32}$, $\frac{46969655}{32768}$, $-\frac{349938025}{8}$ & 1, 3, 5, 15 \\
			& \href{https://beta.lmfdb.org/ModularCurve/Q/15.36.1.b.1}{15.36.1.b.1} & $\frac{1331}{8}$, $-\frac{1680914269}{32768}$ & 2, 2, 10, 10 \\
			17 & \href{https://beta.lmfdb.org/ModularCurve/Q/17.18.1.a.1}{17.18.1.a.1} & $-\frac{297756989}{2}$, $-\frac{882216989}{131072}$ & 1, 17 \\
			21 & \href{https://beta.lmfdb.org/ModularCurve/Q/21.32.1.a.1}{21.32.1.a.1} & $\frac{3375}{2}$, $-\frac{140625}{8}$, $-\frac{5745^{3}}{2^{7}}$, $-\frac{9 \cdot 505^{3}}{2^{21}}$ & 1, 3, 7, 21 \\
			28 & \href{https://beta.lmfdb.org/ModularCurve/Q/28.64.3.b.1}{28.64.3.b.1} & $\frac{351}{4}$, $-\frac{13 \cdot 1437^{3}}{2^{14}}$ & 3, 3, 21, 21 \\
			37 & \href{https://beta.lmfdb.org/ModularCurve/Q/37.38.2.a.1}{37.38.2.a.1} & $-9317$, $-7 \cdot 285371^{3}$ & 1, 37 \\ \bottomrule
		\end{tabular}
	\end{table}
	
	\clearpage
	
	\begin{table}
		\caption{The known conjugacy classes $\cG$ of open subgroups of $\GL(\Zhat)$ which occur as the $\cB_{1}(n)$-closure, for some $n \geq 1$, of $\cG_{j}$ for finitely many non-CM $j \in \Q$. For each conjugacy class $\cG$, the level $n$ and LMFDB label of $\cG$ is given, as well as the set of $j$-invariants $j \in \Q \setminus \Qcm$ such that $\cG = \clo{\cB_{1}(n)}{\cG_{j}}$. The sizes of the orbits of $B_{1}(n) \backslash {\GL(\Zhat)} / \cG$ are also given.} \label{tbl:finite_b_closures:finite_b1_closures}
		\renewcommand*{\arraystretch}{1.25}
		\begin{tabular}{rlll}
			$n$ & LMFDB & $j$-invariants & $B_{1}(n) \backslash {\GL(\Zhat)} / \cG$ \\ \toprule
			7 & \href{https://beta.lmfdb.org/ModularCurve/Q/7.56.1.b.1}{7.56.1.b.1} & $\frac{2268945}{128}$ & 6, 9, 9 \\
			11 & \href{https://beta.lmfdb.org/ModularCurve/Q/11.12.1.a.1}{11.12.1.a.1} & $-121$, $-24729001$ & 5, 55 \\
			12 & \href{https://beta.lmfdb.org/ModularCurve/Q/12.48.1.q.1}{12.48.1.q.1} & $\frac{3375}{64}$ & 12, 12, 24 \\
			& \href{https://beta.lmfdb.org/ModularCurve/Q/12.64.1.b.1}{12.64.1.b.1} & $\frac{109503}{64}$ & 3, 3, 6, 18, 18 \\
			& \href{https://beta.lmfdb.org/ModularCurve/Q/12.64.1.b.2}{12.64.1.b.2} & $-\frac{35937}{4}$ & 6, 6, 9, 9, 18 \\
			13 & \href{https://beta.lmfdb.org/ModularCurve/Q/13.91.3.a.1}{13.91.3.a.1} & $-\frac{143 \cdot 1040^{3}}{3^{13}}$, $\frac{130 \cdot 442^{3}}{3^{13}}$, $\frac{12077 \cdot 1957713745728^{3}}{305^{13}}$ & 36, 48 \\
			15 & \href{https://beta.lmfdb.org/ModularCurve/Q/15.48.1.a.1}{15.48.1.a.1} & $-\frac{121945}{32}$, $\frac{46969655}{32768}$ & 2, 2, 6, 6, 20, 60 \\
			& \href{https://beta.lmfdb.org/ModularCurve/Q/15.48.1.a.2}{15.48.1.a.2} & $-\frac{25}{2}$, $-\frac{349938025}{8}$ & 4, 10, 10, 12, 30, 30 \\
			& \href{https://beta.lmfdb.org/ModularCurve/Q/15.72.1.a.1}{15.72.1.a.1} & $-\frac{1680914269}{32768}$ & 8, 8, 20, 20, 40 \\
			& \href{https://beta.lmfdb.org/ModularCurve/Q/15.72.1.a.2}{15.72.1.a.2} & $\frac{1331}{8}$ & 4, 4, 8, 40, 40 \\
			16 & \href{https://beta.lmfdb.org/ModularCurve/Q/16.96.3.fa.1}{16.96.3.fa.1} & $\frac{4097^{3}}{2^{4}}$ & 16, 16, 16, 16, 32 \\
			& \href{https://beta.lmfdb.org/ModularCurve/Q/16.96.3.fa.2}{16.96.3.fa.2} & $\frac{16974593}{256}$ & 8, 8, 8, 8, 64 \\
			17 & \href{https://beta.lmfdb.org/ModularCurve/Q/17.36.1.a.1}{17.36.1.a.1} & $-\frac{297756989}{2}$ & 8, 68, 68 \\
			& \href{https://beta.lmfdb.org/ModularCurve/Q/17.36.1.a.2}{17.36.1.a.2} & $-\frac{882216989}{131072}$ & 4, 4, 136 \\
			18 & \href{https://beta.lmfdb.org/ModularCurve/Q/18.72.2.c.1}{18.72.2.c.1} & $406749952$ & 27, 27, 27, 27 \\
			& \href{https://beta.lmfdb.org/ModularCurve/Q/18.72.2.c.2}{18.72.2.c.2} & $1792$ & 9, 9, 9, 81 \\
			20 & \href{https://beta.lmfdb.org/ModularCurve/Q/20.48.1.a.1}{20.48.1.a.1} & $\frac{1026895}{1024}$ & 12, 12, 120 \\
			& \href{https://beta.lmfdb.org/ModularCurve/Q/20.48.1.a.2}{20.48.1.a.2} & $-\frac{1723025}{4}$ & 24, 60, 60 \\
			21 & \href{https://beta.lmfdb.org/ModularCurve/Q/21.64.1.a.1}{21.64.1.a.1} & $-\frac{9 \cdot 505^{3}}{2^{21}}$ & 6, 18, 21, 21, 126 \\
			& \href{https://beta.lmfdb.org/ModularCurve/Q/21.64.1.a.2}{21.64.1.a.2} & $\frac{3375}{2}$ & 6, 9, 9, 42, 126 \\
			& \href{https://beta.lmfdb.org/ModularCurve/Q/21.64.1.a.3}{21.64.1.a.3} & $-\frac{140625}{8}$ & 3, 3, 18, 42, 126 \\
			& \href{https://beta.lmfdb.org/ModularCurve/Q/21.64.1.a.4}{21.64.1.a.4} & $-\frac{5745^{3}}{2^{7}}$ & 6, 18, 42, 63, 63 \\
			24 & \href{https://beta.lmfdb.org/ModularCurve/Q/24.72.2.hl.1}{24.72.2.hl.1} & $4913$ & 32, 32, 128 \\
			& \href{https://beta.lmfdb.org/ModularCurve/Q/24.72.2.hl.2}{24.72.2.hl.2} & $16974593$ & 64, 64, 64 \\
			28 & \href{https://beta.lmfdb.org/ModularCurve/Q/28.128.5.b.1}{28.128.5.b.1} & $-\frac{13 \cdot 1437^{3}}{2^{14}}$ & 18, 18, 63, 63, 126 \\
			& \href{https://beta.lmfdb.org/ModularCurve/Q/28.128.5.b.2}{28.128.5.b.2} & $\frac{351}{4}$ & 9, 9, 18, 126, 126 \\
			37 & \href{https://beta.lmfdb.org/ModularCurve/Q/37.114.4.b.1}{37.114.4.b.1} & $-9317$ & 6, 6, 6, 666 \\
			& \href{https://beta.lmfdb.org/ModularCurve/Q/37.114.4.b.2}{37.114.4.b.2} & $-7 \cdot 285371^{3}$ & 18, 222, 222, 222 \\ \bottomrule
		\end{tabular}
	\end{table}
	
	\clearpage
	
	\printbibliography
\end{document}